\documentclass{amsart}[11pt]
\usepackage{amssymb,amscd}
\usepackage{xcolor}
\usepackage{tabmac}
\usepackage{verbatim}
\usepackage{hyperref}
\usepackage{array}

\newcommand{\af}{{\mathrm{af}}}
\newcommand{\al}{\alpha}
\newcommand{\A}{{\mathbb{A}}}
\newcommand{\B}{{\mathbb{B}}}
\newcommand{\BB}{{\mathcal{B}}}
\newcommand{\C}{{\mathbb{C}}}

\newcommand{\CC}{\mathcal{C}}

\newcommand{\del}{\partial}
\newcommand{\Dec}[2]{I\!\!\downarrow_{#2}^{#1}}
\newcommand{\dom}{\trianglerighteq}

\newcommand{\F}{\mathbb F}

\newcommand{\first}{\mathrm{first}}
\newcommand{\fgap}{\mathrm{firstgap}}

\newcommand{\fp}{\mathrm{fp}}

\newcommand{\G}{\mathcal G}
\newcommand{\geh}{\mathfrak{g}}
\newcommand{\Gr}{\mathrm{Gr}}
\newcommand{\Hc}{\Gamma}
\newcommand{\Hcinf}{\Gamma^*}
\newcommand{\Hm}{\Gamma}
\newcommand{\Hminf}{\Gamma_*}
\newcommand{\ip}[1]{\langle #1 \rangle}
\newcommand{\ipp}[1]{[ #1 ]}
\newcommand{\Inc}[2]{I\!\!\uparrow_{#1}^{#2}}
\newcommand{\la}{\lambda}
\newcommand{\last}{\mathrm{last}}
\newcommand{\lgap}{\mathrm{lastgap}}
\newcommand{\La}{\Lambda}

\newcommand{\NilCox}{\A_0}
\newcommand{\NilH}{\A}
\newcommand{\nor}{\mathrm{nor}}
\renewcommand{\O}{\mathbb O}
\newcommand{\om}{\omega}
\newcommand{\OP}{\mathcal{OP}}
\newcommand{\PC}{\mathcal{P}_C}
\newcommand{\Pf}{P^{(n)}}
\newcommand{\PA}{\mathbb{P}}
\newcommand{\pzt}{\phi_0^{(2)}}
\newcommand{\Q}{\mathbb{Q}}
\newcommand{\Qf}{Q^{(n)}}

\newcommand{\rdom}{\trianglelefteq}

\newcommand{\re}{\mathrm{re}}

\newcommand{\Red}{\mathcal{R}}
\newcommand{\RN}{{\overleftarrow{\mathrm{N}}}}
\newcommand{\spe}{\mathrm{sp}}
\newcommand{\gspn}{\mathfrak{sp}_{2n}(\C)}
\newcommand{\SP}{\mathcal{SP}}
\newcommand{\Sc}{\mathcal{S}}
\newcommand{\st}{\mathrm{st}}
\newcommand{\Supp}{\mathrm{Supp}}
\newcommand{\sw}{\subset}
\newcommand{\tC}{{\tilde{C}}}
\newcommand{\tu}{{\tilde{u}}}
\newcommand{\tv}{{\tilde{v}}}
\newcommand{\TT}{T}
\newcommand{\uh}{\widehat{u}}

\newcommand{\vb}{\mathbf{v}}
\newcommand{\vn}{\varnothing}

\newcommand{\WS}{\mathcal{Z}}
\newcommand{\Z}{{\mathbb{Z}}}

\newtheorem{lem}{Lemma}
\newtheorem{thm}[lem]{Theorem}
\newtheorem{prop}[lem]{Proposition}

\theoremstyle{definition}
\newtheorem{remark}[lem]{Remark}

\newtheorem{example}[lem]{Example}

\numberwithin{equation}{section}

\numberwithin{lem}{section}

\begin{document}
\author{Thomas Lam}
\address{Department of Mathematics, Harvard University, Cambridge MA
02138 USA} \email{tfylam@math.harvard.edu}

\author{Anne Schilling}
\address{Department of Mathematics, University of California, One
Shields Ave. Davis, CA 95616-8633 USA} \email{anne@math.ucdavis.edu}

\author{Mark Shimozono}
\address{Department of Mathematics, Virginia Polytechnic Institute
and State University, Blacksburg, VA 24061-0123 USA}
\email{mshimo@vt.edu}
\thanks{\textit{Date}: September 2007}

\title[Schubert Polynomials for affine Grassmannian of symplectic group]
{Schubert polynomials for the affine Grassmannian of the symplectic group}

\begin{abstract}
We study the Schubert calculus of the affine Grassmannian Gr of the
symplectic group. The integral homology and cohomology rings of Gr
are identified with dual Hopf algebras of symmetric functions,
defined in terms of Schur's $P$ and $Q$ functions.  An explicit
combinatorial description is obtained for the Schubert basis of the
cohomology of Gr, and this is extended to a definition of the affine
type $C$ Stanley symmetric functions.  A homology Pieri rule is also
given for the product of a special Schubert class with an arbitrary
one.
\end{abstract}

\maketitle

\section{Introduction}
Let $G$ be a simply-connected simple complex algebraic group and let
$\Gr_G$ denote the affine Grassmannian of $G$.  Following Peterson
\cite{P} and Lam \cite{Lam2} we study the homology and cohomology
Schubert calculus of $\Gr_{Sp_{2n}(\C)}$.

The structure of $H_*(\Gr_G)$\footnote{Our (co)homologies are with
$\Z$-coefficients unless otherwise specified.} and $H^*(\Gr_G)$ is
particularly rich because of the interaction of two phenomena. On
the one hand, $\Gr_G$ inherits free $\Z$-module {\it Schubert bases}
$\{\xi_x \in H_*(\Gr_G)\}$ and $\{\xi^x \in H^*(\Gr_G)\}$ from its
presentation $\Gr_G = \G/{\mathcal P}$ where $\G$ is the affine
Kac-Moody group associated to $G$ and ${\mathcal P} \subset \G$ is a
maximal parabolic subgroup.  On the other hand, it is a classical
result of Quillen \cite{Q} (see also \cite{GR} and \cite{Pr}) that
$\Gr_G$ is homotopy equivalent to the based loops $\Omega K$ into
the maximal compact subgroup $K \subset G$. The group structure on
$\Omega K$ endows $H_*(\Gr_G)$ and $H^*(\Gr_G)$ with the structure
of dual Hopf algebras.

The dual Hopf algebras $H_*(\Gr_G)$ and $H^*(\Gr_G)$ were first
studied intensively by Bott \cite{B}.  Bott gave an algorithm to
compute these Hopf algebras in terms of the Cartan data of $G$,
essentially by transgressing elements of $H^*(K)$ to obtain the
primitive elements in $H^*(\Gr_G)$.  With $\Q$-coefficients, $H^*(K,
\Q)$ is an exterior algebra with odd-dimensional generators so
$H^*(\Gr_G, \Q)$ is a polynomial algebra on even-dimensional
generators.  The situation is even more favorable when $G =
Sp_{2n}(\C)$ since $Sp_{2n}(\C)$ is torsion-free and
$H_*(\Gr_{Sp_{2n}(\C)})$ is a polynomial algebra over $\Z$.  Bott
comments that his description does not give polynomial generators
for $H_*(\Gr_{Sp_{2n}(\C)})$. We resolve this by producing $n$
special Schubert classes which are polynomial generators over $\Z$.

Our main result identifies $H_*(\Gr_{Sp_{2n}(\C)})$ and
$H^*(\Gr_{Sp_{2n}(\C)})$ with certain dual Hopf algebras $\Hc_{(n)}$
and $\Hc^{(n)}$ of symmetric functions, defined in terms of Schur's
$P$- and $Q$-functions \cite{Mac}.  We explicitly describe symmetric
functions $\Qf_w \in \Hc^{(n)}$ which represents the cohomology
Schubert basis. These symmetric functions are constructed using the
combinatorics of a remarkable subset $\WS \subset \tC_n$ of the
affine Weyl group $\tC_n$ of $Sp_{2n}(\C)$. In fact the cohomology
representatives $\Qf_w$ extend to a larger family of symmetric
functions: the type $C$ affine Stanley symmetric functions.

\subsection{Peterson's work on affine Schubert calculus}
Our results rely heavily on the work of Peterson \cite{P} who
defined a Hopf embedding $j: H_T(\Gr_G)\to \A$ of the
$T$-equivariant cohomology of $\Gr_G$ as a commutative subalgebra of
the nilHecke ring $\A = \A_G$ of Kostant and Kumar \cite{KK}.  Here
$T \subset K$ is a maximal torus.  Peterson characterizes the image
$j(\xi_x)$ of the Schubert basis of $H_T(\Gr_G)$ in terms of certain
identities inside $\A$.  In the non-equivariant case, Lam
\cite{Lam2} showed that Peterson's embedding specializes to a Hopf
isomorphism $j_0: H_*(\Gr_G) \to \B$ with an algebra which he called
the affine Fomin-Stanley subalgebra.  We give an explicit
combinatorial formula for generators of $\B$ in the case $G =
Sp_{2n}(\C)$.

\subsection{Earlier work for $G = SL_n(\C)$}
For $G=SL_n(\C)$, Lam \cite{Lam2} identified the Schubert basis of
$H_*(\Gr_{SL_n(\C)})$ with symmetric functions, called $k$-Schur
functions, of Lapointe, Lascoux and Morse \cite{LLM}; these arose in
the study of Macdonald polynomials.  The Schubert basis of
$H^*(\Gr_{SL_n(\C)})$ are the dual $k$-Schur functions \cite{LM}
which are generalized by the affine Stanley symmetric functions
\cite{Lam1}. In \cite{LLMS} Pieri rules were given for the
multiplication of Bott's generators on the Schubert bases of Bott's
realization of $H_*(\Gr_{SL_n(\C)})$ and $H^*(\Gr_{SL_n(\C)})$.
Furthermore, a combinatorial interpretation of the pairing between
$H_*(\Gr_{SL_n(\C)})$ and $H^*(\Gr_{SL_n(\C)})$ is given.

\subsection{Two Hopf algebras of symmetric functions} \label{SS:ring}
Let $\Lambda$ denote the ring of symmetric functions over $\Z$.  Let
$P_i$ and $Q_i$ denote the Schur $P$- and $Q$-functions with a
single part \cite[III.8]{Mac}. Define the Hopf subalgebras of
$\Lambda$ given by $\Hminf = \Z[P_1,P_2,\ldots]$ and $\Hcinf =
\Z[Q_1,Q_2,\ldots]$. There is a natural pairing (see \eqref{E:pair})
$\ipp{\cdot,\cdot}:\Hminf \times \Hcinf \to \Z$ making $\Hminf$ and
$\Hcinf$ into dual Hopf algebras.  For $n \geq 1$ the subspace
$\Hc_{(n)} = \Z[P_1,P_2,\ldots,P_{2n}] \subset \Hminf$ is a Hopf
subalgebra and we let $\Hcinf \twoheadrightarrow \Hc^{(n)} $ be the
dual quotient Hopf algebra.

\subsection{Special classes}\label{SS:special}
The affine Weyl group of $Sp_{2n}(\C)$, denoted $\tC_n$, has simple
generators $s_0,s_1,\ldots,s_n$ with the relations
\eqref{E:AffWeylRelations}. Let $\tC_n^0$ denote the minimal length
coset representatives of $\tC_n/C_n$, also called the {\it
Grassmannian} elements of $\tC_n$. Define $\rho_i\in \tC_n^0$ by
\begin{equation} \label{E:specialclass}
  \rho_i =
  \begin{cases}
    s_{i-1} s_{i-2}\dotsm s_1 s_0 &  \text{for $1\le i\le n$} \\
    s_{2n+1-i} s_{2n+2-i} \dotsm s_{n-1} s_n s_{n-1} \dotsm s_1 s_0 &\text{for $n+1\le i\le 2n$.}
  \end{cases}
\end{equation}
The homology Schubert classes $\xi_{\rho_i} \in
H_*(\Gr_{Sp_{2n}(\C)})$ for $1\le i\le 2n$, are called
\textit{special classes}.

\subsection{Zee-s}\label{SS:Z}
Let $\WS$ be the Bruhat order ideal in $\tC_n$ generated by the
conjugates of the element $\rho_{2n}$, 
that is, the set of $w\in \tC_n$ which have a reduced word that is a
subword of a rotation of the unique reduced word of $\rho_{2n}$.  An
element of $\WS$ is called a $Z$.  Let $\WS_r = \{w \in \WS \mid
\ell(w) = r\}$ denote the set of $Z$-s of length $\ell(w)$ equal to
$r$.

\begin{example} Let $n=2$. Then $\WS$ consists of the elements of
$\tC_2$ which have a reduced word that is a subword of one of the
words $1210$, $2101$, $1012$, $0121$.
\end{example}

Given a word $u$ with letters in $I_\af$, its support
$\Supp(u)\subset I_\af$ is by definition the set of letters
appearing in $u$. For $w\in\tC_n$ define $\Supp(w)=\Supp(u)$ for any
reduced word $u$; this is independent of the choice of $u$. A
component of a subset of $I_\af$ is by definition a maximal nonempty
subinterval. Let $c(w)$ denote the number of components of
$\Supp(w)$.

\subsection{Affine type $C$ Stanley symmetric
functions}\label{SS:affStan} For $w \in \tC_n$ we define the
generating function
\begin{equation}\label{E:affStan}
  \Qf_w[Y] = \sum_{(v^1,v^2,\dotsm)} \prod_i 2^{c(v^i)}
  y_i^{\ell(v^i)},
\end{equation}
where the sum runs over the factorizations $v^1v^2\dotsm=w$ of $w$
such that $v^i\in\WS$ and $\ell(v^1) + \ell(v^2) + \cdots =
\ell(w)$.

\begin{thm}\label{T:affStan}
The series $\Qf_w$ is symmetric and defines an element of
$\Hc^{(n)}$. The subset $\{\Qf_v \mid v \in \tC_n^0\}$ forms a basis
of $\Hc^{(n)}$ such that all product and coproduct structure
constants are positive and every $\Qf_w$ for $w \in \tC_n$ is
positive in this basis.
\end{thm}

The symmetric functions $\Qf_w$ are type $C$ analogues of the affine
Stanley symmetric functions studied in \cite{Lam1}.
Some examples for the type $C$ affine Stanley symmetric functions
are given in Appendix~\ref{S:Qfdata}.
They have the following geometric interpretation.

Let $LSp(n)$ and $\Omega Sp(n)$ denote respectively the space of all
loops and based loops, into the maximal compact subgroup $Sp(n)
\subset Sp_{2n}(\C)$ and let $T \subset Sp(n)$ be the maximal torus.
Let $p: \Omega Sp(n) \to LSp(n)/T$ denote the composition $\Omega
Sp(n) \hookrightarrow LSp(n) \to LSp(n)/T$ of the inclusion and
natural projection.  The type $C$ affine Stanley symmetric functions
$\Qf_w$ can be identified via Theorem \ref{T:main} (see below) with the
pullbacks $p^*(\xi^w)$ of the Schubert classes $\xi^w \in
H^*(LSp(n)/T)$.  This follows from \eqref{E:Stacoeff} and \cite[Remark
8.6]{Lam2}.  See also \cite[Remark 4.6]{Lam2}. For $w\in \tC_n^0$,
$p^*(\xi^w)$ is itself
a Schubert class in $H^*(\Omega Sp(n))\cong H^*(\Gr_{Sp_{2n}(\C)})$
as detailed below.

\subsection{(Co)homology Schubert polynomials}\label{SS:Schub}
The Hopf algebras $H^*(\Gr_{Sp_{2n}(\C)})$ and
$H_*(\Gr_{Sp_{2n}(\C)})$ are dual via the cap product.  The Schubert
bases $\{\xi_x \in H_*(\Gr_G)\}$ and $\{\xi^x \in H^*(\Gr_G)\}$ are
dual under the cap product and are both indexed by the Grassmannian
elements $x \in \tC_n^0$.

\begin{thm} \label{T:main} There are dual Hopf algebra isomorphisms
\begin{align*}
\Phi: \Hc_{(n)} \to H_*(\Gr_{Sp_{2n}(\C)}) \\
\Psi: H^*(\Gr_{Sp_{2n}(\C)}) \to \Hc^{(n)}
\end{align*}
such that
\begin{align*}
\Phi(P_i) &= \xi_{\rho_i} &\qquad& \text{for $1\le i\le 2n$, and} \\
\Psi(\xi^w) &= \Qf_w & \qquad &\text{for $w \in \tC_n^0$.}
\end{align*}
\end{thm}
Since $\Hc_{(n)} = \Z[P_1,P_3,\ldots,P_{2n-1}]$, we obtain in
particular that $H_*(\Gr_{Sp_{2n}(\C)})$ is a polynomial algebra on
$\xi_{\rho_1}, \xi_{\rho_3}, \ldots, \xi_{\rho_{2n-1}}$. It also
follows that the basis $\{ \Pf_w \mid w\in \tC_n^0\}$ of $\Hc_{(n)}$
dual to $\{\Qf_w \mid w \in \tC_n^0\} \subset \Hc^{(n)}$ maps to the
homology Schubert classes $\xi_w \in H^*(\Gr_{Sp_{2n}(\C)})$.  The
symmetric functions $\Pf_w$ are Schubert polynomials for
$H_*(\Gr_{Sp_{2n}(\C)})$ and are the type $C$ analogue of $k$-Schur
functions \cite{Lam2,LLM}. Examples are given in
Appendix~\ref{S:Pfdata}.

\subsection{Pieri rule for $H_*(\Gr_{Sp_{2n}(\C)})$}
We also give a positive formula for the multiplication of an
arbitrary homology class by a special class.

\begin{thm}\label{T:Pieri}
Let $w \in \tC_n^0$.  Then in $H_*(\Gr_{Sp_{2n}(\C)})$ we have
$$
\xi_{\rho_i} \; \xi_w = \sum_{v \in \WS_i} 2^{c(v) - 1} \, \xi_{vw}
$$
where the sum is taken over all $v \in \WS_i$ such that $\ell(vw) =
\ell(v) + \ell(w)$ and $vw \in \tC_n^0$.
\end{thm}

\subsection{Future work and other directions}

\subsubsection{Pieri rule for $H^*(\Gr_{Sp_{2n}(\C)})$ and explicit
description of homology Schubert basis}

We hope to describe the symmetric functions $\{ \Pf_w \mid w\in
\tC_n^0\} \subset \Hc_{(n)}$ explicitly in the future, perhaps in a
manner similar to the {\it strong tableaux} in \cite{LLMS}.  As is
explained in \cite{LLMS}, the description of $\Pf_w$ is essentially
equivalent to the description of a Pieri rule for
$H^*(\Gr_{Sp_{2n}(\C)})$.

\subsubsection{The special orthogonal groups}
A generalization of our work to the special orthogonal groups $G =
SO_n(\C)$, together with the $G = SL_n(\C)$ case in \cite{Lam2},
would complete the analysis of the classical groups. The symmetric
function description of $H_*(\Gr_{SO_n(\C)})$ is likely to be more
involved as it is not a polynomial algebra over $\Z$.

\subsubsection{Comparison with finite case}
We hope to explore the relationship between our symmetric functions
and the ``type $B$'' Stanley symmetric functions and classical type
Schubert polynomials of Fomin and Kirillov \cite{FK}, and of Billey
and Haiman \cite{BH}.  In particular, specializing $A_n = 0$ in
Theorem \ref{T:P} we obtain an expression nearly the same as the
formula \cite[(4.1)]{FK}.

\subsubsection{Embedding of groups and branching of Schubert classes}
We intend to study the behavior of the affine Schubert classes
studied here and in \cite{Lam2} induced by the inclusions of compact
groups:
\begin{align*}
SU(n) \subset SU(n+1) \ \ \ Sp(n) \subset Sp(n+1) \ \ \ \
 Sp(n) \subset SU(2n) \ \ \ SU(n) \subset
Sp(n).
\end{align*}

In particular, the symmetric functions $\Pf_w$ and $\Qf_w$ have
positivity properties with respect to expansions involving Schur
$P$-functions, Schur $Q$-functions, and ordinary Schur functions.

\subsubsection{Work of Ginzburg and Bezrukavnikov, Finkelberg and Mirkovic}
The rings $H_*(\Gr_G)$ and $H^*(\Gr_G)$ were also studied by
Ginzburg \cite{Gin} and by Bezrukavnikov, Finkelberg and Mirkovic
\cite{BFM} from the point of view of geometric representation
theory.  The connection with our point of view is unclear since the
Schubert basis is as yet unavailable in their descriptions, although
part of the Schubert basis is described by Ginzburg.

\subsection{Organization}
In section \ref{S:Sym} we give notation for symmetric functions and
describe the dual Hopf algebras $\Hc_{(n)}$ and $\Hc^{(n)}$.  In
section \ref{S:root} we fix notation concerning affine root systems
and Weyl groups.  In section \ref{S:nilHecke} we explain the
connection between the Peterson and the Fomin-Stanley subalgebras,
and the homology of the affine Grassmannian.  The material in
sections \ref{S:root}--\ref{S:nilHecke} are valid for the affine
Grassmannian of any simply-connected simple complex algebraic group
$G$.

Apart from Proposition \ref{prop:del A}, the remainder of the paper
specializes to the case $G = Sp_{2n}(\C)$.  In section \ref{S:main}
we prove our main results (Theorems \ref{T:affStan}, \ref{T:main}
and \ref{T:Pieri}) modulo two nilHecke algebra calculations --
Theorems \ref{T:P} and \ref{T:phiP}.  Section \ref{S:Pieri} is
devoted to the study of the Bruhat order of $\tC_n$ restricted to
$\WS$, and to the proof of Theorem \ref{T:P}.  Section \ref{S:Hopf}
presents a general formula (Proposition \ref{prop:del A}) for the
coproduct in a nilHecke algebra and uses it to prove Theorem
\ref{T:phiP}. Some data, in particular for the type $C$ affine
Stanley symmetric functions $\Qf_w$ and $k$-Schur functions $\Pf_w$,
is given in Appendices~\ref{S:Qfdata} and \ref{S:Pfdata}.

\subsection{Acknowledgements.}
Many thanks to Jennifer Morse for pointing us in the right direction
for finding the leading monomial in a Grassmannian $\Qf_w$. We thank
Mike Zabrocki for discussions at an early stage of this work, and
Nicolas Thi\'ery and Florent Hivert for their support with MuPAD-Combinat~\cite{HT:2003}.
This work was partially supported by the NSF grants DMS--0600677,
DMS--0501101, DMS--0652641, DMS--0652648, and DMS--0652652.

\section{Symmetric functions}\label{S:Sym}
In this section we study a subring $\Hm_{(n)}$ and subquotient
$\Hc^{(n)}$ of the ring of symmetric functions.  Let $\La$ be the
Hopf algebra of symmetric functions over $\Z$. It has a number of
bases indexed by partitions $\la$:
\begin{alignat*}{2}
  s_\la & \qquad&\text{Schur \cite[I.3]{Mac}} \\
  h_\la & &\text{homogeneous \cite[I.1]{Mac}} \\
  p_\la && \text{power sums \cite[I.1]{Mac}} \\
  m_\la && \text{monomial \cite[I.1]{Mac}} \\
  P_\la[X;t]&& \text{Hall-Littlewood $P$ \cite[III.2]{Mac}}  \\
  Q_\la[X;t]&& \text{Hall-Littlewood $Q$ \cite[III.2]{Mac}}
\end{alignat*}
The power sums are a basis over $\Q$ \cite[I.2.12]{Mac}
and the Hall-Littlewood $P$- and $Q$-functions are
a basis over $\Q(t)$ \cite[III.2.7,2.11]{Mac}.

Let $\ip{\cdot,\cdot}:\La\otimes\La \to \Z$ be the pairing defined by
\begin{align} \label{E:scalar}
  \ip{h_\la,m_\mu} = \delta_{\la\mu}.
\end{align}
It has reproducing kernel \cite[I.4.1,4.2]{Mac}
\begin{equation} \label{E:kernel}
\begin{split}
  \Omega &:= \prod_{i,j\ge1} \dfrac{1}{1-x_iy_j} \\
  &= \sum_\la h_\la[X] m_\la[Y] \\
  &= \sum_\la z_\la^{-1} p_\la[X] p_\la[Y],
\end{split}
\end{equation}
where $z_\la = \prod_{i\ge1} i^{m_i(\la)} m_i(\la)!$ and $m_i(\la)$ is the
number of times the part $i$ occurs in $\la$.
Here and elsewhere, unless otherwise specified the sum runs over the
set of all partitions.

The Schur $P$ and $Q$ functions are defined by \cite[III.8]{Mac}
\begin{equation} \label{E:PQdef}
\begin{split}
  P_\la[X] &= P_\la[X;-1] \\
  Q_\la[X] &= Q_\la[X;-1] = 2^{\ell(\la)} P_\la[X]
\end{split}
\end{equation}
where $\ell(\la)$ is the number of nonzero parts of $\la$. We have \cite[III.8.7]{Mac}
\begin{equation} \label{E:strictvanish}
P_\la[X]=Q_\la[X]=0\qquad\text{if $\la\not\in \SP$}
\end{equation}
where $\SP$ is the set of strict partitions $\la$, those with $\la_1>\la_2>\dotsm$;
see~\eqref{E:Pbasis} and \eqref{E:Qbasis}.

\subsection{Homology ring}

Define the Hopf subalgebra $\Hminf\subset\La$ by
\begin{align}
  \Hminf &= \Z[P_1,P_3,P_5,\dotsc].
\end{align}
The $P_i$ for $i$ odd, are algebraically independent, so that
\begin{align}
\Hminf = \bigoplus_{\la\in\OP} \Z P_{\la_1} P_{\la_2}\dotsm
\end{align}
where $\OP$ is the set of partitions with odd parts. The Hopf
structure on $\Hminf$ is given by
\begin{align} \label{E:DeltaP}
  \Delta(P_r) = 1 \otimes P_r + P_r \otimes 1 + 2 \sum_{0<s<r} P_s
  \otimes P_{r-s},
\end{align}
where the $P_i$ for $i$ even, satisfy only the relations
\cite[III.8.2']{Mac}
\begin{align} \label{E:Prel}
  P_{2i} = 2 \left(P_1 P_{2i-1} -  P_2 P_{2i-2} + \dotsm + (-1)^{i-2}
  P_{i-1}P_{i+1}\right)+(-1)^{i-1}
  P_i^2.
\end{align}

Iterating~\cite[III.8.15]{Mac} yields the relation
\begin{align} \label{E:Ptri}
  P_{\la_1} P_{\la_2} \dotsm = \sum_{\substack{\mu\in\SP \\
  \mu \dom \la}} L_{\mu\la} P_\mu,
\end{align}
where $L_{\mu\la}\in\Z_{\ge0}$ and $L_{\mu\mu}=1$.  Here $\dom$
denotes the dominance partial order on partitions \cite[I.1]{Mac}.  It follows that
\begin{align} \label{E:Pbasis}
  \Hminf = \bigoplus_{\la\in\SP} \Z P_\la.
\end{align}

Define the Hopf subalgebra $\Hm_{(n)}\subset\Hminf$ by
\begin{equation}
\begin{split}
  \Hm_{(n)} &= \Z[P_1,P_2,\dotsc,P_{2n}] \\
  &= \Z[P_1,P_3,\dotsc,P_{2n-1}] \\
  &= \bigoplus_{\substack{\la\in\OP \\ \la_1 \le 2n-1}} \Z
  P_{\la_1}P_{\la_2}\dotsm.
\end{split}
\end{equation}

\subsection{Cohomology ring}
Define
\begin{align}
  \Hcinf &= \Z[Q_1,Q_2,\dotsc] \subset \La.
\end{align}
By \eqref{E:Ptri} we have
\begin{align} \label{E:Qbasis}
  \Hcinf = \bigoplus_{\la\in\SP} \Z Q_\la.
\end{align}
Define the pairing $\ipp{\cdot,\cdot}:\Hminf\times \Hcinf \to \Z$ by
\cite[III.8.12]{Mac}
\begin{align}\label{E:pair}
\ipp{P_\la,Q_\mu}=\delta_{\la\mu}\qquad\text{for $\la,\mu\in\SP$.}
\end{align}
The pairing $\ipp{\cdot,\cdot}$ has reproducing kernel
\begin{equation} \label{E:-1kernel}
\begin{split}
  \Omega_{-1} &:= \prod_{i,j\ge1} \dfrac{1+x_iy_j}{1-x_iy_j} \\
  &= \sum_{\la\in\SP} P_\la[X] Q_\la[Y] \\
  &= \sum_\la P_{\la_1}[X]P_{\la_2}[X]\dotsm M_\la[Y] \\
  &= \sum_{\la\in\OP} z_\la^{-1} 2^{\ell(\la)} p_\la[X] p_\la[Y],
\end{split}
\end{equation}
where $M_\la = 2^{\ell(\la)} m_\la$.
These equalities hold by definition,
\cite[III.8.13]{Mac}, setting $t=-1$ in \cite[III.4.2]{Mac},
and \cite[III.8.12]{Mac}.

Let $J_n\subset \Hcinf$ be the ideal given by the annihilator of
$\Hm_{(n)}\subset\Hminf$ with respect to $\ipp{\cdot,\cdot}$. Define
\begin{align}
  \Hc^{(n)} = \Hcinf/J_n
\end{align}
which is a Hopf quotient algebra of $\Hcinf$.  The pairing
$\ipp{\cdot,\cdot}$ descends to a perfect pairing $\Hc_{(n)} \otimes
\Hc^{(n)}\to\Z$ which by \eqref{E:-1kernel} has reproducing kernel
\begin{align}
\Omega_{-1}^{(n)} = \sum_{\la_1\le 2n}
P_{\la_1}[X]P_{\la_2}[X]\dotsm  M_\la[Y].
\end{align}

\subsection{Comparing $\La$ with $\Hminf$ and $\Hcinf$}
Since $\La=\bigoplus_\la \Z h_\la$ \cite[I.2.8]{Mac}
one may define a surjective ring homomorphism $\theta:\La\to\Hcinf$
defined by $\theta(h_i)=Q_i$ for $i\in\Z_{>0}$. Over $\Q$ it may be
defined by $\theta(p_{2i})=0$ and $\theta(p_{2i-1})=2p_{2i-1}$ for
$i\in\Z_{>0}$ \cite[Ex. III.8.10]{Mac}.

Let $\iota:\Hminf\to\La$ be the inclusion
map.
\begin{lem}\label{L:IPs}
\begin{align}
 \ip{\iota(f),g} = \ipp{f,\theta(g)}\qquad\text{for $f\in\Hminf,\,g\in\La$.}
\end{align}
\end{lem}
\begin{proof}
By linearity one may reduce to the case $f=p_\la$ and $g=p_\mu$ for $\la,\mu\in\OP$.
By \eqref{E:kernel} and \eqref{E:-1kernel} we have
\begin{align*}
  \ipp{p_\la,\theta(p_\mu)} &= 2^{\ell(\mu)} \ipp{p_\la,p_\mu} \\
  &= 2^{\ell(\mu)-\ell(\la)} z_\la \delta_{\la\mu} \\
  &= z_\la\delta_{\la\mu} \\
  &= \ip{\iota(p_\la),p_\mu}.
\end{align*}
\end{proof}

Lemma \ref{L:IPs} can be restated as
\begin{align} \label{E:thetaOmega}
\theta^Y \Omega = \Omega_{-1}
\end{align}
where $\theta^Y$ means the operator $\theta$ applied to the $Y$ variables.

\begin{lem} \label{L:QM} For $\nu\in\SP$
\begin{align*}
  Q_\nu = \sum_\la L_{\nu\la} M_\la.
\end{align*}
In particular $\Hcinf\subset \bigoplus_\la \Z M_\la$.
\end{lem}
\begin{proof} Since both $\Omega$ and $\Omega_{-1}$ are
invariant under exchanging the $X$ and $Y$ variables, by~\eqref{E:thetaOmega}
and~\eqref{E:Ptri} we have
\begin{align*}
  \theta^Y \Omega &= \Omega_{-1} = \theta^X \Omega \\
  &= \theta^X \sum_\la h_\la[X] m_\la[Y] \\
  &= \sum_\la Q_{\la_1}[X] Q_{\la_2}[X] \cdots  m_\la[Y] \\
  &= \sum_\la P_{\la_1}[X] P_{\la_2}[X] \dotsm M_\la[Y] \\
  &= \sum_\la \sum_{\nu\in\SP} L_{\nu\la} P_\nu[X] M_\la[Y] \\
  &= \sum_{\nu\in\SP} P_\nu[X] \sum_\la L_{\nu\la} M_\la[Y].
\end{align*}
By \eqref{E:-1kernel} and \eqref{E:Pbasis}, taking the coefficient of $P_\nu[X]$, the Lemma follows.
\end{proof}

\subsection{A monomial-like basis for $\Hc^{(n)}$}
\label{sec:basishc} For a partition $\la$, define $\TT_\la =
\theta(m_\la)$.  We shall give a ``monomial'' basis of $\Hc^{(n)}$
using the $\TT_\la$.
Let $\chi(\textrm{true})=1$ and $\chi(\textrm{false})=0$.

\begin{lem} \label{L:TtoM}
For every partition $\la$, $$T_\la \in \chi(\la\in \OP) M_\la +
\sum_{\mu\triangleright \la} \Z M_\mu.$$
\end{lem}
\begin{proof} Define $y_\la = \prod_{i\ge1} m_i(\la)!$ where
$m_i(\la)$ is the multiplicity of the part $i$ in $\la$. By
expanding $p_\la$ it is easy to see that $p_\la\in y_\la m_\la +
\sum_{\mu\triangleright\la} \Z y_\mu m_\mu$. It follows that
$m_\la\in y_\la^{-1} p_\la  + \sum_{\mu\triangleright\la} \Q p_\mu$
and
\begin{equation} \label{E:Texp}
\begin{split}
\TT_\la &\in y_\la^{-1} \theta(p_\la) +\sum_{\mu\triangleright\la} \Q \theta(p_\mu) \\
  &= \chi(\la\in \OP) y_\la^{-1} 2^{\ell(\la)} p_\la +
  \sum_{\substack{\mu\triangleright\la \\ \mu\in \OP}} \Q p_\mu \\
  &= \chi(\la\in\OP) M_\la + \sum_{\mu\triangleright\la} \Q
  m_\mu.
\end{split}
\end{equation}
By Lemma \ref{L:QM}, $\TT_\la=\theta(m_\la)\in \Hcinf$ is a $\Z$-linear combination
of the $M_\mu$. Since by definition the $M_\mu$ are integer multiples of
the $m_\mu$, \eqref{E:Texp} expresses $\TT_\la$ as a $\Q$-linear combination
of the $M_\mu$.
Since the $M_\mu$ are independent, the coefficients must then be integers.
\end{proof}

\begin{lem} \label{L:PT} For $\la,\mu\in\OP$,
\begin{align} \label{E:PTtri}
\ipp{P_{\la_1}P_{\la_2}\dotsm,T_\mu} &= 0\qquad\text{unless
$\la\dom\mu$,} \\
\label{E:PTdiag}
   \ipp{P_{\la_1}P_{\la_2}\dotsm,T_\la} &= 1.
\end{align}
\end{lem}
\begin{proof} Let $A=(a_{\la\nu})$ and $B = (b_{\la\nu})$ be the change of basis matrices
\begin{align*}
h_\lambda = \sum_{\nu \rdom \la} a_{\la\nu} p_\nu
\qquad\text{and}\qquad p_\nu = \sum_{\rho \rdom \nu} b_{\nu\rho}
h_\rho.
\end{align*}
They are unitriangular and mutually inverse. We have
\begin{align*}
P_{\la_1}P_{\la_2}\dotsm &= 2^{-\ell(\la)} \theta(h_\la) \\
&= \sum_{\substack{\nu \rdom \la  \\
\nu \in \OP}} \sum_{\rho \rdom \nu} 2^{\ell(\nu)-\ell(\la)}
a_{\la\nu}b_{\nu\rho} h_\rho.
\end{align*}
For any partition $\mu$,
by Lemmata \ref{L:IPs} and \ref{L:TtoM} we have
\begin{align*}
  \ipp{P_{\la_1} P_{\la_2} \dotsm, T_\mu} &= \ip{P_{\la_1}P_{\la_2}\dotsm, m_\mu} \\
  &= \sum_{\substack{\nu \rdom \la  \\
\nu \in \OP}} \sum_{\rho \rdom \nu} 2^{\ell(\nu)-\ell(\la)}
a_{\la\nu}b_{\nu\rho} \delta_{\mu\rho}
\end{align*}
by \eqref{E:scalar}. But this sum is zero unless $\la\dom\mu$,
proving \eqref{E:PTtri}. When $\mu = \la$ we have
$\ipp{P_{\la_1}P_{\la_2}\dotsm,T_\la} = a_{\la\la} b_{\la\la} = 1$,
since $AB = I$.
\end{proof}

\begin{prop} \label{P:thetambasis} We have
\begin{align*}
  \Hcinf = \bigoplus_{\la\in\OP} \Z \TT_\la, \ \ \ \ \
  \Hc^{(n)} = \bigoplus_{\substack{\la\in\OP \\ \la_1 \le 2n}} \Z
  \TT_\la,
  \  \ \
  \ \
  J_n = \bigoplus_{\substack{\la\in\OP \\ \la_1 \ge 2n+1}} \Z
  T_\la.
\end{align*}
\end{prop}
\begin{proof}
This follows from Lemma \ref{L:PT}, which says that
$\{T_\mu\mid\mu\in\OP\}$ is unitriangularly related (over $\Z$) to the
$\Z$-basis of $\Hcinf$ that is $\ipp{\cdot,\cdot}$-dual to the
$\Z$-basis of $\Hminf$ given by $\{P_{\la_1}P_{\la_2}\dotsm\mid
\la\in\OP\}$.
\end{proof}
%


\subsection{Another realization of $\Hc^{(n)}$}
Let $I_k\subset \La$ be the ideal generated by $m_\la$ for $\la_1\ge
k$.  There is a natural ring isomorphism
\begin{align*}
  \Hcinf/(\Hcinf\cap I_{2n+1}) \cong (\Hcinf+I_{2n+1})/I_{2n+1}.
\end{align*}
By Proposition \ref{P:thetambasis}, we have $\Hcinf\cap I_{2n+1} =
J_n$. Therefore
\begin{align} \label{E:Ccoiso}
  \Hc^{(n)} \cong (\Hcinf+I_{2n+1})/I_{2n+1}.
\end{align}
It follows from \eqref{E:Ccoiso} and Lemma \ref{L:IPs} that
\begin{equation}\label{E:Mcoeff}
[P_{\la_1}P_{\la_2}\cdots, f] \ \text{is the coefficient of
$M_\lambda$ in $f$}
\end{equation}
for $f \in \Gamma^{(n)}$ and $\lambda$ satisfying $\la_1 \le 2n$.


\section{Affine root systems}
\label{S:root}

\subsection{Weyl group}
\label{SS:Weyl}

A Cartan datum is a pair $(I,A)$ where $I$ is a finite set (the set
of Dynkin nodes) and $A=(a_{ij}\mid i,j\in I)$ is a generalized
Cartan matrix, which by definition satisfies $a_{ii}=2$ for $i\in
I$, $a_{ij}\le0$ for $i\ne j$, and $a_{ij}<0$ if and only if
$a_{ji}<0$. The Cartan datum $(I,A)$ is of finite type if $A$ is
nonsingular and of affine type if $A$ has corank one.

Given a Cartan datum $(I,A)$, for $i,j\in I$ with $i\ne j$, define
the integers $m_{ij}=2,3,4,6,\infty$ according as $a_{ij}a_{ji}$ is
$0,1,2,3$, or $\ge4$. The Weyl group $W=W(I,A)$ is the Coxeter group
with generators $s_i$ for $i\in I$ such that $s_i^2=1$ for all $i\in
I$ and braid relations
\begin{align} \label{E:braid}
\underbrace{s_i s_j s_i \dotsm }_{\text{$m_{ij}$ times}} =
\underbrace{s_j s_i s_j \dotsm }_{\text{$m_{ij}$
times}}\qquad\text{for $i\ne j$.}
\end{align}
The length function $\ell: W \to \Z$ is given by $\ell(w) = l$ if a
shortest expression $w=s_{i_1}s_{i_2}\dotsm s_{i_l}$ of $w$ as a
product of the $s_i$, is of length $l$.  We call such an expression
$w=s_{i_1}s_{i_2}\dotsm s_{i_l}$ a {\it reduced expression}.  The
word $i_1i_2\dotsm i_l$ consisting of the indices of a reduced
expression is called a {\it reduced word} for $w$. We denote by
$\Red(w)$ the set of reduced words for $w$. We write $u\equiv u'$ if
$u,u'\in\Red(w)$ for some $w\in W$.

An element $s\in W$ is a reflection if $s=ws_iw^{-1}$ for some $i\in
I$ and $w\in W$.

The Bruhat order on $W_\af$ is defined by $v\le w$ if some
(equivalently every) reduced word of $w$ has a subword that is a
reduced word for $v$. Alternatively $v\lessdot w$ if $v^{-1}w$ is a
reflection and $\ell(w)=\ell(v)+1$.

We now fix notation for an affine Cartan datum. Let $(I,A)$ be the
finite Cartan datum associated with the Lie algebra $\geh$ of a
simple simply-connected complex algebraic group $G$. Let
$I=\{1,2,\dotsc,n\}$ where $n$ is the rank of $\geh$. Let
$(I_\af,A_\af)$ be the affine Cartan datum for the untwisted affine
algebra $\geh_\af = (\C[t,t^{-1}] \otimes_\C \geh) \oplus \C K
\oplus \C d$ \cite[\S 7.2]{Kac}. We write $I_\af=\{0\} \sqcup I$
where $0\in I_\af$ is the distinguished Kac $0$ node \cite[\S
4.8]{Kac}. Let $W=W(I,A)$ be the finite Weyl group and
$W_\af=W(I_\af,A_\af)$ the affine Weyl group. We denote by
$W^0_\af\subset W_\af$ the set of \textit{Grassmannian} elements,
which by definition are the minimal length coset representatives of
$W_\af/W$.

\begin{example} \label{X:WeylC}
If $\geh=\gspn$, the Cartan matrix for $\geh_\af$ is given by
$a_{ii}=2$ for $i\in I_\af$, $a_{i,i+1}=a_{i+1,i}=-1$ for $1 \le
i\le n-2$, $a_{01}=-1$, $a_{10}=-2$, $a_{n-1,n}=-2$, $a_{n,n-1}=-1$,
and $a_{ij}=0$ if $|i-j|\ge 2$. The affine Weyl group $W_\af=\tC_n$
has generators $\{s_0,s_1,\dotsc,s_n\}$ and relations
\begin{align}
\notag s_i^2 &= 1 &  \\
\notag s_i s_j &= s_j s_i & \text{if $|i - j| > 1$} \\
\label{E:AffWeylRelations}
s_i s_{i+1} s_i &= s_{i+1}s_i s_{i+1} & \text{if $1\le i\le n-2$} \\
\notag s_0s_1s_0 s_1 & = s_1 s_0 s_1 s_0 \\
\notag s_{n-1} s_n s_{n-1} s_n &= s_{n}  s_{n-1} s_n s_{n-1}.
\end{align}
For $W_\af=\tC_n$, we use the notation $W_\af^0=\tC_n^0$ and
$W=C_n$.
\end{example}

\subsection{Affine root, coroot, and weight lattices}
Let $P_\af= \Z\delta \oplus \bigoplus_{i\in I_\af} \Z \La_i$ be the
affine weight lattice, where $\delta$ is the null root and the
$\La_i$ are the fundamental weights. Let
$P_\af^*=\mathrm{Hom}_\Z(P_\af,\Z)$ be the dual weight lattice and
$\ip{\cdot,\cdot}:P_\af^* \times P_\af \to \Z$ the natural perfect
pairing. Let $\{d\} \cup \{\al_i^\vee\mid i\in I_\af\}$ be the basis
of $P_\af^*$ dual to the above basis of $P_\af$; in particular,
\begin{align*}
  \ip{\al_i^\vee,\La_j} &= \delta_{ij} &\qquad&\text{for $i,j\in
  I_\af$} \\
  \ip{\al_i^\vee,\delta} &= 0 &&\text{for $i\in I_\af$}
\end{align*}
where $\delta_{ij}$ is the Kronecker delta. The $\al_i^\vee$ are
called simple coroots. For $j\in I_\af$ define the simple root
$\al_j\in P_\af$ by
\begin{align}
  \al_j = \sum_{i\in I_\af} a_{ij} \La_i + \delta_{j0} \delta.
\end{align}
Note that
\begin{align} \label{E:Cartan}
  \ip{\al_i^\vee, \al_j} = a_{ij}\qquad\text{for all $i,j\in
I_\af$.}
\end{align}
Due to a linear dependence among the columns of the Cartan matrix,
we have
\begin{align} \label{E:nullroot}
  \delta &= \al_0 + \theta
\end{align}
where $\theta$ is the highest root of $\geh$. Let
$Q_\af=\bigoplus_{i\in I_\af} \Z \al_i\subset P_\af$ and
$Q_\af^\vee=\bigoplus_{i\in I_\af} \Z \al_i^\vee$ be the affine root
and coroot lattices. The nullroot satisfies
\begin{align} \label{E:nullz}
  \ip{\mu,\delta}=0\qquad\text{for all $\mu\in Q_\af^\vee$.}
\end{align}
Similarly a dependence among the rows of the Cartan matrix, yields
the \textit{canonical central element} $K\in Q_\af^\vee$ defined by
\begin{align} \label{E:central}
  K &= \al_0^\vee + \theta^\vee
\end{align}
where $\theta^\vee$ is the coroot associated to $\theta$ (defined in
the next subsection). $K$ satisfies
\begin{align} \label{E:Kz}
  \ip{K,\la}=0 \qquad\text{for all $\la\in Q_\af$.}
\end{align}
\begin{example} \label{X:deltaK} For $\geh=\gspn$, $\geh_\af$ has
nullroot $\delta=\al_0+2(\al_1+\dotsm+\al_{n-1})+\al_n$ and
canonical central element $K=\al_0^\vee+\dotsm+\al_n^\vee$.
\end{example}
The affine Weyl group $W_\af$ acts on $P_\af$ and $P_\af^\vee$ by
\begin{align} \label{E:WonP}
  s_i \la &= \la - \al_i \ip{\al_i^\vee,\la}&\qquad&\text{for
  $\la\in P_\af$} \\
  s_i \mu &= \mu - \al_i^\vee\ip{\mu,\al_i}&&\text{for $\mu\in
  P_\af^\vee$.}
\end{align}
One may show that
\begin{align}
  \ip{w\mu,w\la}=\ip{\mu,\la}\qquad\text{for $w\in W_\af$,
$\la\in P_\af$, $\mu\in P_\af^\vee$.}
\end{align}
By \eqref{E:nullz} and \eqref{E:Kz} we have
\begin{align}
  w \delta = \delta, \qquad w K = K \qquad\text{for all $w\in
W_\af$.}
\end{align}

\subsection{Finite root, coroot, and weight lattices}
The finite coroot lattice is defined by $Q^\vee=\bigoplus_{i\in I}
\Z \al_i^\vee \subset Q_\af^\vee$. The finite root and weight
lattices $Q$ and $P$ are quotients of their affine counterparts
$Q_\af$ and $P_\af$, but by abuse we will define them as
sublattices. The finite root lattice is defined by
$Q=\bigoplus_{i\in I} \Z\al_i \subset Q_\af\subset P_\af$. The
finite weight lattice is defined by $P=\bigoplus_{i\in I}
\Z\om_i\subset P_\af$ where
\begin{align}
  \om_i = \La_i - \ip{K,\La_i} \La_0
\end{align}
for $i\in I$; these are the fundamental weights of $\geh$. We have
\begin{align}
  \ip{\al_i^\vee,\om_j} = \delta_{ij} \qquad\text{for $i,j\in I$.}
\end{align}



\subsection{Roots}
\label{SS:affroots}

The root system of $\geh$ may be defined by
\begin{align}
  R = W \cdot \{\al_i \mid i\in I\}.
\end{align}
Given $\al\in R$ with $\al=u\al_i$ for some $u\in W$ and $i\in I$,
its associated coroot is defined by $\al^\vee=u \al_i^\vee\in
Q^\vee$. Its associated reflection is $s_\al=u s_i u^{-1}$. Both are
independent of the choice of $u$ and $i$. There is a decomposition
$R = R^+ \cup -R^+$ where $R^+=R\cap \bigoplus_{i\in I} \Z_{\ge0}
\al_i$ is the set of positive roots.

The set of affine roots $R_\af\subset Q_\af$ is given by the set of
nonzero elements in the set $R+\Z\delta$. We have $R_\af = R_\af^+
\cup -R_\af^+$ where $R_\af^+$ is the set of positive affine roots,
which have the form $\al+m\delta$ where either $m>0$ or both $m=0$
and $\al\in R^+$. Equivalently, $R_\af^+ = R_\af \cap
\bigoplus_{i\in I_\af} \Z_{\ge0} \al_i$.

The set of real affine roots is defined by
$$R_\af^\re = W_\af \cdot \{\al_i\mid i\in I_\af\}.
$$
For $\al=u\al_i\in R_\af^\re$ for $u\in W_\af$ and $i\in I_\af$
define the associated coroot by $\al^\vee = u\al_i^\vee$ and
associated reflection $s_\al\in W_\af$ by $s_\al = u s_i u^{-1}$; as
before one may show these definitions are independent of $u$ and
$i$.

Let $v\lessdot w$ in $W_\af$. Then $s=v^{-1}w$ is a reflection $s=u
s_i u^{-1}$ for some $i\in I_\af$ and $u\in W_\af$. Let $u$ be
shortest so that $\al=u\al_i$ is a positive real root. For later use
we denote this root $\al$ by $\al_{vw}$ and its associated coroot by
$\al_{vw}^\vee$.

\begin{example} \label{X:covercoroot}
Let $W_\af=\tilde{C}_3$, $w=s_1s_2s_3s_2s_1s_0$, and
$v=s_1s_3s_2s_1s_0$; this defines a cover $v\lessdot w=vs_\alpha$.
Then $s_\al=(s_0s_1s_2s_3)s_2(s_3s_2s_1s_0)$ and
\begin{align*}
  \al_{vw}^\vee &= s_0s_1s_2s_3(\al_2^\vee) \\
  &= 2\al_0^\vee+\al_1^\vee+\al_2^\vee+2\al_3^\vee.
\end{align*}
\end{example}

\subsection{Level 0 action}
There is a surjective group homomorphism $W_\af\to W$ given by
$s_i\mapsto s_i$ for $i\in I$ and $s_0\mapsto s_\theta$ where
$\theta\in R^+$ is the highest root.
Since $W$ acts on $P$, $W_\af$ acts on $P$ via the above
homomorphism; this is called the level zero action. It is not
faithful since $s_0$ and $s_\theta$ are different elements of
$W_\af$.

\section{NilHecke algebra and affine Grassmannian}
\label{S:nilHecke}
\subsection{(Co)homology of affine Grassmannian}
For this section we fix $G$ a simple and simply-connected complex
algebraic group with Weyl group $W$, and Cartan datum $(I,A)$ as in
section \ref{SS:Weyl}. Let $K$ denote a maximal compact subgroup of
$G$ and $T$ denote a maximal torus in $K$.

Let $\F = \C((t))$ and $\O = \C[[t]]$.  The affine Grassmannian
$\Gr_G$ is the ind-scheme $G(\F)/G(\O)$ (see~\cite{Kum}).  It is a
homogeneous space for the affine Kac-Moody group $\G$ associated to
$W_\af$.  It is a classical result due to Quillen that the space
$\Gr_G$ is homotopy-equivalent to the space $\Omega K$ of based
loops in $K$; see for example~\cite{GR,Pr}.

The group $\G$ possesses a {\it Bruhat decomposition} $\G =
\bigcup_{w \in W_\af}\BB w \BB$ where $\BB$ denotes the Iwahori
subgroup. The Bruhat decomposition induces a decomposition of
$\Gr_G$ into {\it Schubert cells} $\Omega_w = \BB w G(\O) \subset
G(\F)/G(\O)$.  Thus the equivariant homology $H_T(\Gr_G)$ and
cohomology $H^T(\Gr_G)$ of $\Gr_G$ are free $S = H^T({\rm
pt})$-modules with Schubert bases $\xi_x^T \in H_T(\Gr_G)$ and
$\xi^x_T \in H^T(\Gr_G)$.  Similarly, the homology $H_T(\Gr_G)$ and
cohomology $H^T(\Gr_G)$ of $\Gr_G$ are free $\Z$-modules with
Schubert bases $\xi_x \in H_T(\Gr_G)$ and $\xi^x \in H^T(\Gr_G)$.
The index $x$ varies over the Grassmannian elements $W_\af^0$. We
refer the reader to \cite{KK, Kum} for the general construction and
properties of Schubert bases in the Kac-Moody setting.

The pointwise multiplication of loops on $K$ induces the structure
of dual Hopf algebras over $\Z$ to $H_*(\Gr_G)$ and $H^*(\Gr_G)$,
and the structure of dual Hopf algebras over $S$ to $H_T(\Gr_G)$ and
$H^T(\Gr_G)$.   This is a special feature of the affine Grassmannian
unavailable in the more general Kac-Moody setting.

\subsection{NilCoxeter algebra}
The nilCoxeter algebra $\NilCox$ is the associative $\Z$-algebra
with generators $A_i$ for $i \in I$ and relations $A_i^2=0$ for
$i\in I$ and braid relations
\begin{align}
\underbrace{A_i A_j A_i \dotsm }_{\text{$m_{ij}$ times}} =
\underbrace{A_j A_i A_j \dotsm }_{\text{$m_{ij}$
times}}\qquad\text{for $i\ne j$.}
\end{align}
Since these are the same braid relations \eqref{E:braid} satisfied
by $s_i\in W$, for $w\in W$ one may define $A_w=A_{i_1}A_{i_2}\dotsm
A_{i_l}$ for any $i_1i_2\dotsm i_l\in \Red(w)$.

The algebra $\NilCox$ is a free $\Z$-module with basis $\{A_w \mid w
\in W\}$. In this basis, the multiplication is given by
$$
A_v A_u = \begin{cases} A_{vu} & \text{if $\ell(v) + \ell(u) = \ell(vu)$}\\
0 & \text{otherwise.}
\end{cases}
$$

\begin{example} \label{X:NilCoxC}
For the affine Cartan datum of Example \ref{X:WeylC}, $\NilCox$ has
generators $A_i$ for $i\in I_\af$ and relations
\begin{align*}
A_i^2 &= 0 &  \\
A_i A_j &= A_j A_i & \text{if $|i - j| > 1$} \\
A_i A_{i+1} A_i &= A_{i+1}A_i A_{i+1} & \text{if $1\le i\le n-2$} \\
A_0A_1A_0 A_1 & = A_1 A_0 A_1 A_0 \\
A_{n-1} A_n A_{n-1} A_n &= A_{n}  A_{n-1} A_n A_{n-1}
\end{align*}
\end{example}

\subsection{Kostant and Kumar's NilHecke algebra}
Let $P$ be the weight lattice of $\geh$ and $S={\rm Sym}(P)$ the
symmetric algebra. The Peterson affine nilHecke algebra $\A$ is by
definition\footnote{The nilHecke algebra of Kostant and Kumar
\cite{KK} for the affine Cartan datum, uses a larger weight lattice
than Peterson's nilHecke algebra. See \cite{Lam2} for a comparison
of the two.} the associative $\Z$-algebra generated by $S$ and the
nilCoxeter algebra $\NilCox$ for the affine Cartan datum
$(I_\af,A_\af)$ with
\begin{align} \label{E:easycommute}
  A_i \la = (s_i\cdot\la) A_i + \ip{\al_i^\vee,\la}1 \qquad\text{for $i\in
I_\af$ and $\la\in P$.}
\end{align}
Consequently $\NilH$ is a free left $S$-module with basis $\{A_w
\mid w\in W_\af\}$.

Iterating \eqref{E:easycommute} produces the following relation.

\begin{lem} \label{lem:commute}
For $x \in W_\af$ and $\la \in P$,
\begin{align} \label{E:AonP}
  A_x \la = (x \cdot \la) A_x + \sum_{y\lessdot x}
\ip{\al_{yx}^\vee, \la} A_y
\end{align}
where $\al_{yx}^\vee$ is defined in section \ref{SS:affroots}.
\end{lem}

\begin{prop} \cite{P} \label{P:Atensoraction}
Let $M$ and $N$ be left $\A$-modules. Define $M \otimes_S N =
(M\otimes_\Z N)/\langle sm \otimes n - m \otimes sn \mid s\in S;
m\in M; n\in N\rangle$. Then $M \otimes_S N$ is a left $\A$-module
via
\begin{equation} \label{E:Atensoraction}
\begin{split}
  A_i \cdot (m \otimes n) &= (A_i \cdot m ) \otimes n + m \otimes
  (A_i\cdot n) - \alpha_i (A_i \cdot m) \otimes (A_i \cdot n) \\
  s \cdot (m \otimes n) &= sm \otimes n.
\end{split}
\end{equation}
\end{prop}
Since the proof of this proposition is not readily available in the
literature, we include a proof.
\begin{proof}
We verify \eqref{E:easycommute}.
\begin{align*}
&A_i \cdot (\la \cdot (m \otimes n)) \\&= (A_i \cdot \la m) \otimes
n + \la m \otimes (A_i \cdot n) - \alpha_i
(A_i \cdot\la m) \otimes (A_i \cdot n) \\
&= (s_i \la)(A_i m \otimes n )+ \ip{\al_i^\vee,\la}m \otimes n + \la
m \otimes A_i n \\ & \ \ \ \ \  - \alpha_i(s_i \la)A_im \otimes A_in
- \alpha_i\ip{\al_i^\vee,\la}m
\otimes A_in \\
&= (s_i \la)(A_i m \otimes n ) + (\la- \alpha_i\ip{\al_i^\vee,\la})
m \otimes A_in \\ & \ \ \ \ \ - (s_i \la)\alpha_iA_im \otimes A_in +
\ip{\al_i^\vee,\la}m \otimes n \\
&= (s_i \la) \cdot A_i \cdot (m \otimes n) + \ip{\al_i^\vee,\la}m
\otimes n.
\end{align*}
We verify $A_i^2 = 0$.
\begin{align*}
&A_i \cdot (A_i \cdot (m \otimes n)) = 2A_i m \otimes A_i n - (A_i \alpha_i A_i m) \otimes n \\
&= 2A_i m \otimes A_i n + \alpha_i A_i^2 m \otimes n -
\ip{\al_i^\vee,\al_i}A_i m \otimes A_i n = 0.
\end{align*}
To verify the braid relations, it is convenient to introduce the
elements $r_i = 1 - \alpha_i A_i \in \A$.  It is not difficult to
see that given \eqref{E:easycommute} the relation $(A_iA_j)^{m_{ij}}
= (A_j A_i)^{m_{ij}}$ is equivalent to $(r_i r_j)^{m_{ij}} = (r_j
r_i)^{m_{ij}}$.  An easy calculation shows that the $r_i$ act on $M
\otimes_S N$ by $r_i \cdot (m \otimes n) = r_i m \otimes r_i n$.  It
is clear that this action satisfies the braid relations for the
$r_i$.
\end{proof}

Thus there is a left $S$-module homomorphism $\Delta: \A \to \A
\otimes_S \A$ defined by
\begin{equation}
\begin{split}
  \Delta(a) &= a \cdot (1\otimes 1)\qquad\text{for $a\in \A$.}
\end{split}
\end{equation}
By \eqref{E:Atensoraction} we have
\begin{align}
\label{E:DeltaA}
  \Delta(A_i) &= A_i \otimes 1 + 1 \otimes A_i - A_i \otimes
  \alpha_i A_i \\
\label{E:DeltaS}
  \Delta(s) &= s\otimes 1.
\end{align}
The map $\Delta$ is injective so there is a linear map $\Delta(\A)
\otimes (\A\otimes_S \A) \to (\A\otimes_S \A)$ defined by
\begin{align}
  \Delta(a) \otimes (x\otimes y) \mapsto a \cdot (x\otimes y)
\end{align}
using the left $\A$-module structure on $\A\otimes_S \A$ afforded by
Proposition \ref{P:Atensoraction}. We deduce that this map yields a
ring structure on $\Delta(\A)$ and an action of $\Delta(\A)$ on
$\A\otimes_S \A$.

It follows by induction using Proposition \ref{P:Atensoraction} that
this action is computed explicitly as follows. Let $a\in\A$ and
$\Delta(a)=\sum_{w,v} A_w \otimes a_{wv} A_v$. Then
\begin{align}
  \Delta(a) \cdot (x\otimes y) = \sum_{w,v} A_w x \otimes a_{wv} A_v
  y.
\end{align}
In particular, if $b\in\A$ and $\Delta(b)=\sum_{w,v} A_w \otimes
b_{wv} A_v$ for $b_{wv}\in S$, then
\begin{align} \label{E:Deltamult}
  \Delta(ab)=\Delta(a)\Delta(b)=
  \sum_{w,v,w',v'} A_w A_{w'} \otimes a_{wv} A_v b_{w'v'} A_{v'}.
\end{align}
The ring structure on $\Delta(\A)$ does not extend to all of $\A
\otimes_S \A$ by the formula $(a \otimes b)(c \otimes d)=ac \otimes
bd$, because if it did, then since $s\otimes 1=1\otimes s$ we would
have $fs \otimes g = (f\otimes g)(s\otimes 1)=(f\otimes g)(1\otimes
s)=f \otimes gs$, which is false in general (say, for $g=1$ and
$fs\ne sf$). Equation \eqref{E:Deltamult} says that when this
``obvious" generally ill-defined multiplication formula is applied
to expressions coming from the action of $\Delta(\A)$ on $\A
\otimes_S \A$, the result is well-defined.

\subsection{The Peterson subalgebra and equivariant cohomology of affine Grassmannian}
The \textit{Peterson subalgebra} of $\A$ is the centralizer
$Z_\A(S)$ of $S$. It is a Hopf algebra over $S$ since the factorwise
product on $Z_\A(S) \otimes_S Z_\A(S)$ gives it an $S$-algebra
structure under which the restriction of $\Delta$ to $Z_\A(S)$, is
an $S$-algebra homomorphism.

\begin{thm} \cite{P} \cite[Theorem 4.4]{Lam2}
There is an $S$-Hopf algebra isomorphism $$j: H_T(\Gr_G) \to
Z_\A(S)$$ which is characterized by the property that for all $x\in
W_\af^0$, $j(\xi_x^T)$ is the unique element of $Z_\A(S)\cap (A_x +
\sum_{y\in W_\af\setminus W_\af^0} S\, A_y)$.
\end{thm}

For $x\in W_\af^0$ and $y\in W_\af$ let $j_x^y\in S$ be defined by
\begin{align} \label{E:jdef}
  j(\xi_x^T) = \sum_{y\in W_\af} j_x^y A_y.
\end{align}

\begin{prop} \cite{P} \cite[Theorem 6.3]{LS} \label{P:jco} \
\begin{enumerate}
\item
For $x\in W_\af^0$ and $y\in W_\af$, the polynomial $j_x^y$ is
either zero or homogeneous of degree $\ell(y)-\ell(x)$; in
particular it is zero if $\ell(y)<\ell(x)$.
\item
For $x,z\in W_\af^0$ we have
\begin{align}\label{E:HomTstructureconst}
  \xi_x^T \xi_z^T = \sum_y j_x^y \xi_{yz}^T
\end{align}
where $y$ runs over the $y\in W_\af$ such that $yz\in W_\af^0$ and
$\ell(yz)=\ell(y)+\ell(z)$.
\end{enumerate}
\end{prop}

We wish to compute $j_x^y$ in the ``nonequivariant case''
$\ell(x)=\ell(y)$, when $j_x^y\in\Z_{\ge0}$. For this purpose we
consider the maps that forget the $T$-equivariance.

\subsection{Affine Fomin-Stanley subalgebra}
\label{SS:FominStanley} Let $\phi_0:S \to \Z$ be the map that sends
a polynomial to its evaluation at $0$. By abuse of notation define
$\phi_0:\NilH\to \NilCox$ by $\phi_0(\sum_w s_wA_w) = \sum_w
\phi_0(s_w) A_w$ for $s_w\in S$. Peterson's $j$-map induces an
injective ring homomorphism $j_0:H_*(\Gr_G)\to \NilCox$ such that
the diagram commutes:
\begin{align} \label{E:zerospec}
\begin{CD}
  H_T(\Gr_G) @>j>> \NilH \\
  @V{\epsilon}VV @VV{\phi_0}V \\
  H_*(\Gr_G) @>>{j_0}> \NilCox
\end{CD}
\end{align}
where $\epsilon: H_T(\Gr_G) \to H_*(\Gr_G)$ is obtained by $\xi_x^T
\mapsto \xi_x$ and the evaluation $\phi_0$.

By \eqref{E:jdef} and \eqref{E:zerospec} we have
\begin{align} \label{E:jcoef}
  j_0(\xi_w) = \sum_{\substack{u\in W_\af \\ \ell(u)=\ell(w)}} j_w^u
  A_u\qquad\qquad\text{for $w\in W_\af^0$.}
\end{align}

The affine Fomin-Stanley subalgebra is defined in \cite{Lam2} by
\begin{align} \label{E:inB}
\B = \{a\in \NilH \mid \phi_0(s)a = \phi_0(as)\text{ for every $s\in
S$} \}.
\end{align}
Define $\pzt:\A\otimes_S \A \to \A_0\otimes_\Z \A_0$ by
$$\pzt\left(\sum_{w,v\in W_\af} a_{w,v} A_w \otimes A_v\right) = \sum_{w,v\in
W_\af} \phi_0(a_{w,v}) A_w \otimes A_v$$ for $a_{w,v}\in S$. Then
$\B$ is a Hopf algebra with coproduct given by the restriction of
$\pzt \circ \Delta$ to $\B$.

\begin{thm}
[{\cite[Prop. 5.4, Thm. 5.5]{Lam2}}] \label{T:iso} The map $j_0$ is
a Hopf algebra isomorphism $H_*(\Gr_G)\cong \B$. Moreover, for every
$w\in W_\af^0$, $j_0(\xi_w)$ is the unique element of $\B\cap
(A_w+\sum_{u\in W_\af\setminus W_\af^0} \Z A_u)$.
\end{thm}

$\B$ has a basis $\{\PA_w\mid w\in W_\af^0\}$ defined by
\begin{align} \label{E:ncdef}
  \PA_w = j_0(\xi_w)\qquad\text{for $w\in W_\af^0$.}
\end{align}
For $G=SL_n(\C)$ these are the noncommutative $k$-Schur functions of
\cite{Lam2}. The following Lemma is an aid for computing the
elements $\PA_w$.

\begin{lem} \label{L:equivB} Let $a=\sum_{w\in W_\af} c_w A_w\in
\NilCox$ with $c_w\in\Z$. Then $a\in \B$ if and only if
$\sum_{w\gtrdot v} c_w \al_{vw}^\vee\in \Z K$ for all $v\in W_\af$.
\end{lem}
\begin{proof} The following are equivalent:
\begin{enumerate}
\item
Equation \eqref{E:inB} holds for $a$.
\item
$\phi_0(a\la)=0$ for all $\la\in P$.
\item
$\sum_w c_w \sum_{v\lessdot w} \ip{\al_{vw}^\vee,\la} A_v=0$
for all $\la\in P$.
\item
$\sum_{w\gtrdot v} c_w \ip{\al_{vw}^\vee,\la}=0$ for all $v\in
W_\af$ and all $\la \in P$.
\item $\sum_{w\gtrdot v} c_w \al_{vw}^\vee \in \Z K$
for all $v\in W_\af$.
\end{enumerate}
(1) and (2) are easily seen to be equivalent. The equivalence of (2)
and (3) follows from equation \eqref{lem:commute}. (3) and (4) are
equivalent because the $A_v$ form a basis of $\NilCox$. (4) and (5)
are equivalent because $\Z K = \{\mu\in Q_\af^\vee\mid
\ip{\mu,P}=0\}$.
\end{proof}

\section{Schubert polynomials for $H_*(\Gr_{Sp_{2n}(\C)})$ and $H^*(\Gr_{Sp_{2n}(\C)})$}
\label{S:main} In this section we outline the proofs of Theorems
\ref{T:affStan}, \ref{T:main} and \ref{T:Pieri}, relegating two
technical calculations to sections \ref{S:Pieri} and \ref{S:Hopf}.

\subsection{Special generators of Fomin-Stanley subalgebra}
Recall the special elements $\rho_i$ defined in
\eqref{E:specialclass}.  For $1\le i\le 2n$ define
\begin{align} \label{E:PA}
  \PA_i = \PA_{\rho_i}
\end{align}
where $\PA_w$ is defined in \eqref{E:ncdef}.

We now state the explicit expansion of the elements $\PA_i\in\B$
that correspond to homology generators.  Recall the set $\WS$
defined in section \ref{SS:Z}.  Note that $\rho_r$ is the unique
Grassmannian element in $\WS_r$ for $1 \leq r \leq 2n$.

\begin{thm} \label{T:P}
For $1\le r\le 2n$,
\begin{align} \label{E:P}
  \PA_r = \sum_{w\in \WS_r} 2^{c(w)-1} A_w.
\end{align}
\end{thm}

This result is proved in section \ref{S:Pieri}. Some examples for
$\PA_r$ are given in Appendix~\ref{S:PAdata}.

\begin{remark}
It follows from Theorem \ref{T:main} that the elements $\PA_r \in
\B$ generate the affine Fomin-Stanley subalgebra $\B$.
\end{remark}

\subsection{Relations among special generators}

Let $\PC^n$ be the set of partitions $\la$ with $\la_1\le 2n$, which
have at most one part of size $i$ for all $i\le n$. We first note
the following result which is essentially \cite[Lemma 24]{EE}. The
bijection of Lemma \ref{L:EE} was first brought to our attention by
Morse \cite{Mor} who discovered it independently.

\begin{lem}\label{L:EE}
Let $w \in \tC_0^n$.  Then $w$ has a unique length-additive
factorization
$$
w = \rho_{\la_l} \cdots \rho_{\la_2} \rho_{\la_1}
$$
into Grassmannian $Z$-s such that every left factor $\rho_{\la_l}
\cdots \rho_{\la_i}$ is Grassmannian.  Furthermore the map $w
\mapsto \lambda(w)$ is a bijection $\tC_0^n \to \PC^n$ such that
$\ell(w)=|\la(w)|$.
\end{lem}
\begin{proof}
The result follows nearly immediately from \cite[Lemma 24]{EE}.  In
\cite{EE} one associates to $w \in \tC_n$ the window
$$[-w_n,\ldots,-w_1,0,w_1,\ldots,w_n]$$ of an affine
permutation.  This corresponds to the embedding of $\tC_n$ into the
affine symmetric group $\tilde{S}_{2n+2}$.  In \cite{EE} the
parabolic subgroup is generated by $\{s_0, \ldots, s_{n-1}\}$ rather
than by $\{s_1, \ldots, s_n\}$ so we must apply the notational
involution $s_i \leftrightarrow s_{n-i}$ to be compatible with
\cite{EE}.

In any case, for $w \in \tC_n^0$ it is shown in \cite[Lemma 24]{EE}
that the window of $w$ can be successively {\it sorted} to become
the identity. Each sorting operation corresponds to right
multiplication by a factor $\rho_{\la_i}$.  The requirement that
every left factor $\rho_{\la_l} \cdots \rho_{\la_i}$ is Grassmannian
corresponds to asking for the window of $w$ to be completely sorted
at each step. The rest of the statement now follows from \cite{EE}.
\end{proof}

\begin{prop} \label{P:evenP}
The elements $\PA_i \in \B$ satisfy
$$
\PA_{2m} = 2 \left(\PA_1 \PA_{2m-1} -  \PA_2 \PA_{2m-2} + \dotsm +
(-1)^{m-2}
  \PA_{m-1}\PA_{m+1}\right)+(-1)^{m-1}\PA_{m}^2
$$
for $1 \leq m \leq n$.
\end{prop}
\begin{proof}
We use the explicit computation of the $\PA_i$ given in Theorem
\ref{T:P}.  By evaluating the statement of Proposition \ref{P:jco}
at 0, we observe that
\begin{equation}\label{E:PAmult}
\PA_i \; \PA_j = \sum_{w = v\rho_j} 2^{c(v) - 1}\PA_w
\end{equation}
where the summation is over all $w = v\rho_j$ such that (a) $v \in
\WS_i$, (b) $\ell(w) = i + j$, and (c) $w \in \tC_n^0$.  Now any
reduced expression for $v \in \WS_i$ can have at most one occurrence
of $s_0$, so by Lemma \ref{L:EE}, we deduce that the set of $w$ such
that $\PA_w$ can occur in a product of the form $\PA_i \; \PA_j$ has
the form $\rho_a\; \rho_b$ where $a + b = i + j$.

Now fix $1 \leq m \leq n$ and let us compute
$$
S = 2 \left(\PA_1 \PA_{2m-1} -  \PA_2 \PA_{2m-2} + \dotsm +
(-1)^{m-2}
  \PA_{m-1}\PA_{m+1}\right)+(-1)^{m-1}\PA_{m}^2.
$$

First via a direct calculation we note that $\rho_m\rho_m \notin
\tC_n^0$ for $1 \leq m \leq n$. We claim that for $1 \leq j \leq m$
and $w = \rho_i \rho_{2m-i}$ satisfying $w \in \tC_n^0$ and $\ell(w)
= 2m$ we have
$$
[\PA_{w}] \PA_{j} \PA_{2m-j} = \begin{cases} 0 &
\mbox{if $i > j$} \\
1 & \mbox{if $ i = j$} \\
2 & \mbox{if $0 < i < j$} \\
1 & \mbox{if $i = 0$.}
\end{cases}
$$
where $[\PA_{w}]b$ denotes the coefficient of $\PA_w$ in $b \in \B$.
The case $i > j$ follows from Lemma \ref{L:EE}.  The case $i = j$ is
immediate since $c(\rho_i) = c(\rho_{2m-i}) = 1$.  For the case $i <
j$ we must consider $v = \rho_i \rho_{2m-i} \rho_{2m-j}^{-1}$.  We
observe that $$\Supp(v) = [0,i-1] \cup \Supp(\rho_{2m-i}\,
\rho_{2m-j}^{-1})$$ and $i$ is smaller than all elements of
$\Supp(\rho_{2m-i} \rho_{2m-j}^{-1})$ (we use the inequality $2m - j
> i$).  Thus $v$, being a product two
``non-touching'' $Z$-s, is itself a $Z$ and we have $c(v) = 2$. The
final case $i = 0$ is trivial.

Now it follows that the $S = \PA_{2m}$, as required.
\end{proof}

\subsection{Coproduct formula for special generators}
The following result is proved in section \ref{S:Hopf}.
\begin{thm} \label{T:phiP}
For $1 \le r \le 2n$
$$
\pzt(\Delta(\PA_r)) = 1 \otimes \PA_r + \PA_r \otimes 1 + 2\sum_{0 <
s < r}\PA_s \otimes \PA_{r-s}.
$$
\end{thm}

\begin{remark}
An alternative formulation of Theorem \ref{T:phiP} is that the
coefficient of $\xi^{\rho_r}$ in $\xi^{\rho_s}\,\xi^{\rho_{r-s}} \in
H^*(\Gr_{Sp_{2n}(\C)})$ is equal to 2, for $1 \leq s \leq r-1$.
\end{remark}

\subsection{Affine type $C$ Cauchy kernel}\label{S:ker}
Define $\Phi: \Gamma_{(n)} \to H_*(\Gr_{Sp_{2n}(\C)})$ by $P_i
\mapsto \xi_{\rho_i}$ for $1 \leq i \leq 2n$ as in Theorem
\ref{T:main}. By Proposition \ref{P:evenP} and Theorem \ref{T:iso}
this map is well-defined. Define $\Omega_{-1}^\B\in \B \hat \otimes
\Hc^{(n)}$ by taking the image of $\Omega_{-1}^{(n)}$ under the
composition $\Phi_B = j_0 \circ \Phi : \Hc_{(n)} \to \B$:
\begin{equation}
\label{E:Bkern}
\begin{split}
    \Omega_{-1}^\B &= \sum_{\la_1\le 2n} \PA_{\la_1} \PA_{\la_2}\dotsm \otimes
  M_\la[Y] \\
  &= \sum_{\substack{\al \\ \al_i\le 2n}} \PA_{\al_1} \PA_{\al_2}\dotsm
  \otimes 2^{\ell(\al)} y^{\al}
\end{split}
\end{equation}
where $\al$ runs over compositions whose parts have size at most
$2n$. The second equality holds since $\B$ is a commutative ring.
For $w\in \tC_n$, the type $C$ affine Stanley function $\Qf_w$ is
defined by
\begin{equation} \label{E:keraffStan}
  \Omega_{-1}^\B = \sum_{w\in \tC_n} A_w \otimes \Qf_w[Y].
\end{equation}
A straightforward computation shows that this definition agrees with
\eqref{E:affStan}.  Note that \eqref{E:affStan} defines an element
of the ring $\Hc^{(n)}$ via \eqref{E:Ccoiso}. By Theorem \ref{T:iso}
we have
\begin{align} \label{E:kerPF}
  \Omega_{-1}^\B = \sum_{w\in \tC_n^0} \PA_w \otimes \Qf_w[Y]
\end{align}
where $\PA_w$ is defined by \eqref{E:ncdef}.

\subsection{Proof of Theorem \ref{T:Pieri}} Theorem \ref{T:Pieri}
follows immediately from applying the non-equivariant part of
Proposition \ref{P:jco} to Theorem \ref{T:P}.

\subsection{Proof of Theorem \ref{T:main}}
It follows from Proposition \ref{P:evenP}, Theorem \ref{T:phiP} and
Theorem \ref{T:iso} that $\Phi: \Gamma_{(n)} \to
H_*(\Gr_{Sp_{2n}(\C)})$ is a bialgebra morphism.  Since both
$\Gamma_{(n)}$ and $H_*(\Gr_{Sp_{2n}(\C)})$ are graded commutative
and cocommutative Hopf algebras, $\Phi$ must in addition be a Hopf
morphism.  Recall that we define $\Psi: H^*(\Gr_{Sp_{2n}(\C)})\to
\Hc^{(n)}$ by the linear map $\xi^w \mapsto \Qf_w$ for $w \in
\tC_n^0$.

We first show that $\Psi:H^*(\Gr_{Sp_{2n}(\C)})\to \Hc^{(n)}$ and
$\Phi: \Hc_{(n)} \to H_*(\Gr_{Sp_{2n}(\C)})$ are dual with respect
to the pairing $\ip{.,.}: H_*(\Gr_{Sp_{2n}(\C)}) \times
H^*(\Gr_{Sp_{2n}(\C)}) \to \Z$ induced by the cap product and the
pairing $\ipp{.,.}:\Hc_{(n)} \times \Hc^{(n)} \to \Z$ of section
\ref{sec:basishc}.  It suffices to show that for each $w \in
\tC_n^0$ we have $\ip{\Phi(f),\xi^w} = \ipp{f,\Psi(\xi^w)}$ as $f$
varies over the spanning set $\{P_{\lambda_1}\cdots P_{\lambda_l}
\mid \lambda_1 \leq 2n\}$ of $\Hc_{(n)}$. Identifying $\B$ with
$H_*(\Gr_{Sp_{2n}(\C)})$ via the map $j_0$ of Theorem~\ref{T:iso},
we calculate
\begin{align*}
\ipp{P_{\lambda_1}\cdots P_{\lambda_l},\Psi(\xi^w)} & = \ipp{P_{\lambda_1}\cdots P_{\lambda_l},\Qf_w} \\
& = \ipp{P_{\lambda_1}\cdots
P_{\lambda_l},\ip{\Omega_{-1}^\B,\xi^w}}
\\
& = \ip{\ipp{P_{\lambda_1}\cdots
P_{\lambda_l},\Omega_{-1}^\B},\xi^w}
\\
& = \ip{\PA_{\lambda_1}\cdots \PA_{\lambda_l},\xi^w}
\\
& = \ip{\Phi_\B(P_{\lambda_1}\cdots P_{\lambda_l}),\xi^w}.
\end{align*}
The second equality holds by \eqref{E:kerPF}. The fourth holds by
\eqref{E:Bkern} and \eqref{E:Mcoeff}. The other equalities hold by
definition.

Since $\Phi$ is a Hopf-morphism, we deduce that $\Psi$ is also a
Hopf-morphism.  It only remains to prove that $\Psi$ is a bijection.
 For surjectivity, since the $Q_r$ generate $\Hc^{(n)}$ as an
algebra, it suffices to show that $\Qf_{c_r}=Q_r$ in $\Hc_{(n)}$,
where $c_r\in \tC_n^0$ is the length $r$ element of the form
$$
c_r = \cdots s_1 \; s_0 \; s_1 \; \cdots \;s_{n-1} \; s_n \;s_{n-1}
\;\cdots \; s_2\;s_1\; s_0.
$$
It is easy to see that $c_r$ has a unique reduced word.  So a
length-additive factorization of $c_r$ into a product $c_r = \prod_i
v^i$ with each $v^i\in\WS$, is equivalent to a composition
$(\alpha_1,\alpha_2,\ldots,\alpha_s)$ of $r$ into parts of size less
than $2n$, where each $v^i$ is either the identity or has one
component. After multiplying by $2^{t}$ where $t = \#\{i \mid
\alpha_i > 0\}$, we see that $\Qf_{c_r}$ is the generating function
of shifted tableaux $T$ whose shape is a single row of length $r$
where no letter can be used more than $2n$ times. The tableau $T$ is
obtained from the composition $\alpha$ by setting $\alpha_i$ letters
equal to $i$. The factor $2^{t}$ comes from the two possible choices
of marking for the leftmost occurrence of each letter. This matches
$\Qf_{c_r}$ to the combinatorial definition of $Q_r$ using tableaux
\cite[III.8.16]{Mac}.

For injectivity, it suffices to show that $\{\Qf_w \mid w \in
\tC^0_n\}$ is linearly independent.  We shall establish the
triangularity property
$$
\Qf_w = \sum_{\mu \le \la(w)} a_{\mu,w} M_\mu
$$
where $w \mapsto \la(w)$ is the bijection between $\tC_n^0$ and
$\PC^n$ of Lemma \ref{L:EE}, and $\le$ is the lexicographic order on
partitions. Furthermore $a_{\mu,w}$ is unitriangular.

We first observe that if $w \in \tC^0_n$ and $w = v^{s} \; \cdots \;
v^{1}$ is a factorization into $Z$'s then $v^{1}$ must be
Grassmannian, so it is one of the $\rho_r$'s for $r \in [1,2n]$. But
if $w = \rho_{\la_l} \cdots \rho_{\la_2} \rho_{\la_1}$ where
$\rho_{\la_l} \cdots \rho_{\la_2}$ is Grassmannian then
$w\,(\rho_r)^{-1}$ cannot be length subtractive for $2n \geq r >
\la_1$.  This is because every reduced expression for $\rho_{\la_l}
\cdots \rho_{\la_2}$ ends in $s_0$.  Repeating this, we see that the
matrix of coefficients $a_{\mu,w}$ is triangular with respect to the
lexicographic order.  We are using the fact that if $a_{\mu,w} \neq
0$ then the factorization $w = v^{s} \; \cdots \; v^{1}$ can be
chosen so that $\ell(v^{i}) = \mu_i$.

Finally, the factorization $w = \rho_{\la_l} \cdots \rho_{\la_2}
\rho_{\la_1}$ of Lemma \ref{L:EE} shows that $a_{\la(w),w} = 1$
since $c(\rho_i) = 1$.

\subsection{Proof of Theorem \ref{T:affStan}}
The fact that $\Qf_w$ is symmetric and defines an element of
$\Hc^{(n)}$ follows from the definition \eqref{E:keraffStan} via the
affine type $C$ Cauchy kernel.  The statement that $\{\Qf_w \mid w
\in \tC_n^0\}$ forms a basis follows from Theorem \ref{T:main} and
the fact that $\{\xi^w \mid w \in \tC_n^0\}$ is a basis for
$H^*(\Gr_{Sp_{2n}(\C)})$.  The positivity of the product structure
constants is a general theorem due to Graham \cite{Gra} and Kumar
\cite{Kum}.

The coproduct structure constants of $\{\Qf_w \mid w \in \tC_n^0\}$
are the same as those of $\{\xi^w \mid w \in \tC_n^0\}$. By the
duality of $H_*(\Gr_{Sp_{2n}(\C)})$ and $H^*(\Gr_{Sp_{2n}(\C)})$ and
their Schubert bases, the above constants are the same as the
product structure constants for the homology classes $\{\xi_w \mid w
\in \tC_n^0\}$. Using the nonequivariant case $\ell(y)=\ell(x)$ of
\eqref{E:HomTstructureconst}, these constants are given by the
coefficients $j_x^y$ of \eqref{E:jdef}. But these are known to be
nonnegative from the work of Peterson \cite{P} and Lam and Shimozono
\cite{LS}; they are equal to certain three-point genus zero
Gromov-Witten invariants of the (finite) flag variety.

For the final positivity statement we claim that
\begin{equation}\label{E:Stacoeff}
\text{the coefficient of $\Qf_v$ where $v \in \tC_n^0$ in $\Qf_w$ is
equal to $j_v^w$} \end{equation} that is, the coefficient of $A_w$
in $\PA_v$. But this follows from expanding \eqref{E:kerPF} using
the definition of $\PA_w$.

\section{The combinatorics of Zee-s}
\label{S:Pieri}

It is obvious that $\WS_r$ contains a unique Grassmannian element,
namely, $\rho_r$, and that $c(\rho_r)=1$.  To prove Theorem
\ref{T:P}, by Theorem \ref{T:iso} it remains to show that the right
hand side of \eqref{E:P} is an element of $\B$. By Lemma
\ref{L:equivB} and Example \ref{X:deltaK} it suffices to prove the
following result, whose proof occupies the rest of this section.

\begin{prop} \label{p:2}
For any $v\in \WS$ with $\ell(v)<2n$, let $\CC_v=\{w\in\WS \mid
w\gtrdot v\}$. Then
\begin{align}\label{E:corootsum}
  \sum_{w\in \CC_v} 2^{c(w)-1} \al_{vw}^\vee = 2^{c(v)} K.
\end{align}
\end{prop}

\begin{example} Let $n=3$ and $v=s_0s_2s_3s_2\in \WS$. Every
$w\in\CC_v$ is obtained by putting a $1$ into some reduced word for
$v$. For each $w\in \CC_v$, a reduced word and the coroot
$\al_{vw}^\vee$ is given below. They may be computed as in Example
\ref{X:covercoroot}.
\begin{align*}
\begin{array}{|c||c|}\hline
\text{red. word} & \al_{vw}^\vee \\ \hline
10232 & 2\al_0^\vee+\al_1^\vee+2\al_2^\vee+2\al_3^\vee \\
01232 & \al_1^\vee+2\al_2^\vee+2\al_3^\vee \\
23210 & 2\al_0^\vee+\al_1^\vee \\
23201 & \al_1^\vee \\ \hline
\end{array}
\end{align*}
The sum of these coroots is $4K$, which agrees with the fact that
$\Supp(v)$ has two components, $\{0\}$ and $\{2,3\}$.
\end{example}

Let $w \in \WS$.  Since $s_is_j=s_js_i$ for $i$ and $j$ in different
components of $\Supp(w)$, there exists a factorization
$w=w_{I_1}\dotsm w_{I_c}$ where $I_1,I_2,\dotsc,I_c$ are the
components of $\Supp(w)$ and $\Supp(w_{I_p})=I_p$. Let us index the
components so their elements are ordered consistently with the total
order on $I_\af$. Then the above factorization is unique. For a
component $C=I_p$ of $\Supp(w)$ define $w_C=w_{I_p}$, which is
called the $C$-component of $w$.

\begin{example} Let $n=9$ and $u=4689852102$; we have
$u\in\Red(w)$ for some $w\in\WS$. We have $I_1=\{0,1,2\}$,
$I_2=\{4,5,6\}$, and $I_3=\{8,9\}$, and $w_{I_1}=s_2s_1s_0s_2$,
$w_{I_2}=s_4s_6s_5$, and $w_{I_3}=s_8s_9s_8$.
\end{example}

\subsection{Bruhat covers in $\WS$}
\label{SS:BruhatZ}

To prove Proposition \ref{p:2} we study in detail the Bruhat order
of $\tC_n$ when restricted to the subset $\WS$.  The results in
these subsections may be of independent combinatorial interest.

We construct the set of covers $\CC_v$ in $\WS$, of a fixed element
$v\in \WS$. For $k',k\in I_\af$ with $k'<k$ let
\begin{align*}
  N_{k,k'} &= k(k+1)\dotsm(n-1)n(n-1)\dotsm101\dotsm (k'-1)k' \\
\RN_{k',k}&=k'(k'-1)\dotsm101\dotsm(n-1)n(n-1)\dotsm (k+1) k.
\end{align*}
For $w\in\WS$, we define
\begin{align*}
  \Red^N(w) &= \{u\in \Red(w)\mid u\sw N_{k,k-1} \text{ for some
  $1\le k \le n$}\} \\
  \Red^\RN(w) &= \{u\in \Red(w)\mid u\sw \RN_{k-1,k} \text{ for some
  $1\le k \le n$}\}
\end{align*}
where $u\sw u'$ denotes a specific embedding of a word $u$ as a
subword of a word $u'$. Then by definition $w\in\WS$ if and only if
$\Red^N(w)\cup \Red^\RN(w)\ne\vn$.

Therefore $w\in\CC_v$ if and only if either (1) there is a word
$u\in\Red^N(v)$ with an embedding of the form $u\sw N_{k,k-1}$ and a
letter $j\sw N_{k,k-1}$ that is missing from $u$, such that the word
$\tu$ obtained by inserting $j$ into $u$, is a reduced word of $w$,
or (2) there is a $u\in\Red^\RN(v)$ with an embedding of the form
$u\sw\RN_{k-1,k}$ and a letter $j\in\RN_{k-1,k}$ missing from $u$,
such that inserting $j$ into $u$ yields $\tu\in\Red(w)$.

\begin{lem} \label{L:addable} Let $v\in\WS$ and $u\in\Red^N(v)$ with
$u\sw N_{k,k-1}$ (resp. $u\in\Red^\RN(v)$ with $u\sw \RN_{k-1,k}$).
Let $j\sw N_{k,k-1}$ (resp. $j\sw \RN_{k-1,k}$) be a letter that is
not in $u$. Then adding this copy of $j$ to $u$, produces a word in
$\Red^Z(w)$ for some $w\in\CC_v$, if and only if (1) $j\not\in
\Supp(u)$ or (2) $j+1\in \Supp(u)$ for $j\ge k$ or $j-1\in\Supp(u)$
for $j\le k-1$.
\end{lem}
\begin{proof} This follows directly from the Coxeter relations for
$\tC_n$.
\end{proof}

We define the reduced words
\begin{align*}
 V^{k,k'} &= k(k-1)\dotsm101\dotsm (k'-1)k' &\qquad&\text{for $k,k'<n$} \\
 \La_{k,k'} &= k(k+1)\dotsm(n-1)n(n-1)\dotsm(k'+1)k' &\qquad&\text{for $k,k'>0$} \\
 \Inc{k'}{k} &= k'(k'+1)\dotsm (k-1)k && \text{for $k'\le k$} \\
 \Dec{k}{k'} &= k(k-1)\dotsm(k'+1)k' &&\text{for $k'\le k$}
\end{align*}
A word is an N if it is a subword of $N_{k,k-1}$ for some $1\le k\le
n$ and a reverse N (abbreviated by the symbol $\RN$) if it is a
subword of $\RN_{k-1,k}$ for some $1\le k\le n$. The name N is
suggested by the definition: the values in such a word go up, then
down, and then up, like the letter N. A word $v$ is a Z if it is an
N or a $\RN$. For $w\in\tC_n$ let $\Red^Z(w)$ be the set of reduced
words for $w$ that are Zs. Then by definition, $w \in \tC_n$ is a
$Z$ if and only if $\Red^Z(w)\ne\varnothing$. Let $\Red^N(w)$ (resp.
$\Red^{\RN}(w)$) be the subset of reduced words of $w$ that are Ns
or (resp. $\RN$s).

A \textit{saturated} N (resp. $\RN$) is a word of the form
$N_{k,k'}$ (resp. $\RN_{k',k}$). An N or $\RN$ is \textit{proper} if
it contains both the letters $0$ and $n$. We emphasize the important
fact that if $u$ is a proper N, then $\first(u)>\last(u)$, where
$\first(u)$ and $\last(u)$ are the first and last letters of $u$
respectively. Similarly if $u$ is a proper $\RN$ then
$\first(u)<\last(u)$.

Let $u=i_1i_2\dotsm i_M$ be a word with letters in $I_\af$. We say
that $u$ has a \textit{peak} at $p$ if $1<p<M$ and
$i_{p-1}<i_p>i_{p+1}$ or if $p=1$ and $i_1>i_2$ or if $p=M$ and
$p_{M-1}<p_M$ or if $M=1$. We say that $u$ has a \textit{valley} at
$p$ if $1<p<M$ and $i_{p-1}>i_p<i_{p+1}$ or if $p=1$ and $i_1<i_2$
or if $p=M$ and $p_{M-1}>p_M$ or if $M=1$.

We say that a word is a \textit{V} (resp. $\La$) if it is either
empty or has exactly one valley (resp. peak). Note that only the
empty word has no valleys (resp. peaks). Note that Vs and $\La$s are
both Ns and $\RN$s. Write $\Red^V(w)$ and $\Red^\La(w)$ for the sets
of reduced words of $w$ that are respectively Vs and $\La$s. A
saturated V (resp. $\La$) is one of the form $V^{k,k'}$ (resp.
$\La_{k,k'}$).

\begin{example} Let $n=4$. Then $234101$ is a proper $N$,
$20143$ is a proper $\RN$, $312$ is a V, and $24321$ is a $\La$.
\end{example}

\subsection{Equivalences for reduced words and rotation}
The following Lemma is essentially a special case of Edelman-Greene
insertion \cite{EG}. It says that a $\Lambda$ with no $(n-1)n(n-1)$
is equivalent to a V. Similarly a V with no $101$ is equivalent to a
$\Lambda$.

\begin{lem} \label{L:hook}
Suppose $i_1i_2\dotsm i_p j_1 j_2\dotsm j_q\in \Red(w)$ for some
$w\in\WS$ such that $i_1<i_2<\dotsm<i_p<j_1>j_2>\dotsm>j_q$ and
$(n-1)n(n-1)$ is not a subword. Then there is a $k_1k_2\dotsm k_q
l_1l_2\dotsm l_p\in\Red(w)$ such that $i_1$ occurs in $k_1k_2\dotsm
k_q$, $k_1>k_2>\dotsm>k_q<l_1<l_2<\dotsm<l_p$ and $k_s\le j_s$ for
$1\le s\le q$ and $i_s<l_s$ for $1\le s\le p$.
\end{lem}

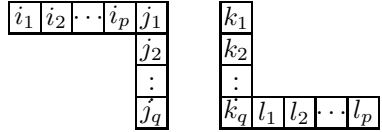
$$
\tableau[sbY]{i_1,i_2,\dotsm,i_p,j_1| %
\bl,\bl,\bl,\bl,j_2|%
\bl,\bl,\bl,\bl,\vdots| %
\bl,\bl,\bl,\bl,j_q} \qquad %
\tableau[sbY]{k_1|k_2|\vdots|k_q,l_1,l_2,\dotsc,l_p}
$$

\begin{proof} The result is trivial if $p=0$ or $q=0$.
Suppose $p=1$. If both $i_1+1$ and $i_1$ occur in $j_1j_2\dotsm j_q$
then $i_1>0$ and $i_1 j_1\dotsm j_q \equiv j_1\dotsm j_q (i_1+1)$
using the braid relations. We take $k_s=j_s$ for $1\le s\le q$ and
$l_1=i_1+1$, which satisfies $l_1>i_1 \ge j_q$. Otherwise let $r$ be
maximal such that $i_1<j_r$. It cannot be the case that
$j_{r+1}=i_1$ for then $i_1j_1\dotsm j_q$ is not reduced. We have
$i_1j_1\dotsm j_q\equiv (j_1\dotsm j_{r-1} i_1 j_{r+1}\dotsm j_q)
j_r$ and the latter word has the desired form. Note that in the case
$p=1$, $i_1$ occurs in $k_1\dotsm k_q$. Finally suppose $p>1$. By
induction $i_2\dotsm i_p j_1\dotsm j_q\equiv k_1'\dotsm k_q'
l_2\dotsm l_p$ with $k_1'>\dotsm>k_q'<l_2<\dotsm<l_p$ with $j_s\le
k_s'$ for $1\le s\le q$ and $i_s>l_s$ for $2\le s\le p$. Since
$i_1<i_2$ and $i_2$ occurs in $k_1'\dotsm k_q'$, we may apply the
$p=1$ case and obtain $i_1 k_1'\dotsm k_q' \equiv k_1\dotsm k_q l_1$
with $k_1>\dotsm>k_q<l_1$ and $k_s\le k_s'$ for $1\le s\le q$. Since
$i_2$ was in $k_1'\dotsm k_q'$ and $i_1<i_2$, it follows by
considering the $p=1$ case that $l_1\le i_2<l_2$. It follows that
$k_1\dotsm k_q l_1\dotsm l_p$ is the desired reduced word.
\end{proof}

\begin{lem} \label{L:Vcombine}
Suppose $u$ and $u'$ are two Vs (resp. $\Lambda$s) such that all
letters of $u$ are greater than those in $u'$. Then $uu'$ and $u'u$
are both equivalent to a V (resp. $\Lambda$).
\end{lem}
\begin{proof}
Let $u$ and $u'$ be Vs with $u=u_1mu_2$ where $m$ is the valley of
$u$. Then $u_1 m u' u_2$ and $u_1 u' m u_2$ are Vs that are
equivalent to $uu'$ and $u'u$ respectively. The proof for $\Lambda$s
is similar.
\end{proof}

Given a word $u$, let $u^+$ (resp. $u^-$) be the word obtained by
adding (resp. subtracting) one from each letter in $u$.

\begin{lem} \label{L:LaVunique}
Let $w\in\WS$ and $J=\Supp(w)$.
\begin{enumerate}
\item Suppose $n\not\in J$. Then $\Red^V(w)\ne\vn$. Moreover if $J$
is an interval then $\Red^V(w)$ is a singleton.
\item Suppose $0\not\in J$. Then $\Red^\La(w)\ne\vn$.
Moreover if $J$ is an interval then $\Red^\La(w)$ is a singleton.
\item Suppose $J$ is an interval $[m,M]$ with $0<m\le M<n$. Let $u_1
M u_2\in \Red^\La(w)$ and $u_2' m u_1'\in\Red^V(w)$. Then
$u_2'=u_2^+$ and $u_1'=u_1^+$.
\end{enumerate}
\end{lem}
\begin{proof} We shall prove (1) as (2) is similar.
Let $u\in \Red^Z(w)$. Suppose that $n\not\in J$ and that $u$ is an
$N$; the case of a $\RN$ is similar. Say $u$ is embedded in
$N_{k,k-1}$. Then $u=u_1u_2$ where $u_1$ is a $\La$ such that
$\Supp(u_1)\subset[k,n-1]$ and $u_2$ is a V with
$\Supp(u_2)\subset[0,k-1]$. Then $\Red^V(w)\ne\vn$ by Lemmata
\ref{L:hook} and \ref{L:Vcombine}.

Let $u\in\Red^V(w)$ with $J$ an interval. We prove its uniqueness by
induction on $\ell(w)$. For $\ell(w)\le 3$ this is evident from
\eqref{E:AffWeylRelations}. Let $M=\max(J)<n$. Suppose first that
$u$ contains a single $M$. Then $u$ has the form $u = M \uh$ or $u =
\uh M$. We assume the former as the latter has an analogous proof.
We have $\uh \in\Red^V(s_M w)$ and $s_Mw\in\WS$. By induction $\uh$
is unique. Now let $u'\in\Red^V(w)$. Since $\Red(w)$ is connected by
the braid relations \eqref{E:AffWeylRelations}, every reduced word
for $w$ (and in particular $u'$) has a single $M$ which precedes
every $M-1$. Since $u'$ is a V it must start with $M$. Therefore
$u'=M\uh=u$ by the uniqueness of $\uh$.

Otherwise $u$ must have the form $u=M \uh M$. Let $u'\in \Red^V(w)$.
Clearly $u'$ must contain an $M$ which must be at the beginning or
end. We suppose $u'$ has the form $u'=Mu''$ as the case $u'=u''M$ is
similar. By induction $\Red^V(s_M w)$ is a singleton. Therefore
$u''=\uh M$ and $u'=u$ as desired.

(3) is proved by induction on the length of $u_1 M u_2$. If either
$u_1$ or $u_2$ is empty then the result certainly holds. Write
$u_1=u_3x$ and $u_2=yu_4$ where $x$ and $y$ are letters. Since
$\Supp(u_1Mu_2)=[m,M]$, $x=M-1$ or $y=M-1$. Suppose $x=y=M-1$. By
induction we have $u_1 M u_2 \equiv u_3 (M-1)M(M-1) u_4\equiv u_3
M(M-1)Mu_4\equiv M u_3 (M-1) u_4 M\equiv M u_4^+ m u_3^+ M = u_2^+ m
u_1^+$. Suppose next that $x=M-1>y$. Then again by induction we have
$u_1 M u_2 = u_3 (M-1) M u_2 \equiv u_3 (M-1) u_2 M \equiv u_2^+ m
u_3^+ M = u_2^+ m u_1^+$. The case $y=M-1>x$ is similar.
\end{proof}

\begin{example} For $n>7$ the N $676545$ is equivalent to a V:
$676545\equiv 767545 \equiv 765457$.
\end{example}

Suppose $u\sw N_{k,k-1}$ is a subword and $\ell\sw N_{k,k-1}$ is a
subletter (resp. $u \sw \RN_{k-1,k}$ is a subword and $\ell\sw
\RN_{k-1,k}$ is a subletter) with $\ell$ missing from $u$. We give
an explicit way to obtain another embedded word $u'\in\Red^Z(v)$
such that $\ell$ is at the beginning or end of the ambient $N$ or
$\RN$. We call this process \textit{rotation}. The only cases not
treated in Lemma \ref{L:rotate} are $\ell=0$ or $\ell=n$, in which
case we may use Lemma \ref{L:LaVunique} to obtain an equivalent
reduced word that is a $\La$ or V respectively, and these can be
embedded into an N or $\RN$ with the missing letter at the beginning
or end.

\begin{lem} \label{L:rotate}
Suppose $w\in\WS$ and $u\in\Red^N(w)$ (resp. $u\in\Red^\RN(w)$) with
$u\sw N_{k,k-1}$ (resp. $u\sw \RN_{k-1,k}$).
\begin{enumerate}
\item
If there is an $\ell$ such that $k\le \ell < n$ and $\ell$ does not
appear in the part of $u$ that is embedded in $\Inc{k}{n}\sw
N_{k,k-1}$ (resp. $\Dec{n}{k}\sw \RN_{k-1,k}$), then there is a
$u'\in\Red^N(w)$ (resp. $u'\in \Red^\RN(w)$) such that $u'\sw
N_{\ell+1,\ell}$ (resp. $u'\sw \RN_{\ell,\ell+1}$).
\item
If there is an $\ell$ such that $0 < \ell \le k-1$ and $\ell$ does
not appear in the part of $u$ that is embedded in $\Inc{0}{k-1}\sw
N_{k,k-1}$ (resp. $\Dec{k-1}{0}\sw \RN_{k-1,k}$), then there is a
$u'\in\Red^N(w)$ (resp. $u'\in\Red^\RN(w)$) such that $u'\sw
N_{\ell,\ell-1}$ (resp. $u'\sw \RN_{\ell-1,\ell}$).
\item
If there is an $\ell$ such that $0<\ell<n$ and $\ell$ does not
appear in the part of $u$ that is embedded in $\Dec{n}{0}\sw
N_{k,k-1}$ (resp. $\Inc{0}{n}\sw \RN_{k-1,k}$) then there is a
$u'\in\Red^\RN(w)$ (resp. $u'\in\Red^N(w)$) such that $u'\sw
\RN_{\ell-1,\ell}$ (resp. $u'\sw N_{\ell,\ell-1}$). Moreover $k-1$
or $k$ is missing in the increasing (resp. decreasing) part of $u'$,
according as $\ell<k$ or $\ell \ge k$.
\end{enumerate}
\end{lem}
\begin{proof}
We prove (1) for $u\in\Red^N(w)$; the other cases of (1) and (2) are
similar. Let $u=u_1u_2u_3u_4$ where $u_1\sw \Inc{k}{\ell-1}$,
$u_2\sw \La_{\ell+1,\ell+1}$, $u_3\sw \Dec{\ell}{k}$, and $u_4\sw
V^{k-1,k-1}$. We have $u\equiv u_2u_1u_3u_4$. $u_1u_3$ is reduced
since it is a factor of a reduced word. Since $u_1u_3$ is a $\La$
with no $n$, by Lemma \ref{L:LaVunique} it is equivalent to a V:
$u_1u_3\equiv u_3'u_1'$ where $u_3'u_1'$ is a V with valley
$\last(u_3')$ such that $u_3'\sw \Dec{\ell}{k}$ and $u_1' = u_1^+\sw
\Inc{k+1}{\ell}$. Then $u\equiv u_2u_1u_3u_4\equiv
u_2u_3'u_1'u_4\equiv u_2 u_3'u_4u_1'\sw N_{\ell+1,\ell}$.

We prove (3) for $u\in\Red^N(w)$ and $\ell < k$; the cases that
$\ell\ge k$ and $u\in\Red^\RN(w)$, are similar. Let $u=u_1u_2u_3u_4$
where $u_1\sw \La_{k,k}$, $u_2\sw \Dec{k-1}{\ell+1}$, $u_3\sw
V^{\ell-1,\ell-1}$, and $u_4\sw \Inc{\ell}{k-1}$. We have $u\equiv
u_1u_3u_2u_4$. $u_2u_4$ is a reduced word supported on $[\ell,k-1]$
that is a V. By Lemma \ref{L:LaVunique} there is an equivalent
$\La$: $u_2u_4\equiv u_4'u_2'$ where $u_4'u_2'$ is a $\La$ with peak
$\first(u_2')$ such that $u_4'\sw \Inc{\ell}{k-2}$ and $u_2'\sw
\Dec{k-1}{\ell}$. Then $u\equiv u_1u_3u_2u_4\equiv u_1u_3
u_4'u_2'\equiv u_3u_4'u_1u_2'\sw \RN_{\ell-1,\ell}$.
\end{proof}

\begin{example}
Let $n=6$. We start with a reduced word for an element of $w\in\WS$
and apply rotations, choosing $\ell$ to be the first break from the
left, indicated by the symbol $\bullet$, in the given reduced word.
$$
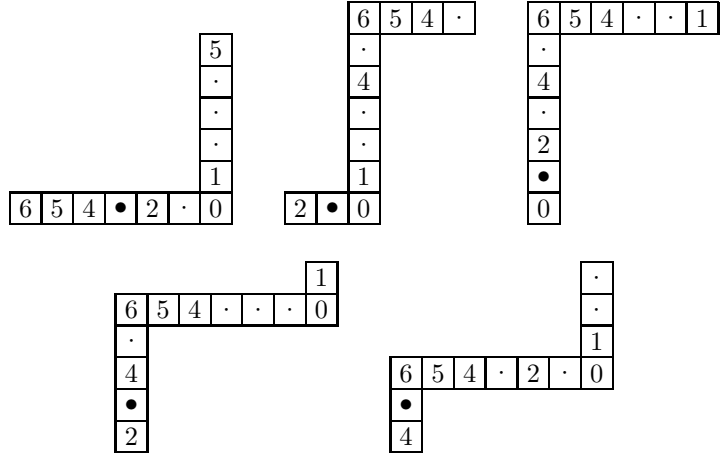
\tableau[sbY]{\bl,\bl,\bl,\bl,\bl,\bl,5|
\bl,\bl,\bl,\bl,\bl,\bl,\cdot| \bl,\bl,\bl,\bl,\bl,\bl,\cdot|
\bl,\bl,\bl,\bl,\bl,\bl,\cdot| \bl,\bl,\bl,\bl,\bl,\bl,1|
              6,5,4,\bullet,2,\cdot,0|}
\qquad
\tableau[sbY]{\bl,\bl,6,5,4,\cdot| \bl,\bl,\cdot|
\bl,\bl,4|\bl,\bl,\cdot|\bl,\bl,\cdot|\bl,\bl,1|2,\bullet,0|} \qquad
\tableau[sbY]{6,5,4,\cdot,\cdot,1| \cdot|4|\cdot|2|\bullet|0|}
$$
$$
\tableau[sbY]{\bl,\bl,\bl,\bl,\bl,\bl,1|6,5,4,\cdot,\cdot,\cdot,0|
\cdot|4|\bullet|2|}\qquad
\tableau[sbY]{\bl,\bl,\bl,\bl,\bl,\bl,\cdot|\bl,\bl,\bl,\bl,\bl,\bl,\cdot|\bl,\bl,\bl,\bl,\bl,\bl,1|
              6,5,4,\cdot,2,\cdot,0|
              \bullet|
              4}
$$
Rotating the last word yields the first one. There are two other
words in $\Red^Z(w)$, which are obtained from the first and third
words above, by commutations.
$$
\tableau[sbY]{\bl,\bl,\bl,\bl,\bl,\bl,5| \bl,\bl,\bl,\bl,\bl,\bl,4|
\bl,\bl,\bl,\bl,\bl,\bl,\cdot| \bl,\bl,\bl,\bl,\bl,\bl,\cdot|
\bl,\bl,\bl,\bl,\bl,\bl,1|
              6,5,\cdot,\cdot,2,\cdot,0|}
\qquad \tableau[sbY]{6,5,4,\cdot,2,1| \cdot|4|\cdot|\cdot|\cdot|0|}
$$
\end{example}

\subsection{Normal words}

The set $\WS$ has a partition into three subsets: the elements $w$
with $\Red^\RN(w)=\vn$, those with $\Red^N(w)=\vn$, and those with
both $\Red^N(w)\ne\vn$ and $\Red^\RN(w)\ne\vn$. We give a criterion
for membership in these subsets.

\begin{lem} \label{L:hasN} Let $w\in\WS$.
\begin{enumerate}
\item
$\Red^\RN(w)=\vn$ (resp. $\Red^N(w)=\vn$) if and only if some word
in $\Red^N(w)$ (resp. $\Red^\RN(w)$) contains $\Dec{n}{0}$ (resp.
$\Inc{0}{n}$) as a factor, if and only if every word in $\Red^N(w)$
(resp. $\Red^\RN(w)$) does.
\item There is a $u\in\Red^N(w)$ that does not contain $\Dec{n}{0}$
as a factor, if and only if there is a $u'\in \Red^\RN(w)$ that does
not contain $\Inc{0}{n}$ as a factor.
\end{enumerate}
\end{lem}
\begin{proof} (2) follows from Lemma \ref{L:rotate}(3).

For (1) we observe that the property of having $\Dec{n}{0}$ as a
subword, is invariant under the braid relations, which connect
$\Red(w)$.

Suppose $\Red^N(w)$ contains a word with factor $\Dec{n}{0}$. In
particular it contains $\Dec{n}{0}$ as a subword. Therefore the same
is true for all of $\Red(w)$. Now every N that contains $\Dec{n}{0}$
as a subword must contain it as a factor. This proves the second
equivalence in (1). Moreover no $\RN$ contains $\Dec{n}{0}$ as a
subword, so $\Red^\RN(w)=\vn$. Conversely, suppose
$\Red^\RN(w)=\vn$. Then $\vn\ne \Red^Z(w)=\Red^N(w)$. Let
$u\in\Red^N(w)$. Then $u$ must contain $\Dec{n}{0}$ as a factor, for
otherwise (2) yields a contradiction.
\end{proof}

Let $w \in \WS$ satisfy $\Supp(w) = I_\af$.  A \textit{normal} word
for $w\in\WS$ is an element $u\in \Red^Z(w)$ such that:
\begin{enumerate}
\item If $\Red^N(w)\ne\vn$ 
then $u$ has the form $u=\Inc{k}{n}\dotsm$.
\item If $\Red^N(w)=\vn$, then $u\in\Red^\RN(w)$ has the form $u=
\Dec{k}{0}\dotsm$.
\end{enumerate}

\begin{lem} \label{L:tri} Let $v\in\WS$ be such that $\Supp(v)=I_\af$.
Then $v$ has a unique normal word, denoted $v_\nor$.
\end{lem}
\begin{proof} Existence holds by Lemma \ref{L:rotate}.
Suppose that $\Red^N(v)\ne\vn$.  The case $\Red^N(v) = \vn$ is
analogous. Let $u=\Inc{k}{n}u_1$ and $u'=\Inc{k'}{n}u_1'$ be normal
words for $v$.


Suppose first that $k'<k$. We cannot have $k=n$, because the form of
$u$ implies that $s_n v < v$ while that of $u'$ implies $s_nv>v$. So
$k<n$. We have $v'=s_kv<v$. By the Exchange Property there is a
letter in $u'$ whose removal gives a reduced word $u''$ for $v'$.
Since $k \Inc{k'}{n}$ is a reduced word, the removed letter does not
occur in $\Inc{k'}{n}$.  In particular $k \in \Supp(v')$.  But
$\Supp(\Inc{k+1}{n}) \supset (I_\af \setminus \{k\})$ so $\Supp(v')
= \Supp(v) = I_\af$.  By induction on length, $u''=\Inc{k+1}{n}u_1$,
which is a contradiction.  Similarly $k<k'$ leads to a
contradiction. Therefore $k=k'$. But then $u_1$ and $u_1'$ are
reduced words and Vs for the same element of $\WS$, so by Lemma
\ref{L:LaVunique}, $u_1=u_1'$ and therefore $u=u'$.
\end{proof}

\subsection{Special words}
In this section we assume that $v\in\WS$ is such that
$\Supp(v)=I_\af$. By Lemma \ref{L:tri}, $v$ has a unique normal word
$v_\nor$. We say that an embedded subword $u\sw N_{k,k-1}$ (resp.
$u\sw \RN_{k-1,k}$) is \textit{normally embedded} if $u=v_\nor$ for
some $v\in \WS$ with $\Supp(v)=I_\af$ and $u=\Inc{k}{n}\dotsm\sw
N_{k,k-1}$ (resp. $u=\Dec{k-1}{1}\Inc{0}{n}\dotsm\sw \RN_{k-1,k}$).

Suppose $u\sw u'$ and $u\ne u'$. Define $\fgap(u\sw u')$ (resp.
$\lgap(u\sw u')$) to be the first (resp. last) letter $j\sw u'$ that
is not in $u$.

We say that $u\in\Red^Z(v)$ is \textit{special} if it has a
\textit{special embedding}, that is, an embedding of the form $u\sw
u'$ where $u'=N_{a,a-1}$ or $u'=\RN_{a-1,a}$ for some $1\le a\le n$
such that, if $j=\fgap(u\sw u')$, then adding $j$ to $u$ produces
the normally embedded word $w_\nor\sw u'$ for some $w\in\CC_v$. More
specifically, one of the following holds:
\begin{enumerate}
\item $u\in\Red^N(v)$ and $u\sw N_{a,a-1}$ for some $1\le a \le n$
such that $u$ contains all but exactly one of the letters in
$\Inc{a}{n}\subset N_{a,a-1}$, or
\item $u\in\Red^\RN(v)$ with $u\sw \RN_{a-1,a}$
for some $1\le a\le n$ and $u$ contains all but exactly one of the
letters in $\Dec{a-1}{1}\Inc{0}{n}\sw \RN_{a-1,a}$.
\end{enumerate}

\begin{lem} \label{L:spe} Let $v\in\WS$ with $\Supp(v)=I_\af$ and
$\ell(v)<2n$. Then $v$ has a unique specially embedded word, denoted
$v_\spe\sw u''$, which is obtained by rotating the normal embedding
$v_\nor\sw u'$ at $p=\lgap(v_\nor\sw u')$. This given, we define the
special cover $v^*\in \CC_v$ of $v$, to be the unique cover
$w\in\CC_v$ such that the normally embedded word $w_\nor$ is
obtained from the specially embedded word $v_\spe\sw u''$ by
inserting $\fgap(v_\spe\sw u'')$. Moreover, if $\ell=\fgap(v_\nor\sw
u')$ and $\ell(v)<2n-1$ then $\ell=\fgap(v^*_\nor\sw u'')$, except
when $u'=N_{k,k-1}$ and $\lgap(v_\nor\sw u')<\ell \le k-1$, in which
case $\fgap(v^*_\nor\sw u'')=\ell-1$.
\end{lem}
\begin{proof}
Suppose $v_\nor \sw N_{k,k-1}$. Let $\ell=\fgap(v_\nor\sw
N_{k,k-1})$.

Suppose $p\sw \Inc{1}{k-1}\sw N_{k,k-1}$ is missing from $v_\nor$.
Let $v_\nor=u_1u_2pu_3u_4\sw N_{k,k-1}$ where $u_1\sw \La_{k,k}$,
$u_2\sw \Dec{k-1}{p+1}$, $u_3\sw V^{p-1,p-1}$, and $u_4\sw
\Inc{p+1}{k-1}$. Then using Lemma \ref{L:LaVunique}(3) we have
$v_\nor\equiv u_1u_2pu_4u_3\equiv u_1 u_4^- (k-1) u_2^- u_3\equiv
u_4^- u_1 (k-1)u_2^- u_3=:u \sw N_{p,p-1}$. Now $u_4^-\sw
\Inc{p}{k-2}$. In this case, $u$ is special if and only if
$u_4=\Inc{p+1}{k-1}$, that is, $p=\lgap(v_\nor\sw N_{k,k-1})$.
Suppose so. Then $\fgap(v^*_\nor\sw N_{p,p-1})$ is $\ell$ unless
$\ell\sw \Dec{k-1}{p+1}\sw N_{k,k-1}$, in which case the answer is
$\ell-1$.

Suppose $p\sw \Dec{k-1}{1}\sw N_{k,k-1}$ is missing from $v_\nor$.
Let $v_\nor=u_1 u_2  u_3 p u_4$ where $u_1\sw \La_{k,k}$, $u_2\sw
\Dec{k-1}{p+1}$, $u_3\sw V^{p-1,p-1}$, and $u_4\sw \Inc{p+1}{k-1}$.
Then $v_\nor\equiv u_3 u_1 u_2 p u_4 \equiv u_3 u_1 u_4^- (k-1)
u_2^- \equiv u_3 u_4^- u_1 (k-1) u_2^-=:u\sw \RN_{p-1,p}$. In this
case $u$ is special if and only if $u_3=V^{p-1,p-1}$ and
$u_4=\Inc{p+1}{k-1}$, that is, $p=\lgap(v_\nor\sw N_{k,k-1})$.
Suppose so. Then $\fgap(v^*_\nor\sw \RN_{p-1,p})$ is $\ell$ unless
$\ell\sw \Dec{k-1}{1}\sw N_{k,k-1}$ and $\ell>p$ ($\ell=p$ cannot
happen if $\ell(v)<2n-1$), in which case the answer is $\ell-1$.

Suppose $p\sw \Dec{n-1}{k}\sw N_{k,k-1}$ is missing from $v_\nor$.
Let $v_\nor=\Inc{k}{p-1}p u_1 u_2 u_3\sw N_{k,k-1}$ where $u_1\sw
\La_{p+1,p+1}$, $u_2\sw \Dec{p-1}{k}$, and $u_3\sw V^{k-1,k-1}$. We
have $v_\nor \equiv \Inc{k}{p-1} p u_2 u_3 u_1 \equiv u_2^+ k
\Inc{k+1}{p} u_3 u_1 \equiv u_2^+ k u_3 \Inc{k+1}{p} u_1=:u\sw
\RN_{p,p+1}$. In this case $u$ is special if and only if
$u_2=\Dec{p-1}{k}$ and $u_3=V^{k-1,k-1}$, that is,
$p=\lgap(v_\nor\sw N_{k,k-1})$. Suppose so. Then $\ell\sw
\Dec{n-1}{p+1}\sw N_{k,k-1}$, and $\ell=\fgap(v^*_\nor\sw
\RN_{p,p+1})$.

Suppose $v_\nor\sw \RN_{k-1,k}$. Let $\ell=\fgap(v_\nor\sw
\RN_{k-1,k})$. Let $p\sw \RN_{k-1,k}$ be missing for $v_\nor$. Then
from the definitions we have $p\sw \Dec{n-1}{k}\sw \RN_{k-1,k}$.
Write $v_\nor=\Dec{k-1}{1}\Inc{0}{n} u_1 u_2$ where $u_1\sw
\Dec{n-1}{p+1}$ and $u_2\sw \Dec{p-1}{k}$. Therefore $v_\nor\equiv
u_2^+ \Dec{k-1}{1}\Inc{0}{n} u_1=:u\sw \RN_{p,p+1}$. $u$ is special
if and only if $u_2=\Dec{p-1}{k}$, that is, $p=\lgap(v_\nor\sw
\RN_{k-1,k})$. Suppose so. Then $\ell=\fgap(v^*_\nor\sw
\RN_{p,p+1})$.

Thus rotation at $\lgap(v_\nor\sw u')$ creates a particular
specially embedded word which we shall denote by $v_\spe\sw u''$. It
remains to show that $v_{\spe}$ is unique.  Suppose $u\in\Red(v)$ is
such that $u\sw u'$ is a special embedding. Rotating $u\sw u'$ at
$\fgap(u\sw u')$, we obtain the normal embedding of $v_\nor$, which
is unique. The explicit computation of this rotation shows that it
is the inverse of the rotation at the last gap of the normal
embedding of $v_\nor$ (which was given above explicitly in all
cases). It follows that there is a unique specially embedded word
for $v$.
\end{proof}

For later use we summarize the construction of Lemma \ref{L:spe} in
the following table, where $p$ is the last gap.  We have indicated
the form of $v_\spe$, and used the symbol $*$ to indicate where a
letter (either $k$ or $k-1$) can be added to obtain $v_\nor^*$.
\begin{center}
\begin{tabular}{|c|c|c|c|c|}
\hline $p\subset$ & $v_\spe$ & $u_1\subset$ & $u_2\subset$ & $u_3\subset$  \\
\hline $\Inc{1}{k-1}\sw N_{k,k-1}$ & $\Inc{p}{k-2} * u_1 (k-1)u_2^-
u_3$ &
$\La_{k,k}$ & $\Dec{k-1}{p+1}$ & $V^{p-1,p-1}$\\
\hline $\Dec{k-1}{1}\sw N_{k,k-1}$ & $\Dec{p-1}{1} \Inc{0}{k-2}
* u_1 (k-1) u_2^-$ & $ \La_{k,k}$&$ \Dec{k-1}{p+1}$ &
\\ \hline $\Dec{n-1}{k}\sw N_{k,k-1}$ & $\Dec{p}{1}
\Inc{0}{k-1} *  \Inc{k+1}{p} u_1$ & $\La_{p+1,p+1}$ & &
\\ \hline
$\Dec{n-1}{k}\sw \RN_{k-1,k}$ & $\Dec{p-1}{k+1} *
\Dec{k-1}{1}\Inc{0}{n} u_1$& $\Dec{n-1}{p+1}$ & &\\
\hline
\end{tabular}
\end{center}

\begin{example} Take $n=7$ and $v_\nor=56754310124\subset N_{5,4}$.
In this case $p=3$, $v_\spe=35675431012$, and
$v_\nor^*=345675431012$.
\end{example}

\subsection{Kinds of covers}
Let $v\in\WS$ be fixed. The set $I_\af$ is divided into four kinds
of letters. Let $j$ be $v$-internal if $j\in\Supp(v)$. If $j\not\in
\Supp(v)$, let $j$ be $v$-isolated, $v$-adjoining, and $v$-merging
if the number of components of $\Supp(v)$ adjacent to $j$ is $0$,
$1$, or $2$, that is, $|\{j-1,j+1\}\cap \Supp(v)|$ is $0$, $1$, or
$2$.

Let $w\in\CC_v$ with a reduced word $\tu\in\Red^Z(w)$ and a letter
$j\sw \tu$ whose omission leaves a reduced word $u\in\Red^Z(v)$.
Then we call the cover $w$ internal, isolated, adjoining, or
merging, according as $j$ is (with respect to $v$). Such $w$ have
$c(w)$ equal to $c(v)$, $c(v)+1$, $c(v)$, and $c(v)-1$ respectively.

In the case of an internal cover the omitted letter $j$ may vary if
the reduced word $\tu$ is changed; however the component $C$ of
$j\in \Supp(v)$ depends only on $w$. Moreover $w_C \gtrdot v_C$ and
$w_{C'}=v_{C'}$ for components $C'$ of $\Supp(v)$ with $C'\ne C$.

If $j\not\in\Supp(v)$ then the omitted letter $j$ is uniquely
determined by $w$.

\begin{lem} \label{L:ncov} Let $v\in\WS$ with $\ell(v)<2n$.
\begin{enumerate}
\item
For each $v$-isolated letter $j\in I_\af$ there is a unique cover
$w$ in $\CC_v$ that omits $j$, namely, $s_j v$.
\item
For each $v$-adjoining letter $j\in I_\af$ there are exactly two
covers $w\in \CC_v$ that omit $j$, namely, $s_j v$ and $vs_j$.
\item
For each $v$-merging letter $j\in I_\af$ there are exactly four
covers $w\in\CC_v$ that omit $j$. Let $u\in\Red(v)$ and $u_+$ and
$u_-$ the subwords of $u$ given by the restriction to the letters
greater and less than $j$ respectively and let $v_+$ and $v_-$ be
the corresponding elements of $\WS$. Then the four covers of $v$
that omit $j$ are $s_j v_+v_-$, $v_+s_jv_-$, $v_+v_-s_j$, and
$v_-s_jv_+$.
\end{enumerate}
\end{lem}
\begin{proof} We prove (3) as the other cases are easier. We observe
that $v_+$ and $v_-$ are defined independent of the reduced word
$u$. The four given elements of $\tC_n$ are all covers of $v$ that
omit $j$, and are distinct since $j-1\in\Supp(v_-)$ and
$j+1\in\Supp(v_+)$. We now realize each of them by reduced words
that are Zs. By Lemma \ref{L:LaVunique} let $u_+\in\Red^\La(v_+)$
and $u_-\in\Red^V(v_-)$. Then $ju_+u_-$, $u_+ju_-$, and $u_+u_-j$
are all Ns, and $u_-ju_+$ is a $\RN$, and they are reduced words for
the above elements of $\tC_n$. It remains to show that if
$w\in\CC_v$ omits $j$ then $w$ is one of the four given covers. Let
$\tu\in\Red^Z(w)$ and $j\in \tu$ such that the omission of $j$ from
$\tu$ leaves $u\in\Red^Z(v)$. Suppose $\tu\sw N_{k,k-1}$. Suppose
$j\ge k$. Let $u=u_1u_2u_3$ where $u_2\sw \La_{j+1,j+1}$, so that
$\tu=u_1 ju_2u_3$ or $u_1u_2ju_3$. Here $\Supp(u_1)\subset[0,j-1]$
and $\Supp(u_3)\subset[0,j-1]$ while $\Supp(u_2)\subset[j+1,n]$. If
$\tu=u_1ju_2u_3$ then $\tu\equiv u_1 j u_3 u_2$. But not both $u_1$
and $u_3$ can contain $j-1$, for if they did then $u_1$ ends with
$j-1$ and $u_3$ starts with $j-1$ and $u\equiv u_1 u_3 u_2$ is not
reduced. If $u_1$ does not contain $j-1$ then $\tu\equiv j u_1 u_3
u_2$ and $w=s_j v_-v_+$. If $u_2$ does not contain $j-1$ then
$\tu\equiv u_1u_3ju_2$ and $w=v_-s_jv_+$. The cases that
$\tu=u_1u_2ju_3$, $j<k$ and $\tu\sw \RN_{k-1,k}$ are similar.
\end{proof}

We now classify the internal covers of $v$. For this purpose we may
assume $\Supp(v)$ has a single component. For $k,\ell\le M$ let
$\La_{k,\ell}^M = \Inc{k}{M}\Dec{M-1}{\ell}$ and for $m\le k,\ell$
let $V^{k,\ell}_m = \Dec{k}{m}\Inc{m+1}{\ell}$.

\begin{lem} \label{L:intcov} Suppose $v\in\WS$ is such that
$\Supp(v)$ consists of a single component $[m,M]$.
\begin{enumerate}
\item
If $M<n$ (resp. $m>0$) then the internal covers of $v$ are precisely
those obtained by inserting missing letters into $u\sw V^{M,M}_m$
(resp. $u\sw \La_{m,m}^M$) where $u$ is the unique element of
$\Red^V(v)$ (resp. $\Red^\La(v)$).
\item
If $m=0$ and $M=n$, consider the normal embedding $v_\nor\sw u'$.
Then the internal covers of $v$ are precisely those obtained by
inserting missing letters into $v_\nor\sw u'$ (\textit{normal
covers}), plus the \textit{special cover}, which is obtained from
the special embedding of $v_\spe$ by inserting the first missing
letter.
\end{enumerate}
\end{lem}
\begin{proof} Since internal covers do not change the support and the
support is assumed to be an interval, by Lemma \ref{L:addable}
adding any missing letter of $[m,M]$ creates a cover. Any internal
cover $w\in\CC_v$ has the same support as $v$. If $M<n$ then $w$ has
a reduced word that is a V, and removing one of its letters yields a
reduced word for $v$ that is a V. By uniqueness this word must be
$u$. This proves (1) for $M<n$, and $m>0$ is similar. For (2)
suppose $\Supp(v)=I_\af$. Let $w\in\CC_v$. Consider the normal
embedding of $w_\nor$, which is unique since $\Supp(w)=I_\af$. There
is a unique letter in $w_\nor$ whose removal yields an embedded
reduced word $u$ for $v$. It is easy to check that $u$ is either
$v_\nor$ normally embedded or $v_\spe$ specially embedded.
\end{proof}

\subsection{Associated coroots}
\label{SS:coroot}

Let $v \le v'$ with $v,v'\in \tC_n$ and let $u\in\Red(v)$ and
$u'\in\Red(v')$ be such that $u\sw u'$. For $j\sw u'$, define
$\al^\vee(u\sw u',j)$ to be $\al_{vw}^\vee$ if adding the given
occurrence of $j$ to $u$ creates a reduced word for a cover
$w\gtrdot v$, and $0$ otherwise. In particular the value is $0$ if
$j\sw u$. Define $\al^\vee(u\sw u')=\sum_{j\sw u'} \al^\vee(u\sw
u',j)$. The following Lemma holds by the definitions.

\begin{lem} \label{L:cofactor}
Let $v_1 \le v_1'$, $v_2\le v_2'$, $u_1,u_1',u_2,u_2'$ reduced words
for $v_1,v_1',v_2,v_2'$ such that $u_1\sw u_1'$ and $u_2\sw u_2'$.
Then
\begin{align*}
  \al^\vee(u_1u_2\sw u_1'u_2') =
  v_2^{-1} \al^\vee(u_1\sw u_1') + \al^\vee(u_2\sw u_2').
\end{align*}
\end{lem}

It is straightforward to compute sums of associated coroots for
subwords of increasing or decreasing reduced words.

\begin{lem} \label{L:coincdec} Let $0\le m\le M\le n$ and $u\sw \Inc{m}{M}$ or $u\sw
\Dec{M}{m}$. Then
\begin{equation} \label{E:comono}
\begin{aligned}[3]
 \al^\vee(u\sw \Inc{m}{M})&=
  \al_k^\vee+\al_{k+1}^\vee+\dotsm+\al_M^\vee &\qquad&\text{if
  $M<n$} \\
  \al^\vee(u\sw\Inc{m}{n}) &=
  \al_k^\vee+\al_{k+1}^\vee+\dotsm+\al_{n-1}^\vee+2\al_n^\vee&&
  \\
  \al^\vee(u\sw\Dec{M}{m}) &=
  \al_k^\vee+\al_{k-1}^\vee+\dotsm+\al_m^\vee &\qquad&\text{if $m>0$} \\
  \al^\vee(u\sw\Dec{M}{0}) &=
  \al_k^\vee+\al_{k-1}^\vee+\dotsm+\al_1^\vee+2\al_0^\vee&&
\end{aligned}
\end{equation}
where $k=\fgap(u\sw u')$ for $u'=\Inc{m}{M}$ or $u'=\Dec{M}{m}$. If
$k$ does not exist (that is, $u=u'$) then the sum is $0$.
\end{lem}

Next we compute sums of associated coroots for subwords of Vs and
$\La$s whose support are intervals. We assume there is a letter
missing in the initial monotonic part of the embedded word;
otherwise the result is given by Lemma \ref{L:coincdec}.

\begin{lem} \label{L:coV}
Let $v\in\WS$ have $\Supp(v)=[m,M]\subsetneq I_\af$. Suppose
$u\in\Red^V(v)$ with $M<n$ (resp. $u\in\Red^\La(v)$ with $m>0$) of
the form $u=u_1u_2$ with $u_1\sw\Dec{M}{m}$ (resp.
$u_1\sw\Inc{m}{M}$) and $u_2\sw\Inc{m+1}{M}$ (resp.
$u_2\sw\Dec{M-1}{m}$) so that $u\sw V_m^{M,M}$ (resp. $u\sw
\La_{m,m}^M$). Suppose that $u_1\ne u'$ for $u'=\Dec{M}{m}$ (resp.
$u'=\Inc{m}{M}$) so that $k=\fgap(u_1\sw u')$ is well-defined. Let
$k'=\fgap(u_2\sw u'')$ where $u''=\Inc{m+1}{M}$ (resp.
$u''=\Dec{M-1}{m}$); if $u_2=u''$ then set $k'=M+1$ (resp.
$k'=m-1$). Then
\begin{equation}\label{E:coV}
\begin{aligned}[3]
\al^\vee(u\sw V_m^{M,M}) &=
  (\al_m^\vee+\dotsm+\al_{k-1}^\vee) + (\al_{k'}^\vee+\dotsm+\al_M) &\qquad &\text{if $m>0$} \\
\al^\vee(u\sw V_0^{M,M}) &=
2(\al_0^\vee+\dotsm+\al_{k-1}^\vee)+(\al_k^\vee+\dotsm+\al_M^\vee)&&\\
\al^\vee(u\sw \La_{m,m}^M) &=
(\al_m^\vee+\dotsm+\al_{k'}^\vee)+(\al_{k+1}^\vee+\dotsm+\al_M^\vee)
&\qquad&\text{if $M<n$} \\
\al^\vee(u\sw \La_{m,m}^n) &= (\al_m^\vee+\dotsm+\al_k^\vee)+
2(\al_{k+1}^\vee+\dotsm+\al_n^\vee)
\end{aligned}
\end{equation}
\end{lem}
\begin{proof}
Since $\Supp(v)$ is an interval and we are adding letters in that
same interval, adding any missing letter creates a reduced word by
Lemma \ref{L:addable}. Let $v_2\in\WS$ be such that
$u_2\in\Red(v_2)$.

Let $M<n$ and $u\sw V_m^{M,M}$. Suppose $m>0$. We have
\begin{align*}
  \al^\vee(u_1\sw\Dec{M}{m}) =
  \al_m^\vee+\dotsm+\al_{k-1}^\vee+\al_k^\vee.
\end{align*}
By the assumption on support, since $k\notin\Supp(u_1)$ we have
$k\in \Supp(u_2)$ and $k\ne k'$. Therefore
\begin{align*}
  v_2^{-1}\al^\vee(u_1\sw\Dec{M}{m}) =\al_m^\vee+\dotsm+\al_{k-1}^\vee.
\end{align*}
By Lemma \ref{L:cofactor} the desired expression is obtained.

Suppose  $m=0$. Since $0\in\Supp(w)$ we have
\begin{align*}
  \al^\vee(u_1\sw\Dec{M}{0}) &=
  \al_k^\vee+\al_{k-1}^\vee+\dotsm+\al_1^\vee+2\al_0^\vee \\
  v_2^{-1}\al^\vee(u_1\sw\Dec{M}{0}) &=
  \begin{cases}
   2(\al_0^\vee+\dotsm+\al_{k'-1}^\vee)+(\al_{k'}^\vee+\dotsm+\al_{k-1}^\vee) &\text{if $k>k'$} \\
    2(\al_0^\vee+\dotsm+\al_{k-1}^\vee)+(\al_k^\vee+\dotsm+\al_{k'-1}^\vee)    &\text{if $k<k'$.}
  \end{cases}
\end{align*}
By Lemma \ref{L:cofactor} we obtain the desired formula.

The other computations are similar.
\end{proof}

\begin{lem} \label{L:cofullsup} Suppose $v\in\WS$ is such that
$\Supp(v)=I_\af$ with normal embedding $v_\nor\sw N_{k,k-1}$,
$v_\nor$ does not contain $\Dec{n}{0}\sw N_{k,k-1}$, and
$\ell=\fgap(v_\nor\sw N_{k,k-1})$. Then
\begin{align*}
  \al^\vee(u\sw N_{k,k-1}) =
  \begin{cases}
2(\al_0^\vee+\dotsm+\al_{k-1}^\vee)+(\al_k^\vee+\dotsm+\al_\ell^\vee)
& \text{if $\ell\ge k$} \\
  2(\al_0^\vee+\dotsm+\al_{\ell-1}^\vee)+(\al_\ell^\vee+\dotsm+\al_{k-1}^\vee)
  & \text{if $\ell<k$.}
  \end{cases}
\end{align*}
\end{lem}
\begin{proof} Follows from Lemmata \ref{L:cofactor},
\ref{L:coincdec} and \ref{L:coV}.
\end{proof}

\begin{lem} \label{L:cointernal} Suppose $v\in \WS$ is such that
$\ell(v)<2n$ and $\Supp(v)=[m,M]$ is an interval. Then the sum of
$\al_{vw}^\vee$ as $w$ runs over the internal covers in $\CC_v$, is
given by
\begin{align*}
  &2(\al_m^\vee+\al_{m+1}^\vee+\dotsm+\al_M^\vee) \\
  &-  \chi(M<n)(-\al_{M+1}^\vee+v^{-1} \al_{M+1}^\vee)\\
  &-
  \chi(m>0)(-\al_{m-1}^\vee+v^{-1} \al_{m-1}^\vee).
\end{align*}
\end{lem}
\begin{proof} We begin with the most involved case, when
$\Supp(v)=I_\af$. In this case we must show that
$\sum_{w\in\CC_v}\al_{vw}^\vee=2(\al_0^\vee+\dotsm+\al_n^\vee)=2K$.
By Lemma \ref{L:intcov}, $\sum_{w\in\CC_v} \al_{vw}^\vee
=\al^\vee(v_\nor\subset u')+\al_{vv^*}^\vee$ where $v_\nor\sw u'$ is
the normal embedding. For the computation of the special coroot
$\al^\vee_{vv^*}$ we shall refer back to the proof of Lemma
\ref{L:spe} without further mention, for the explicit computations
of the special embedding $v_\spe \sw u''$ given by rotating the
normal embedding $v_\nor \sw u'$ at $p=\lgap(v_\nor\sw u')$.  The
reader may find the table after Lemma \ref{L:spe} helpful.

Suppose that $v_\nor\sw N_{k,k-1}$ is the normal embedding for some
$1\le k\le n$. Let $\ell=\fgap(v_\nor\sw N_{k,k-1})$ and
$p=\lgap(v_\nor\sw N_{k,k-1})$.

Suppose $\ell\sw \Inc{1}{k-1}\sw N_{k,k-1}$. We have $p\sw
\Inc{\ell}{k-1}\sw N_{k,k-1}$, $v^*_\nor\sw N_{p,p-1}$ is normally
embedded, and $\ell=\fgap(v^*_\nor\sw N_{p,p-1})$. We compute
$\al_{vv^*}^\vee=s_{\ell-1}\dotsm s_1s_0s_1\dotsm s_{n-1}s_n
s_{n-1}\dotsm
s_k(\al_{k-1}^\vee)=2(\al_0^\vee+\dotsm+\al_n^\vee)-(\al_\ell^\vee+\dotsm+\al_{k-1}^\vee)$.
Combining this with $\al^\vee(v_\nor\sw
N_{k,k-1})=\al_\ell^\vee+\dotsm+\al_{k-1}^\vee$ from Lemma
\ref{L:coincdec} we obtain the total $2K$.

Suppose $\ell\sw \Dec{k-1}{1}\sw N_{k,k-1}$. Since $\Supp(v)=I_\af$,
$\ell\sw\Inc{1}{k-1}\sw N_{k,k-1}$ appears in $v_\nor$. Suppose
$p\sw \Inc{1}{k-1}\sw N_{k,k-1}$. Suppose first that $p>\ell$. Then
$v^*_\nor \sw N_{p,p-1}$ is normally embedded with
$\fgap(v^*_\nor\sw N_{p,p-1})=\ell$. We have
$\al^\vee_{vv^*}=s_\ell\dotsm s_{\ell+1}\dotsm
s_{n-1}s_ns_{n-1}\dotsm
s_k\al_{k-1}^\vee=2(\al_\ell^\vee+\dotsm+\al_n^\vee)-(\al_\ell^\vee+\dotsm+\al_{k-1}^\vee)$.
By Lemma \ref{L:cofullsup} for $\ell<k$ we have $\al^\vee(v_\nor\sw
N_{k,k-1})=2(\al_0^\vee+\dotsm+\al_{\ell-1}^\vee)+(\al_\ell^\vee+\dotsm+\al_{k-1}^\vee)$,
and the total is $2K$. Suppose next that $p<\ell$. Again
$v^*_\nor\sw N_{p,p-1}$ is normally embedded and $\fgap(v^*_\nor\sw
N_{p,p-1})=\ell-1$. The coroot computation is similar to the
previous case. By definition $p$ occurs after $\ell$ in $N_{k,k-1}$
so the remaining subcase is $p\sw \Dec{\ell}{1}\sw N_{k,k-1}$. Then
$v^*_\nor\sw \RN_{p-1,p}$ is normally embedded and
$\fgap(v^*_\nor\sw \RN_{p-1,p})=\ell-1\sw \Dec{n-1}{p}$. The coroot
computation is similar.

Suppose $\ell\sw\Dec{n-1}{k}\sw N_{k,k-1}$. Since $\Supp(v)=I_\af$,
$k-1$ must occur in $v_\spe$ after $\ell$. In all cases
$\fgap(v^*_\nor\sw u'')=\ell$. If $p\sw\Inc{1}{k-1}\sw N_{k,k-1}$
then $v^*_\nor\sw N_{p,p-1}$ is normally embedded with
$\al^\vee_{vv^*}=s_{k-1} s_{\ell+1}\dotsm s_{n-1}s_ns_{n-1}\dotsm
s_k (\al_{k-1}^\vee)
=s_{k-1}(2(\al_{\ell+1}^\vee+\dotsm+\al_n^\vee)+(\al_{k-1}^\vee+\dotsm+\al_\ell^\vee))=
2(\al_{\ell+1}^\vee+\dotsm+\al_n^\vee)+(\al_k^\vee+\dotsm+\al_\ell^\vee)$.
Combined with $\al^\vee(v_\nor\sw
N_{k,k-1})=2(\al_0^\vee+\dotsm+\al_{k-1}^\vee)+(\al_k^\vee+\dotsm+\al_\ell^\vee)$
from Lemma \ref{L:cofullsup} we obtain a total of $2K$. If $p\sw
\Dec{k-1}{1}\sw N_{k,k-1}$ then $v^*_\nor \sw \RN_{p-1,p}$ is the
normal embedding with coroot computation proceeding as in the
previous case. If $p\sw \Dec{\ell}{k}\sw N_{k,k-1}$ then
$v^*_\nor\sw \RN_{p,p+1}$ and the coroot computation proceeds in the
same way.

The case $v_\nor\sw \RN_{k-1,k}$ is very similar to the case $v_\nor
\sw N_{k,k-1}$ with $\ell\sw \Inc{1}{k-1}\sw N_{k,k-1}$ and
$p\sw\Inc{\ell}{k-1}$.

This finishes the case $\Supp(v)=I_\af$.

Next we consider the case $m=0$ and $M<n$. Let $u\in\Red^V(v)$ with
$u\sw V^{M,M}$. In this case the sum of $\al_{vw}^\vee$ for
$w\in\CC_v$ an internal cover of $v$, is equal to $\al^\vee(u\sw
V^{M,M})$. Let $u=u_10u_2$ where $u_1\sw \Dec{M}{1}$ and $u_2\sw
\Inc{1}{M}$. Let $a\sw \Dec{M}{1}$ (resp. $b\sw \Inc{1}{M}$) be the
first missing letter from $u_1$ (resp. $u_2$), which exists if
$u_1\ne \Dec{M}{1}$ (resp. $u_2\ne \Inc{1}{M}$). By Lemma
\ref{L:coV} we have
\begin{align*}
  &\al^\vee(u\sw V^{M,M})= \\
  &\begin{cases}
  2(\al_0^\vee+\dotsm+\al_{a-1}^\vee)+(\al_a^\vee+\dotsm+\al_M^\vee)
  & \text{if $u_1\ne \Dec{M}{1}$} \\
  \al_b^\vee+\dotsm+\al_M^\vee&\text{if $u_1=\Dec{M}{1}$ and
  $u_2\ne\Inc{1}{M}$} \\
  0 &\text{if $u_1=\Dec{M}{1}$ and $u_2=\Inc{1}{M}$.}
  \end{cases}
\end{align*}
Consider $\beta=-\al_{M+1}^\vee+v^{-1} \al_{M+1}^\vee$. Suppose
first that $u_1\ne \Dec{M}{1}$. Since $a$ is missing from $u_1$ and
$\Supp(v)=[0,M]$ is an interval, $a\in u_2$. Therefore
$\beta=-\al_{M+1}^\vee+s_a s_{a+1}\dotsm s_M \al_{M+1}^\vee=
\al_a^\vee+\dotsm+\al_M^\vee$, which yields the desired total.
Suppose $u_1=\Dec{M}{1}$ and $u_2\ne \Inc{1}{M}$. Then
$\beta=-\al_{M+1}^\vee+s_{b-1} \dotsm s_1 s_0 s_1 \dotsm s_M
\al_{M+1}^\vee=2(\al_0^\vee+\dotsm+\al_{b-1}^\vee)+(\al_b^\vee+\dotsm+\al_M^\vee)$
as desired. If $u_1=\Dec{M}{1}$ and $u_2=\Inc{1}{M}$ then
$\beta=-\al_{M+1}^\vee+s_M\dotsm s_1 s_0 s_1 \dotsm s_M
\al_{M+1}^\vee = 2(\al_0^\vee+\dotsm+\al_M^\vee)$ as desired.

The case that $m>0$ and $M=n$ is entirely similar to the previous
case. The remaining case is $0<m$ and $M<n$. Using $u\in \Red^V(v)$
and $u_1mu_2=u\sw V^{M,M}_m$, the proof is similar to the case for
$m=0$ and $M=n$ except that one must also compute
$-\al_{m-1}^\vee+v^{-1}\al_{m-1}^\vee$, which equals
$\al_m^\vee+\dotsm+\al_{b-1}^\vee$ if $u_2\ne \Inc{m+1}{M}$ and
equals $\al_m^\vee+\dotsm+\al_M^\vee$ if $u_2=\Inc{m+1}{M}$.
\end{proof}

\subsection{Proof of Proposition \ref{p:2}}
We fix $v\in\WS$ with $\ell(v)<2n$ and $i\in I_\af$. Let
$$\CC_v'=\{w\in\CC_v\mid \text{$\al_i^\vee$ occurs in
$\al_{vw}^\vee$}\}.$$

\textbf{Case 1. $i\not\in \Supp(v)$.} Let $w\in\CC_v'$. Since
$\al_i^\vee$ occurs in $\al_{vw}^\vee$ and $i\not\in\Supp(v)$ it
follows that $i\in\Supp(w)$. It is easy to check that $\al_i^\vee$
occurs in $\al_{vw}^\vee$ with coefficient $1$. The desired
multiplicity is obtained by Lemma \ref{L:ncov}.

\textbf{Case 2. $i\in\Supp(v)$.} Let $C=[m,M]$ be the component of
$i$ in $\Supp(v)$. The covers in $\CC_v'$ add letters that are
either in $C$ or adjacent to $C$.

\textbf{Case 2a. $C=I_\af$}. In this case there are only internal
covers. Therefore $\sum_{w\in\CC_v}
\al_{vw}^\vee=2(\al_0^\vee+\dotsm+\al_n^\vee)$ by Lemma
\ref{L:cointernal}. Since $c(w)=c(v)$ for all such $w$, Proposition
\ref{p:2} is verified in this case.

\textbf{Case 2b.} $C=[0,M]$ with $M<n$. (The case $C=[m,n]$ with
$m>0$ is similar.) Write $v=v_Cv'$ where $v'$ is the product of the
components of $v$ other than $v_C$. Then the internal covers in
$\CC_v'$ consist of the $w\in \CC_v$ such that $w_C \gtrdot v_C$ and
$w_{C'}=v_{C'}$ for components $C'$ of $\Supp(v)$ with $C'\ne C$.
The sum of $\al_{vw}^\vee$ for internal covers of $v$ in $\CC_v'$,
is given by Lemma \ref{L:cointernal}. For such $w$ we have
$c(w)=c(v)$. Suppose $M+2\not\in \Supp(v)$, so that $M+1$ is
$v$-adjoining. Then all the noninternal covers $w\in \CC_v'$ adjoin
the letter $M+1$ to $C$; such $w$ satisfy $c(w)=c(v)$ also. By Lemma
\ref{L:ncov} there are exactly two adjoining covers in $\CC_v$,
namely, $s_{M+1}v=v's_{M+1}v_C$ and $vs_{M+1}$. The latter has
associated coroot $\al_{M+1}^\vee$ and therefore does not contribute
$\al_i^\vee$ for $i\in C$. For $w=s_{M+1}v$ we have
$\al_{vw}^\vee=v_C^{-1}\al_{M+1}^\vee$. Combining this with the sum
of coroots for internal covers associated to the component $C$ of
$\Supp(v)$, by Lemma \ref{L:cointernal} the coefficient of
$\al_i^\vee$ is $2$ as desired. Suppose $M+2\in \Supp(v)$. Then
$M+1$ is $v$-merging. By Lemma \ref{L:ncov} there are four covers
$w\in\CC_v$ that add $M+1$; each has $c(w)=c(v)-1$. Their associated
coroots are
\begin{align*}
  \al^\vee_{v,vs_{M+1}} &= \al_{M+1}^\vee \\
  \al^\vee_{v,v's_{M+1}v_C} &= v_C^{-1} \al_{M+1}^\vee \\
  \al^\vee_{v,s_{M+1}v'v_C} &=
  -\al_{M+1}^\vee+v_C^{-1}\al_{M+1}^\vee+(v')^{-1}\al_{M+1}^\vee \\
  \al^\vee_{v,v_C s_{M+1} v'} &= (v')^{-1} \al_{M+1}^\vee.
\end{align*}
The sum of these coroots, forgetting the $\alpha_j^\vee$ for
$j\not\in C$, is $2(-\alpha_{M+1}^\vee+v_C^{-1} \alpha_{M+1}^\vee)$.
Together with the coroots corresponding to internal covers given by
Lemma \ref{L:cointernal}, which receive a relative factor of 2 since
$c(w)=c(v)$ for internal covers and $c(w)=c(v)-1$ for merging
covers, gives the desired result.

\textbf{Case 2c.} $0<m<M<n$. The computations for this case are
similar to those above.

This completes the proof of Proposition \ref{p:2}.

\section{Hopf property of $\Phi$}
\label{S:Hopf} In this section we prove Theorem \ref{T:phiP}.

\subsection{A coproduct formula for nilHecke algebras}
In Proposition \ref{prop:del A} below, we give a complicated but
explicit formula for $\pzt(\Delta(A_w))$ for $w\in W_\af$.  This
formula is valid for the nilHecke algebra for any Cartan datum.

Let $v\in\Red(w)$ and consider the tuples $\vb =
[v^{(1)},v^{(2)},\ldots,v^{(k)}]$ consisting of subwords $v^{(i)}
\subset v$ (the embedding of the $v^{(i)}$ are fixed).  Let $x_i$
(resp. $y_i$) be the first (resp. last) letters of $v^{(i)}$,
considered as subletters of $v$ via the embedding $x_i \subset
v^{(i)} \subset v$.  Define $\Sc_v$ to be the set of (possibly
empty) tuples $\vb = [v^{(1)},\ldots,v^{(k)}]$ such that:
\begin{enumerate}
\item
$v^{(i)}$ is a subword of length at least two of $v\setminus
\{x_1,\ldots,x_{i-1}\}$, which is the word $v$ with the letters
$x_1,\ldots,x_{i-1}$ removed;
\item
$y=y_1 y_2\cdots y_k$ is a subword of $v$; and
\item
the letters $x_i$ are distinct from the letters $y_j$ as subwords in
$v$.
\end{enumerate}

For a word $u=u_1u_2\cdots u_\ell$ and a tuple
$\vb=[v^{(1)},\ldots,v^{(k)}]$ let
\begin{equation*}
    b_u = \prod_{i=1}^{\ell-1} b_{u_i u_{i+1}} \quad \text{and} \quad
    b_{\vb} = \prod_{i=1}^k b_{v^{(i)}},
\end{equation*}
where $b_{ij}=-\langle \alpha_i^\vee,\alpha_j \rangle=-a_{ij}$ is
the negative of the entry of the Cartan matrix.

For a given $\vb=[v^{(1)},\ldots,v^{(k)}]\in \Sc_v$ let
$x=\{x_1,\ldots,x_k\}$ and $y=\{y_1,\ldots,y_k\}$. Then set
$v\setminus( x \cup y)$ to be the word $v$ with the letters in $x$
and $y$ removed. For a subword $u \subset v \setminus( x \cup y)$
(again with a fixed embedding), define $u.y$ to be the word $u$ with
the letters in $y$ added in the correct order of $v$.

\begin{prop} \label{prop:del A} For $w\in W_\af$ and $v\in\Red(w)$,
\begin{equation} \label{eq:phi on A}
\pzt \Delta (A_w) = \sum_{\vb=[v^{(1)},\ldots,v^{(k)}] \in \Sc_v}
b_{\vb}
   \sum_{u\subset v\setminus (x \cup y)} A_{u.y} \otimes A_{u^\bot.y}
\end{equation}
where $u^\bot$ is the complement  word of $u$ in $v\setminus (x \cup
y)$.
\end{prop}

\begin{example}
Take $v=ijkl$ and $\vb=[v^{(1)}]$ with $v^{(1)}=v$, so that
$v\setminus(x \cup y)=jk$. Take the subword $u=j$ of $v\setminus (x
\cup y)$. Then $u.y=jl$ and $u^\bot.y=kl$, so that the term $A_{u.y}
\otimes A_{u^\bot.y} = A_{jl} \otimes A_{kl}$ appears with
coefficient $b_{ij}b_{jk}b_{kl}$ for this particular $\vb$ and $u$
in the sum. Of course such a term can appear in other summands. For
example taking $\vb=[ikl]$, we also have $v\setminus (x \cup y)=jk$.
Taking again $u=j$, we get the term $A_{jl} \otimes A_{kl}$ with a
coefficient of $b_{ik}b_{kl}$.
\end{example}

\begin{proof}[Proof of Proposition~\ref{prop:del A}]
The proof proceeds by induction on $\ell(w)$. For $\ell(w)=1$, let
$v=i\in\Red(w)$. We have $\Sc_v=\{[\,]\}$, so that $\pzt\Delta(A_w)
= \sum_{u\subset v} A_u \otimes A_{u^\bot} = A_i \otimes 1 + 1
\otimes A_i$.

Now suppose $\ell(w)>1$ and let $v=v'i\in\Red(w)$ where $i\in
I_\af$. 
By induction, \eqref{E:DeltaA} and \eqref{E:Deltamult} we have
\begin{equation} \label{eq:calculation}
\begin{split}
 \pzt(\Delta A_v) &= \pzt( \pzt (\Delta A_{v'}) \Delta A_i)\\
&= \pzt \left[ \left(\sum_{\vb'=[v^{'(1)},\ldots,v^{'(k)}]\in
\Sc_{v'}} b_{\vb'}
   \sum_{u' \subset v'\setminus (x' \cup y')} A_{u'.y'} \otimes A_{(u')^\bot.y'}
   \right)\right. \\
 & \left. \qquad(A_i \otimes 1 + 1 \otimes A_i - A_i \otimes \alpha_i A_i) \right]\\
&= \sum_{\vb'=[v^{'(1)},\ldots,v^{'(k)}]\in \Sc_{v'}} b_{\vb'}
   \sum_{u\subset v\setminus (x' \cup y')} A_{u.y'} \otimes A_{u^\bot.y'}\\
&- \pzt \left[  \sum_{\vb' = [v^{'(1)},\ldots,v^{'(k)}] \in
\Sc_{v'}} b_{\vb'}
   \sum_{u'\subset v'\setminus (x' \cup y')} A_{u'.y'i} \otimes A_{(u')^\bot.y'}
   \alpha_i A_i \right]
\end{split}
\end{equation}
where to obtain the first term in the last equation we have merged
the terms obtained from $A_i \otimes 1$ and $1 \otimes A_i$ which
correspond to $i \in u$ and $i \notin u$ respectively.
From \eqref{E:AonP} and \eqref{E:WonP} we have, for an element $w$
with reduced word $z = z_1z_2 \cdots z_k$
\begin{align*}
    \phi_0\left[A_w (-\alpha_i)\right] &= \sum_{w s_\beta\lessdot w}
    -\langle \beta^\vee, \alpha_i \rangle A_{w s_\beta} \\
    &= - \sum_{j = 1}^k  \langle z_k \cdots z_{j+1} \cdot \alpha_{z_j}^\vee, \alpha_i \rangle A_{z\setminus z_j}\\
    & = \sum_{j = 1}^k \left( \sum_{r_1 \cdots r_l \subset
z_{j+1}\cdots z_k} b_{j,r_1} b_{r_1,r_2} \cdots b_{r_l,i} \; A_{z \setminus z_j} \right)\\
 & = \sum_{j = 1}^k \sum_{p \subset zi} b_p A_{z \setminus z_j}
\end{align*}
where in the last equation $p=p_1\cdots p_\ell$ is a subword of $zi$
satisfying: (a) $\ell\ge 2$, (b) $p_\ell=i$, and (c) $p_1 = z_j$.
Applying this equation to the last summand of~\eqref{eq:calculation}
with $z=(u')^\bot.y$ we see that it suffices to find a bijection
$\Phi$ from the set of triples $$(\vb' = [v^{'(1)},\ldots,v^{'(k)}],
u', p)$$ such that (a) $\vb' \in \Sc_{v'}$, (b) $u' \subset v'
\setminus ( x' \cup y')$, and (c) $p=p_1\cdots p_\ell$ is a subword
of $((u')^\bot.y')i$ satisfying $p_\ell = i$ and $\ell \ge 2$, to
the set of pairs $$(\vb=[v^{(1)},\ldots,v^{(r)}], u)$$ such that (a)
$\vb \in \Sc_v$, (b) $u \subset v \setminus ( x \cup y)$, and (c)
$y_r = i$. Furthermore under $\Phi$ we must have (a) $b_\vb =
b_{\vb'}$, (b) $u.y = u'.y'i$ and (c) $u^\bot.y$ equal to the word
$(u')^\bot y'i$ with the letter $p_1$ removed.

Given $(\vb' = [v^{'(1)},\ldots,v^{'(k)}], u', p)$ we consider two
cases.  If $p_1 \notin y'$ we define $\vb =
[v^{'(1)},\ldots,v^{'(k)},p]$ and $u = u'$.  It is clear then that
$\vb \in \Sc_v$ and we have $x = x' \cup \{p_1\}$ and $y = y' \cup
\{i\}$.  Since $p_1 \in (u')^\bot$ we see that $u \subset v
\setminus ( x \cup y)$ and that $u^\bot$ is $(u')^\bot$ with $p_1$
removed.

Now suppose $p_1 = y_j$ for some $1 \leq j \leq k$.  We define the
{\it fusion} of two words $zi$ and $iz'$, where $z,z'$ are words and
$i$ is a letter to be $zi\star iz'=ziz'$.  Here, all words and
letters are considered subwords of $v$.  Let $\tv = v^{'(j)} \star
p$ be the fusion of $v^{'(j)}$ and $p$, which is defined since the
last letter of $v^{'(j)}$ and the first letter of $p$ are the same.
Define $\tilde \vb = [v^{'(1)},\ldots,\widehat{v^{'(j)}},\ldots,
v^{'(k)},\tv]$ where the hat denotes omission.  Now $\tilde \vb$
satisfies all the conditions of $\Sc_v$ except possibly condition
(1). We produce $\vb$ from $\tilde \vb$ by the following {\it
shuffling} procedure. Suppose $\tv = \tv_1 \cdots \tv_s$ and let $t
\in (1,s)$ be the maximal index (if it exists) such that $\tv_t \in
x'$, say $\tv_t = x'_m$ where $m \in (j,k)$. We now define $v^{(k)}
= \tv_t \tv_{t+1} \cdots \tv_s$ and replace $v^{'(m)}$ with $\tv =
(\tv_1 \cdots \tv_{t-1} \tv_t) \star v^{'(m)}$.  Now repeat the
procedure with the new $\tv$, searching for some $m' \in (j,m)$ such
that $\tv_{t'} = x'_{m'}$ for $t' \in (1,t-1)$.  When no more
shuffling occurs, we label the subwords $v^{(1)},\ldots,v^{(k)}$ in
order. Note that the $y'_r$ for $r \neq j$ are always kept in order.
By construction $\vb \in \Sc_v$ and we have $x = x'$ and $y = (y'
\setminus \{y_j\}) \cup \{i\}$.  We define $(u = u' \cup \{y_j\})
\subset v \setminus ( x \cup y)$ and check that $u.y = u'.y'i$ and
$u^\bot = (u')^\bot$. We now make the crucial observation: {\it
shuffling is invertible if the letter $y_j \in v^{(k)}$ is given} --
we will call this ``performing inverse shuffling at $y_j$''. This
completes the definition of $\Phi$.

We now define $\Phi^{-1}$.  Given $(\vb=[v^{(1)},\ldots,v^{(r)}],
u)$ we consider again two cases.  If $v^{(r)} \cap u = \emptyset$ we
proceed by defining $(\vb' = [v^{'(1)},\ldots,v^{'(r-1)}], u' = u, p
= v^{(r)})$.
Otherwise, suppose $v^{(r)} = v^{(r)}_1 \cdots v^{(r)}_s$ and let $t
\in (1,s)$ be the maximal index such that $v^{(r)}_t \in u$.  We
define $p = v^{(r)}_t v^{(r)}_{t+1} \cdots v^{(r)}_s$, $u' = u
\setminus \{v^{(r)}_t\}$ and to produce $\vb'$ we perform inverse
shuffling at $v^{(r)}_t$. It is straightforward to show that this
process well-defines a map that is inverse to $\Phi$.
\end{proof}

\subsection{Proof of Theorem \ref{T:phiP}}

Recall that by Theorem \ref{T:P},
\begin{equation*}
\PA_r = A_{\rho_r} + \text{non-Grassmannian terms}.
\end{equation*}
By Theorems \ref{T:iso} and \ref{T:P}, in order to prove
Theorem~\ref{T:phiP} it suffices to show that
\begin{equation*}
\pzt(\Delta(A_{\rho_r})) = 1 \otimes A_{\rho_r} + A_{\rho_r} \otimes
1
    + 2\sum_{1\le s <r} A_{\rho_s} \otimes A_{\rho_{r-s}}
    + \text{non-Grassmannian terms.}
\end{equation*}
We have used the fact that if $w$ is not Grassmannian then any term
$A_x \otimes A_y$ occurring in $\pzt(\Delta(A_w))$ has either $x$ or
$y$ non-Grassmannian.

We apply Proposition~\ref{prop:del A} to the case $w=\rho_r$ and $v$
the unique reduced word of $\rho_r$, which by \eqref{E:specialclass}
is given by
\begin{equation*}
v = \begin{cases} (r-1)\; (r-2) \cdots 1\;0 & \text{for $1\le r\le n$,}\\
\overline{2n+1-r} \; \overline{2n+2-r} \cdots \overline{n-1}\; n\;
n-1 \cdots 1\;0 & \text{for $n<r\le 2n$,}
\end{cases}
\end{equation*}
where, by convention, if a letter occurs twice in $\rho_r$ the left
occurrence is distinguished by a bar.

The terms $1\otimes A_{\rho_r} + A_{\rho_r} \otimes 1$ come from
$\vb=[\,]\in \Sc_v$ and $u=\emptyset$ and $u=v$ in~\eqref{eq:phi on A}.
All other $u\subset v$ yield non-Grassmannian terms.

Now we calculate the coefficient of the term $A_{\rho_s} \otimes
A_{\rho_{r-s}}$ for $s \geq 1$. Since the operation $\pzt \circ
\Delta$ is cocommutative, it suffices to consider the case $s\le
r-s$.

Define $R$ to be the set of letters occurring in
$\rho_r\rho_{r-s}^{-1}$ (together with the bars, if any).  If
$\overline{r-s-1} \in R$ (in particular $r-s-1<n$) define
$\overline{R}$ to be $R$ with $\overline{r-s-1}$ replaced by
$r-s-1$; otherwise set $\overline{R} = R$. If $\overline{s-1} \in R$
define $R_-$ to be $R$ with $\overline{s-1}$ removed.  If
$\overline{r-s-1} \in R_-$ define $\overline{R_-}$ to be $R_-$ with
$\overline{r-s-1}$ replaced by $r-s-1$; otherwise set
$\overline{R_-} = R_-$.

\begin{lem} \label{lem:uy}  Suppose $s \le r-s$.  The terms in
Proposition~\ref{prop:del A} which give $A_{\rho_s} \otimes
A_{\rho_{r-s}}$ are exactly the following tuples $\vb \in \Sc_v$ and
$u \subset v \setminus(x \cup y)$:
\begin{enumerate}
\item[Case 1:] $r\le n$ or $r>n$ and $s\le 2n+1-r$: \newline
$y=s-1 \; s-2 \cdots 1 \; 0$, $x = x_1x_2\dotsm x_s$ is a
permutation of the letters in $R$ or $\overline{R}$, and
$u=\emptyset$;
\item[Case 2:] $r>n$, $s< r-s$ and $s>2n+1-r$: \newline
In addition to the possibilities in Case 1 we may also have $y=s-2
\; \cdots 1\; 0$, $x = x_1x_2\dotsm x_{s-1}$ is a permutation of the
letters in $R_-$ or $\overline{R_-}$, and $u=\overline{s-1}$;
\item[Case 3:] $r>n$, $s=r-s$ and $s>2n+1-r$:\newline
We have either $y=s-1\; s-2\dotsm 1\;0$, $u=\emptyset$, and $x$ is a
permutation of $R$; or $y=\overline{s-1}\;s-2\dotsm 1\;0$,
$u=\emptyset$, and $x$ is a permutation of $\overline{R}$; or
$y=s-2\;\dotsm 1\;0$, $u=s-1$, $u^\perp=\overline{s-1}$, and $x$ is
a permutation of $R_-=\overline{R}_-$; or $y=s-2\;\dotsm 1\;0$,
$u=\overline{s-1}$, $u^\perp=s-1$, and $x$ is a permutation of
$R_-=\overline{R}_-$.
\end{enumerate}
\end{lem}
\begin{proof} $u.y=\rho_s$ is embedded as a subword of $\rho_r$
in two ways: $s-1\; s-2 \; \dotsm 1 \; 0$ or $\overline{s-1}\; s-2
\;\dotsm 1\; 0$. Similarly if $r-s>n$ then
$u^\perp.y=\rho_{r-s}\subset \rho_r$ and if $r-s\le n$ then
$u^\perp.y$ is either $r-s-1\;r-s-2\;\dotsm 1\;0$ or
$\overline{r-s-1}\;r-s-2\;\dotsm 1\;0$.

Suppose $0\le p\le s-2$ is such that $p\not\in y$. Then $p\in u$ and
$p\not\in u^\perp$, so that $p\not\in u^\perp.y$, a contradiction.
Therefore $y\supset s-2\;\dotsm 1\;0$. This gives four cases.
\begin{enumerate}
\item \label{I:ep} $u=\emptyset$ and $y=s-1\;s-2\dotsm 1\;0$.
\newline
Here $x$ can be a permutation of $R$ or also of $\overline{R}$
provided $\overline{r-s-1}\in R$.
\item \label{I:eb} $u=\emptyset$ and $y=\overline{s-1}\;s-2\dotsm1\;0$.
\newline
Suppose $r-s>n$. Since $u^\perp.y$ contains $\overline{s-1}$ we have
$2n+1-(r-s)<s$, giving the contradiction $2n+1<r$. Therefore $r-s\le
n$. Again since $u^\perp.y$ contains $\overline{s-1}$ we must have
$r-s=s$ and $u^\perp=\emptyset$. Then $x$ is a permutation of
$\overline{R}$.
\item \label{I:sp} $u=s-1$ and $y=s-2\dotsm1\;0$.\newline
$s-1\not\in u^\perp.y$. This can only occur if $r-s=s$,
$u^\perp.y=\overline{s-1}\;s-2\dotsm1\;0$, and
$u^\perp=\overline{s-1}$. Then $x$ is a permutation of
$R_-=\overline{R}_-$.
\item \label{I:sb} $u=\overline{s-1}$ and $y=s-2\dotsm1\;0$.\newline
If $r-s>n$ then $x$ is a permutation of $R_-$. Suppose $s<r-s\le n$.
If $u^\perp.y=r-s-1\;r-s-2\dotsm 1\;0$ then $x$ is a permutation of
$R_-$ and if $u^\perp.y=\overline{r-s-1}\;r-s-2\dotsm 1\;0$ then $x$
is a permutation of $\overline{R}_-$. If $s=r-s$ then $u^\perp=s-1$
and $x$ is a permutation of $R_-=\overline{R}_-$.
\end{enumerate}
In particular the last three cases only occur if
$\overline{s-1}\in\rho_r$, that is, $2n+1-r<s$. This given, the
Lemma follows.
\end{proof}

By Lemma \ref{lem:uy} the possibilities for the tuples $\vb \in
\Sc_v$ which contribute to the coefficient of $A_{\rho_s} \otimes
A_{\rho_{r-s}}$ in $\pzt(\Delta(A_{\rho_r}))$ are determined by
whether $u = \emptyset$, $u = s-1$ and $u = \overline{s-1}$.  We
denote the corresponding subsets of $\Sc_{\rho_r}$ by $\Sc_u$.  Note
that $\Sc_{s-1} = \Sc_{\overline{s-1}}$.  Define
$$
T'_u = \sum_{\vb \in \Sc_u} b_{\vb}.
$$

\begin{prop}\label{prop:hard}
Depending on the case of Lemma \ref{lem:uy}, we have
$T'_\emptyset=2$ or $T'_\emptyset + T'_{\overline{s-1}}=2$ or
$T'_\emptyset + T'_{s-1} + T'_{\overline{s-1}}=2$.
\end{prop}

Proposition~\ref{prop:hard} shows that the coefficient of $A_{\rho_s} \otimes
A_{\rho_{r-s}}$ in $\pzt(\Delta(A_{\rho_r}))$ is equal to $2$, thereby proving
Theorem~\ref{T:phiP}.

The proof of Proposition~\ref{prop:hard} is given in Section~\ref{SS:proof of hard},
after first preparing some technical preliminary results in
Sections~\ref{SS:notation}-\ref{SS:UUABC}.

\subsection{Notation}
\label{SS:notation}
For the evaluation of $T'_u$ we require slightly more general functions.

Let $\hat{x} I \hat{y}$ be an embedded subword of $\rho_r$ with
$\hat{x}$ and $\hat{y}$ subletters, and $I$ a subword.
Let $\Sc^{\hat{x}\hat{y}}_I$ be the set of all subwords $p$ of
$\rho_r$ with first letter $\hat{x}$ and last letter $\hat{y}$ such
that $p \cap I = \emptyset$. Then define
\begin{equation*}
	T_I^{\hat{x}\hat{y}}  := \sum_{p\in \Sc^{\hat{x}\hat{y}}_I} b_p.
\end{equation*}

Given sequences $x = x_1 x_2 \cdots x_k$ and $y = y_1 y_2 \cdots y_k$ of
subletters of $\rho_r$ we define
$$
T'(x,y) = \prod_{i = 1}^{k} T^{x_i,y_i}_{ \{x_1,\ldots,x_{i-1}\}
\cap (x_i,y_i)}
$$
where $(x_i,y_i)$ denotes the subword of $\rho_r$ occurring between
the letters $x_i$ and $y_i$.
Finally, let $T(x,y)$ be obtained from $T'(x,y)$ by ignoring the extra power of 2 (if any)
which arises when $y_k = 0$.

Given an interval $[t,q]$ of unbarred letters and a set $X$ of
barred letters, let $f(t,q,X)$ denote the sum of $T(x,y)$ as $x$
varies over all permutations of $X$ and $y = q (q-1) \cdots t$. We
always assume $|[t,q]| = |X|$, and write $k=|X|$. Given $(t,q,X)$,
we partition $X$ into subsets $A$, $B$, $C$ where
\begin{enumerate}
\item
$A \subset X$ consists of letters greater than $q$,
\item
$B \subset X$ is a subset of $[t,q]$,
\item
$C \subset X$ consists of letters less than $t$.
\end{enumerate}
Thus $X$ is the disjoint union of $A$, $B$, and $C$.  In the
following we will write $X - a$ to mean the set $X - \{a\}$ with the
element $a \in X$ removed.  For a set $S$ of (barred) integers let
$S^-$ denote $S$ with its maximum element removed.

Now let us suppose that we are given a set $X$ of barred and
unbarred letters, such that $X$ is the disjoint union of sets $A$,
$B$, $C$, $U$, and $U'$, satisfying:
\begin{enumerate}
\item $U'$ is a set of unbarred letters including $n$,
\item $A$ consists of some barred letters in $X$ greater than $q$,
\item $B$ consists of the barred letters in $X$ in the interval
$[t,q]$,
\item $C$ consists of the barred letters in $X$ less than $t$,
\item $U$ consists of some barred letters in $X$ greater than $q$,
\item every letter in $U$ or in $U'$ is greater than every letter in $A$, $B$,
and $C$,
\item the minimum element of $U$ is smaller than or equal to the minimum
element of $U'$.
\end{enumerate}

Denote by $g(t,q,X)$ the sum of $T(x,y)$ as $x$ varies over all
permutations of $X$ {\it such that all letters in $U$ occur to the
left of all letters in $U'$} and $y = q (q-1) \cdots t$.  Let us
call $X$ {\it balanced} if $\min(U') = \min(U)$.  Note that if $X$
is unbalanced it will stay unbalanced if $\min(U')$ is removed from
it.

Let us now suppose that $y_1 = \bar q$ (instead of $q$), but the
rest of $y$ is as before.  Denote the answers by $f'(t,q,X)$ and
$g'(t,q,X)$ and use all the same conventions as before.

Results about these various functions are proven in Sections~\ref{SS:single pair}-\ref{SS:fpgp}.
We are ultimately interested in the case that $X$ is one of $R$, $\bar R$,
$R_-$, or $\bar R_-$ and $y$ is one of the words described in Lemma~\ref{lem:uy}.
In Section~\ref{SS:UUABC} we explain how to construct the various subsets $A$,
$B$, $C$, $U$, and $U'$, before proving Proposition~\ref{prop:hard} in
Section~\ref{SS:proof of hard}.

\subsection{Results for $T_I^{\hat{x}\hat{y}}$}
\label{SS:single pair}
In this section we calculate $T_I^{\hat{x}\hat{y}}$, which gives the contribution from a single pair
of letters $\hat{x}$ and $\hat{y}$.

\begin{lem} \label{lem:sum p}
Let $\hat{x} I \hat{y}$ be an embedded subword of $\rho_r$ with
$\hat{x}$ and $\hat{y}$ subletters, $I$ a subword, and $\hat{y}<n$. Then
\begin{equation*}
T_I^{\hat{x}\hat{y}} = \sum_{p\in \Sc^{\hat{x}\hat{y}}_I} b_p = -
\chi(\hat{x}=\overline{\hat{y}})+ 2^{\chi(\hat{y}=0)}
\begin{cases}
    1 & \text{if $I=\emptyset$ or $\min(I)=\max(\hat{x},\hat{y})$}\\
    -1 & \text{if $I=\{n\}$ or $\min(I) = \{i, \bar i\}$ }\\
    0 & \text{else}
\end{cases}
\end{equation*}
where $\min(I)$ consists of the 0,1, or 2 smallest letters of $I$
and in the comparison $\min(I)=\max(\hat{x},\hat{y})$ we ignore
bars.
\end{lem}

\begin{proof}
Let $p=p_1 p_2\cdots p_\ell\in \Sc^{\hat{x}\hat{y}}_I$ be such that
$b_p\ne0$. Then $p_i-p_{i+1} \in\{\pm1,0\}$ for all $1\le i<\ell$.
Since $\overline{0}$ never occurs in $\rho_r$, $\hat{y} = 0$ and
$\hat{x} = \overline{\hat{y}}$ cannot both hold. If $\hat{y}=0$ then
since $b_{10}=2$, $b_p$ contains a factor of two, giving rise to the
overall factor $2^{\chi(\hat{y}=0)}$.

Since $p$ is a subword of the unique reduced word of $\rho_r$ and
the latter word is a $\La$ (see subsection \ref{SS:BruhatZ}), $p$
crosses at most once from barred to unbarred letters. Suppose
$\bar{i}j$ are consecutive letters in $p$. Then $j\in \{i-1,i,i+1\}$
and the corresponding values of $b_{ij}$ are $1,-2,1$ provided that
$j\ne n$.

The contributions of the paths that only differ by $\ldots
\overline{i} \;i-1 \ldots$, $\ldots \overline{i} \;i \; i-1\ldots$
and $\ldots \overline{i}\; i+1\; i\; i-1\ldots$ all cancel out. For
$\hat{y}>\hat{x}$, the contributions of the paths which differ in
$\ldots \overline{\hat{y}-1} \; \hat{y} \ldots$, $\ldots
\overline{\hat{y}-1}\; \overline{\hat{y}}\; \hat{y}\ldots$, and
$\ldots \overline{\hat{y}-1}\; \overline{\hat{y}} \hat{y}+1\;
\hat{y} \ldots$ all cancel out. For $\hat{x}=\overline{\hat{y}}$ the
paths $\overline{\hat{y}}\; \hat{y}$ and $\overline{\hat{y}} \;
\hat{y}+1\;\hat{y}$ leave a net contribution of $-1$.

The case $\bar{i}j=\overline{n-1}n$ is special since $b_{n-1n}=2$.
Given the above discussion, it is tedious but not difficult to check
all cases claimed in the lemma explicitly.
\end{proof}

In the following sections, we will use Lemma \ref{lem:sum p}
repeatedly without mention.

\subsection{Results for $f(t,q,X)$: Barred letters only}
In this section we derive results for $f(t,q,X)$, which is defined for a pair
of unbarred letters $t,q$ and a set $X$ of barred letters.
It follows from its definition that $f(t,q,X)$ only depends on $B$ and the sizes $|A|$ and
$|C|$. By definition $f(q,q,\emptyset) = 0$.

Let $\epsilon(n)$ denote the function which is 1 when $n$ is even
and $0$ when $n$ is odd.

\begin{lem}\label{L:0}
If $|B| = 0$ then $f(t,q,X) = 1$.
\end{lem}
\begin{proof}
By Lemma~\ref{lem:sum p} the only nonzero contributions to $f(t,q,X)$ occur when
$I=\{x_1,\ldots,x_{i-1}\}\cap(x_i,y_i)=\emptyset$ for all $i$. For this to hold for all
$i\in [1,q-t+1]$, the letters in $x=(x_1,x_2,\ldots,x_{q-t+1})$ must occur in the same
order as in $\rho_r$. This term gives contribution 1 to $f(t,q,X)$.
\end{proof}

\begin{lem}\label{L:1}
Suppose $X \neq \emptyset$.  If $t \notin B$ then $f(t,q,X) =
f(t+1,q,X^-)$.
\end{lem}

\begin{lem}\label{L:2}
Suppose $|C| > 0$.  Then $f(t,q,X) = f(t+1,q,X^-)$.
\end{lem}
\begin{proof}
By Lemma \ref{L:1}, this holds if $t \notin B$.  Suppose $t \in B$.
We have two cases $t = \max(X)$ or $t \neq \max(X)$.  In the first
case, a sequence $x$ can only contribute to $f(t,q,X)$ if $x_k =
\max(C)$.  But $f(t+1,q,X - \max(C)) = f(t+1,q,X^-) = 1$ by Lemma
\ref{L:0}.  In the second case $x_k$ may be $\max(X)$, $t$ or
$\max(C)$, all of which are distinct.  Hence we have
$f(t,q,X)=f(t+1,q,X^-)+f(t+1,q,X-\max(C))-f(t+1,q,X-t)$ by Lemma~\ref{lem:sum p}.
Since $f(t+1,q,X-t) = f(t+1,q,X-\max(C))$, we conclude that
$f(t,q,X) = f(t+1,q,X^-)$.
\end{proof}

\begin{lem}\label{L:3}
Suppose $|C| > 0$.  Then $f(t,q,X) = 1$.
\end{lem}
\begin{proof}
This follows from Lemmas~\ref{L:0} and~\ref{L:2} by induction.
\end{proof}

\begin{lem}\label{L:4}
Suppose $|C| = 0$.  Then $f(t,q,X) = \epsilon(|B|)$.
\end{lem}
\begin{proof}
We proceed by induction on the size of $X$.  Suppose $t \notin B$.
Then $|A|$ must be non-empty.  The inductive step holds by Lemma
\ref{L:1}. Suppose $t \in B$.  If $t = \max(X)$ then $|X| = 1$ and
the result is trivial. Thus we may assume that $t \neq \max(X)$.
Then $f(t,q,X) = f(t+1,q,X^-) - f(t+1,q,X-t)$.  By Lemma \ref{L:3},
$f(t+1,q,X^-) = 1$.  Also by the induction hypothesis $f(t+1,q,X-t)
= \epsilon(|B|-1)$.  The result follows.
\end{proof}

\subsection{Results for $g(t,q,X)$: Unbarred letters as well}
\label{sec:unbar}
In this section we prove results for the function $g(t,q,X)$, where
$X$ is a set of barred and unbarred letters.

\begin{lem}\label{L:5+}
Suppose $|A| = |B| = |C| = 0$.  Then $g(t,q,X) = 1$.
\end{lem}
\begin{proof}
Since by assumption all letters in $U$ occur to the left of $U'$, again by
Lemma~\ref{lem:sum p} the only contribution to $g(t,q,X)$ occurs when
all letters in $x$ occur in the same order as in $\rho_r$.
\end{proof}

\begin{lem}\label{L:5}
Suppose $X$ is unbalanced and $|C| > 0$.  Then $g(t,q,X) = 1$.
\end{lem}
\begin{proof}
We proceed by induction on $|X|$.  If $t \notin B$ then for a
non-zero contribution we must have $x_k = \min(U')$.  If $|U'| = 1$
the result follows from Lemma \ref{L:3}.  Otherwise it is the
inductive hypothesis.  If $t \in B$ then $x_k$ may be equal to
$\max(C)$, $t$, or $\min(U')$, which are distinct.  Hence we have
$g(t,q,X)=g(t+1,q,X-\min(U'))+g(t+1,q,X-\max(C))-g(t+1,q,X-t)$
by Lemma~\ref{lem:sum p}. Since $g(t+1,q,X-t) = g(t+1,q,X-\max(C))$, we
conclude that $g(t,q,X) = g(t+1,q,X-\min(U'))$.  The argument now proceeds
as for $t \notin B$.
\end{proof}

\begin{lem}\label{L:6}
Suppose $X$ is unbalanced and $|C| = 0$. Then $g(t,q,X) =
\epsilon(|B|)$.
\end{lem}
\begin{proof}
We proceed by induction on $|X|$.  If $t \notin B$ the argument is
as for Lemma \ref{L:5}.  If $t \in B$, we may have $x_k \in \{t,
\min(U')\}$.  Thus $g(t,q,X) = g(t+1,q,X-\min(U'))-g(t+1,q,X-t)$. By
induction and Lemma \ref{L:5} this is equal to $1 - \epsilon(|B|
-1)$, as required.
\end{proof}

\begin{lem} \label{L:7}
Suppose $X$ is balanced and $|C| > 0$.  Then $g(t,q,X) =
\epsilon(|A|+|B|+|C|)$.
\end{lem}
\begin{proof}
We proceed by induction on $|X|$.  Let $a = \max(A \cup B \cup C)$.
If $t \notin B$ then $g(t,q,X) = g(t+1,q,X-\min(U'))-g(t+1,q,X-a)$.
By Lemma \ref{L:5}, $g(t+1,q,X-\min(U')) = 1$.  If $a \in C$ and
$|C| = 1$ then the result follows from Lemma \ref{L:5+}, otherwise
it follows by induction. If $t \in B$ then $g(t,q,X) =$
$$
g(t+1,q,X-\min(U')) - g(t+1,q,X-a) - g(t+1,q,X-t) +
g(t+1,q,X-\max(C)).
$$
Note that this formula holds even if $t = a$.  Clearly, the last two
terms cancel, so the result follows by induction and Lemma
\ref{L:5}.
\end{proof}

Let $\theta(n)$ be the function with values $1,-1,2,-2,3,-3,\ldots$
on the nonnegative integers.

\begin{lem}\label{L:8}
Suppose $X$ is balanced, $|C| = 0 = |A|$, and $t\in B$. Then
$g(t,q,X) = \theta(|B|)$.
\end{lem}
\begin{proof}
We proceed by induction on $|X|$. 
We have
$$
g(t,q,X) = g(t+1,q,X-\min(U')) -  g(t+1,q,X-t) - g(t+1,q,X-b)
$$
where $b = \max(B)$.  Again this holds even if $t = b$.  By Lemma
\ref{L:5}, $g(t+1,q,X-\min(U')) = 1$.  The result follows from
induction if $t = b$ since $g(t+1,q,X-t) = \theta(0) = 1$.  Assume
$t \neq b$.  By Lemma \ref{L:7}, $g(t+1,q,X-b) = \epsilon(|B|-1)$.
By the inductive hypothesis, $g(t+1,q,X-t) = \theta(|B|-1)$.  But
$\theta(|B|) = 1 - \theta(|B|-1) - \epsilon(|B|-1)$, proving the
lemma.
\end{proof}

%

\begin{lem}\label{L:10}
Suppose $X$ is balanced and $|C| = 0$.  Furthermore suppose that
either we have $|B| = 0$ or we have $B = [t',t'']$ for some $t' \geq
t$ and $|A| \leq t' - t$. Then $g(t,q,X) = (-1)^{|A|}\theta(|B|) +
\epsilon(|B|)\epsilon(|A|-1)$.
\end{lem}
\begin{proof}
First suppose $|A| = 0$.  If $|B|=0$ then $g(t,q,X) =
g(t+1,q,X-\min(U'))$ and we are done by Lemma \ref{L:6}.  Otherwise
let $t'=\min(B)$. If $t = t'$ we are done by Lemma \ref{L:8}. So let
$t'>t$. We proceed by induction on $t' - t \geq 1$. Thus $t \notin
B$ and $g(t,q,X) = g(t+1,q,X-\min(U')) - g(t+1,q,X-b)$ where $b =
\max(B)$. By inductive hypothesis and Lemma \ref{L:6}, $g(t,q,X) =
\epsilon(|B|) - \theta(|B|-1) = \theta(|B|)$, as required.

Now suppose that $|A| > 0$.  Then $t \notin B$ and by Lemma
\ref{L:6} and induction we have
\begin{align*} g(t,q,X) &= g(t+1,q,X-\min(U')) -
g(t+1,q,X-\max(A))\\
&=\epsilon(|B|) - \left((-1)^{|A|-1}\theta(|B|) +
\epsilon(|B|)\epsilon(|A|-2)\right) \\
&= (-1)^{|A|}\theta(|B|)+\epsilon(|B|)\epsilon(|A|-1) \end{align*}
as required.
\end{proof}


For the next lemma we assume that $\min(U)>\min(U')$.
Suppose that $\min(U) \in U'$ (though one is barred and the other
unbarred). Define $U'_- = \{u \in U' \mid u < \min(U)\}$.
\begin{lem} \label{L:11}
If $|C| = 0$, $B = [t',t'']$ and $|U'_-| + |A| \leq t' - t$, then
Lemma \ref{L:10} holds as stated.
\end{lem}
\begin{proof}
The statements follow from the fact that a non-zero contribution
only occurs with $x_k = \min(U'_-)$.  When $|U'_-|$ becomes zero,
$X$ is balanced and we are in the situation of Lemma \ref{L:10}.
\end{proof}

\subsection{Results for $f'(t,q,X)$ and $g'(t,q,X)$: One barred letter in $y$}
\label{SS:fpgp}
Let us now suppose that $y_1 = \bar q$ (instead of $q$), but the
rest of $y$ is as before.  Denote the answers by $f'(t,q,X)$ and
$g'(t,q,X)$ and use all the same conventions as before.  We state
the relevant results. The proofs are identical to before.  Let
$\epsilon'(n) = 1- \epsilon(n)$ and $\theta'(n)$ be defined on
nonnegative integers by the sequence $0,1,-1,2,-2,\ldots$.

\begin{lem}\label{L:3a}
Suppose $|C| > 0$.  Then $f'(t,q,X) = 1$.
\end{lem}

\begin{lem}\label{L:4a}
Suppose $|C| = 0$.  Then $f'(t,q,X) = \epsilon'(|B|)$.
\end{lem}

\begin{lem}\label{L:5a}
Suppose $X$ is unbalanced and $|C| > 0$.  Then $g'(t,q,X) = 1$.
\end{lem}

\begin{lem}\label{L:6a}
Suppose $X$ is unbalanced and $|C| = 0$. Then $g'(t,q,X) =
\epsilon'(|B|)$.
\end{lem}

\begin{lem} \label{L:7a}
Suppose $X$ is balanced and $|C| > 0$.  Then $g'(t,q,X) =
\epsilon'(|A|+|B|+|C|)$.
\end{lem}


\begin{lem}\label{L:9a}
Suppose $X$ is balanced and $|C| = 0 = |A|$. 
Then $g'(t,q,X) = \theta'(|B|)$.
\end{lem}

\subsection{Defining $U$, $U'$, $A$, $B$, $C$}
\label{SS:UUABC}

Let us set $X$ to be one of $R$, $\bar R$, $R_-$, or $\bar R_-$ and
pick $y$ to be one of the words described in Lemma \ref{lem:uy}.
Thus $t = 0$ and $q \in \{s-1, s-2\}$.  We describe how to construct
$U$, $U'$, $A$, $B$, and $C$.  First we must have $C = \emptyset$
and there is no choice for $B$ and $U'$.  We let $A$ be the set of
barred letters in $X$ which are greater than $q$ and less than all
unbarred letters in $X$.  We let $U$ be any remaining barred letters
in $X$.

For example, in Case 1 of \ref{lem:uy} when $X = R$, $r > n$, and $2n+1-r\le r-s$, we
would have $B=\emptyset$ and
\begin{align*}
U' &= \{r-s,r-s+1,\ldots,n\} \\
U &= \{\overline{r-s}, \overline{r-s+1},\ldots, \overline{n-1}\} \\
A &= \{\overline{r-s-1}, \overline{r-s-2},\ldots,
\overline{2n+1-r}\}. \end{align*}

We claim that the sum $g(t,q,X)$ of subsection \ref{sec:unbar} is
equal to the same sum, but without the assumption that letters in
$U$ occur to the left of letters in $U'$.

This is proved as follows.  In the case that $X \in \{R, \bar R,
R_-,\bar R_-\}$, the letters in $U$ are all present in $U'$.  Let us
take a permutation $x$ of $X$ such that $T(x,y) \neq 0$.  Then it is
clear that unbarred letters of $X$ occur in $x$ in the same order as
in $\rho_r$.  Suppose $x_i = n$ and there is a barred letter $\bar j
\in U$ occurring to the right of $x_i$, say $x_p = \bar j$ where $p
> i$.  Then one checks that in $x$, (a) barred letters greater than
$\overline{j}$ occur before $x_i$, (b) letters in
$\{n-1,\ldots,j+1\}$ occur between $x_i$ and $x_p$ and (c) no
unbarred letter less than $j$ occurs between $x_i$ and $x_p$. We now
define a sign-reversing involution: let $\tilde{x}$ be obtained from
$x$ by swapping the locations of $j$ with $\overline{j}$.  By Lemma
\ref{lem:sum p}, $T(x,y) = -T(\tilde{x},y)$.  In other words, to
calculate the sum of $T(x,y)$ as $x$ varies over all permutations of
$X$ we need only consider permutations $x$ such that letters in $U$
occur to the left of letters in $U'$.

\subsection{Proof of Proposition \ref{prop:hard}}
\label{SS:proof of hard}

\subsubsection{Notation} Since $t=0$  and $q\in\{s-1,s-2\}$ are known for the
cases of Lemma~\ref{lem:uy}, from now on we will
denote the total contributions by $f(X)$, $g(X)$, $f'(X)$, $g'(X)$, where
$X \in \{R, \overline{R}, R_-,\overline{R_-}\}$ is one of the sets in Lemma
\ref{lem:uy}.  We prove Proposition \ref{prop:hard} by splitting
into the cases of Lemma \ref{lem:uy}.  Note that the terms $T'_u$ of
Proposition \ref{prop:hard} always differ from the corresponding
numbers $f(X)$ and $g(X)$ by a factor of 2 (arising from $y_k = 0$).

\subsubsection{Case 1} If $R = \overline{R}$ then $T'_\emptyset = 2f(R)$.
The result follows from Lemma \ref{L:11}.

Suppose $R \neq \overline{R}$.  By Lemma \ref{L:11} we have $g(R) =
(-1)^{a} + \epsilon(a-1)$, where $a = |A|$ for $X = R$. Similarly,
we have $g(\overline{R}) = (-1)^{a-1} + \epsilon(a)$. Thus $g(R) + g(\overline{R}) =
1$, so $T'_\emptyset = 2$.

\subsubsection{Case 2} If $R$ consists only of barred letters, then
we have $f(R) + f(R_-) = 1$ by Lemma \ref{L:4}.

Otherwise we need to calculate $g(R)+g(R_-)+g(\overline{R})+g(\overline{R_-})$.
Let $a$ and $b$ be the sizes of $|A|$ and $|B|$ when $X = R$.  Then
by Lemma \ref{L:11}
\begin{align*}
g(R) &= (-1)^{a}\theta(b) + \epsilon(b)\epsilon(a-1) \\
g(\overline{R}) &= (-1)^{a-1}\theta(b) + \epsilon(b)\epsilon(a) \\
g(R_-) &= (-1)^{a}\theta(b-1) + \epsilon(b-1)\epsilon(a-1) \\
g(\overline{R_-}) &= (-1)^{a-1}\theta(b-1) + \epsilon(b-1)\epsilon(a)
\end{align*}
and the sum is 1.

In both cases $T'_\emptyset + T'_{\overline{s-1}} = 2$ as
required.

\subsubsection{Case 3} Let $b = 3s-2n-1$ be the size of $B$ in $R$.
Then $g(R) = \theta(b)$ and $g(R_-) = \theta(b-1)$ by Lemma~\ref{L:11}.
For the case $y = \overline{s-1} s-2 \cdots 0$ we have a
contribution of $g'(\overline{R})=\theta'(b-1)$ by Lemma \ref{L:9a}. We calculate
$T'_\emptyset + T'_{\overline{s-1}} + T'_{s-1} = 2(\theta(b) +
\theta'(b-1)) +2\theta(b-1) + 2\theta(b-1) = 2$.

\appendix

\section{$\PA_i$}
\label{S:PAdata}

In the data that follows, elements of $\tC_n$ are indicated by
reduced words.

Some $\PA_i$ are expressed in the $A_w$ basis of $\NilCox$.
\subsection{$n=2$}
\begin{align*}
\PA_1 &= A_{0}+A_{1}+A_{2}  \\
\PA_2 &= A_{01}+
    A_{10}+A_{12}+2 A_{20}+A_{21}  \\
\PA_3 &= A_{012}+
    A_{101}+A_{120}+A_{121}+A_{201}+
    A_{210}  \\
\PA_4 &= A_{0121}+A_{1012}+
    A_{1210}+A_{2101}
\end{align*}
\subsection{$n=3$}
\begin{align*}
\PA_1 &= A_{0}+A_{1}+
    A_{2}+A_{3}  \\
\PA_2 &= A_{01}+
    A_{10}+A_{12}+2 A_{20}+A_{21}+A_{23}+2 A_{30}+2 A_{31}+
    A_{32}  \\
\PA_3 &= A_{012}+A_{101}+A_{120}+A_{121}+A_{123}+A_{201}+
    A_{210}+2 A_{230}\\ &+A_{231}+A_{232}+2 A_{301}+2 A_{310}+
    A_{312}+2 A_{320}+A_{321}  \\
\PA_4 &= A_{0121}+A_{0123}+
    A_{1012}+A_{1210}+A_{1230}+A_{1231} \\ &+A_{1232}+
    A_{2101}+A_{2301}+A_{2310}+2 A_{2320}+A_{2321}\\ &+
    A_{3012}+2 A_{3101}+A_{3120}+A_{3121}+A_{3201}+
    A_{3210}  \\
\PA_5 &= A_{01231}+A_{01232}+A_{10123}+A_{12310}+
    A_{12320} + A_{12321}+
    A_{21012}\\ &+A_{23101}+A_{23201}+A_{23210}+
    A_{30121}+A_{31012}+A_{31210}+
    A_{32101}  \\
\PA_6 &= A_{012321}+A_{101232}+
    A_{123210}+A_{210123}+A_{232101}+
    A_{321012}
\end{align*}
\subsection{$n=4$}
\begin{align*}
\PA_1 &= A_{0}+A_{1}+
    A_{2}+A_{3}+A_{4}  \\
\PA_2 &= A_{01}+
    A_{10}+A_{12}+2 A_{20}+A_{21}+A_{23}+2 A_{30}+2 A_{31}+
    A_{32}+A_{34}\\
    &+2 A_{40}+2 A_{41}+2 A_{42}+A_{43}  \\
\PA_3 &= A_{012}+
    A_{101}+A_{120}+A_{121}+A_{123}+A_{201}+
    A_{210}+2 A_{230}+A_{231}\\
    &+A_{232}+A_{234}+2 A_{301}+2 A_{310}+A_{312}+2 A_{320}+
    A_{321}+2 A_{340}\\&+2 A_{341}+A_{342}+
    A_{343}+2 A_{401}+2 A_{410}+2 A_{412}+4 A_{420}+2 A_{421}\\&+A_{423}+2 A_{430}+2 A_{431}+
    A_{432}  \\
\PA_4 &= A_{0121}+A_{0123}+
    A_{1012}+A_{1210}+A_{1230}+A_{1231}+A_{1232}+
    A_{1234}\\&+A_{2101}+A_{2301}+A_{2310}+2 A_{2320}+
    A_{2321}+2 A_{2340}+A_{2341}+A_{2342}\\&+A_{2343}+
    A_{3012}+2 A_{3101}+A_{3120}+A_{3121}+A_{3201}+
    A_{3210}+2 A_{3401}\\&+2 A_{3410}+A_{3412}+2 A_{3420}+
    A_{3421}+2 A_{3430}+2 A_{3431}+
    A_{3432}+2 A_{4012}\\&+2 A_{4101}+2 A_{4120}+2 A_{4121}+
    A_{4123}+2 A_{4201}+2 A_{4210}+2 A_{4230}+A_{4231}\\&+
    A_{4232}+2 A_{4301}+2 A_{4310}+A_{4312}+2 A_{4320}+
    A_{4321}
\end{align*}

\section{$\Qf_w$}
\label{S:Qfdata} In the following tables, for $w\in\tC_n^0$ and
$\la$ a partition, the $(w,\la)$-th entry is the coefficient of
$M_\la=2^{\ell(\la)}m_\la$ in $\Qf_w$. We work in the quotient in
\eqref{E:Ccoiso} and hence we expand in $M_\la$ for $\la_1 \le 2n$.

\setlength{\extrarowheight}{1pt}
\subsection{$n=2$}
\newcolumntype{X}{>{\!}c<{\!}}
\begin{align*}
&\begin{array}{|X||X|} \hline  %
& 1 \\ \hline \hline %
0 & 1 \\ \hline
\end{array}\quad
\begin{array}{|X||X|X|} \hline %
 & 2 & 11 \\ \hline \hline
10 & 1 & 1 \\ \hline
\end{array}\quad
\begin{array}{|X||X|X|X|} \hline %
 & 3 & 21 & 111 \\ \hline \hline
010 & & 1 & 1 \\ \hline
210 & 1 & 1 & 1 \\ \hline
\end{array}\quad
\begin{array}{|X||X|X|X|X|X|} \hline
 & 4 & 31 & 22 & 211 & 1^4 \\ \hline \hline
0210& & 1 & 2 & 2 & 2  \\ \hline
1210&1&1&1&1&1 \\ \hline
\end{array} \\[2mm]
&\begin{array}{|X||X|X|X|X|X|X|} \hline %
 & 41 & 32 & 311 & 221 & 21^3 & 1^5 \\ \hline \hline
10210&  &1&1&2&2&2 \\ \hline
01210& 1 & 1 & 1 & 1 & 1 & 1 \\ \hline
\end{array} \\[2mm]
&\begin{array}{|X||X|X|X|X|X|X|X|X|X|} \hline %
&42&411&33&321&31^3&222&2211&21^4&1^6 \\ \hline %
010210&&&&1&1&2&2&2&2 \\ \hline %
210210& &&1&1&1&2&2&2&2 \\ \hline%
101210&1&1&1&1&1&1&1&1&1 \\ \hline%
\end{array} \\[2mm]
&\begin{array}{|X||X|X|X|X|X|X|X|X|X|X|X|} \hline %
& 43&421&41^3&331&322&3211&31^4&2^31&221^3&21^5&1^7 \\ \hline\hline %
0210210&&&&1&2&2&2&4&4&4&4\\ \hline %
0101210&&1&1&1&2&2&2&3&3&3&3\\ \hline %
2101210&1&1&1&1&1&1&1&1&1&1&1 \\ \hline %
\end{array}
\end{align*}

\subsection{$n=3$}

\begin{align*}
&\begin{array}{|X||X|} \hline  %
& 1 \\ \hline \hline %
0 & 1 \\ \hline
\end{array}\quad
\begin{array}{|X||X|X|} \hline %
 & 2 & 11 \\ \hline \hline
10 & 1 & 1 \\ \hline
\end{array}\quad
\begin{array}{|X||X|X|X|} \hline %
 & 3 & 21 & 111 \\ \hline \hline
010 & & 1 & 1 \\ \hline 210 & 1 & 1 & 1 \\ \hline
\end{array}\quad
\begin{array}{|X||X|X|X|X|X|} \hline
 & 4 & 31 & 22 & 211 & 1^4 \\ \hline \hline
0210& & 1 & 2 & 2 & 2  \\ \hline %
3210&1&1&1&1&1 \\ \hline
\end{array} \\[2mm]
&\begin{array}{|X||X|X|X|X|X|X|X|} \hline %
 & 5 & 41 & 32 & 311 & 221 & 21^3 & 1^5 \\ \hline \hline
10210& & &1&1&2&2&2 \\ \hline %
03210& & 1&2&2&3&3&3 \\ \hline %
23210& 1 & 1 & 1 & 1 & 1 & 1 &1 \\ \hline
\end{array} \\[2mm]
&\begin{array}{|X||X|X|X|X|X|X|X|X|X|X|X|} \hline  %
& 6&51&42&411&33&321&31^3&222&2211&21^4&1^6 \\ \hline \hline %
010210& & & & & & 1&1&2&2&2&2 \\ \hline %
103210& & & 1&1&2&3&3&5&5&5&5 \\ \hline %
023210& & 1&2&2&2&3&3&4&4&4&4 \\ \hline %
123210& 1&1&1&1&1&1&1&1&1&1&1 \\ \hline
\end{array} \\[2mm]
&\begin{array}{|X||X|X|X|X|X|X|X|X|X|X|X|X|X|X|} \hline %
&61&52&511&43&421&41^3&331&322&3211&31^4&2^31&221^3&21^5&1^7\\
\hline\hline %
0103210&&&&&1&1&2&4&4&4&7&7&7&7 \\ \hline %
2103210&&&&1&1&1&2&3&3&3&5&5&5&5 \\ \hline %
1023210&&1&1&1&2&2&2&3&3&3&4&4&4&4\\ \hline %
0123210&1&1&1&1&1&1&1&1&1&1&1&1&1&1 \\ \hline
\end{array}
\end{align*}

\section{$\Pf_w$} \label{S:Pfdata} In the following tables, for
$w\in \tC_n^0$ and $\la$ a strict partition, the $(w,\la)$-th entry
is the coefficient of the Schur $P$-function $P_\la$ in $\Pf_w$.
Again $w$ is given as a reduced word.
\subsection{$n=2$}
\begin{align*}
&\begin{array}{|X||X|} \hline %
& 1 \\ \hline\hline %
0 & 1 \\ \hline
\end{array}
\qquad
\begin{array}{|X||X|} \hline %
 &  2 \\ \hline\hline %
10 & 1 \\ \hline
\end{array}
\qquad
\begin{array}{|X||X|X|} \hline %
 & 3 & 21 \\ \hline\hline
010 & & 1 \\ \hline %
210 & 1 & \\ \hline %
\end{array}
\qquad
\begin{array}{|X||X|X|} \hline %
& 4 & 31 \\ \hline\hline %
0210 & & 1  \\ \hline %
1210 & 1 & \\ \hline %
\end{array} \qquad
\begin{array}{|X||X|X|X|} \hline %
& 5&41&32 \\ \hline\hline %
10210 & &1&1 \\ \hline %
01210 & 1&1& \\ \hline %
\end{array} \\[2mm]
&\begin{array}{|X||X|X|X|X|} \hline %
&6&51&42&321 \\ \hline \hline %
010210 & &1&1&1 \\ \hline  %
210210 & &&1& \\ \hline%
101210 & 1&2&1&  \\  \hline %
\end{array}
\qquad %
\begin{array}{|X||X|X|X|X|X|} \hline %
&7&61&52&43&421\\ \hline \hline %
0210210 & &&1&1&1 \\ \hline%
0101210 & &1&1&&1  \\  \hline %
2101210 & 1&2&2&1&  \\ \hline %
\end{array}
\end{align*}

\subsection{$n=3$}
\begin{align*}
&\begin{array}{|X||X|} \hline \vn & 1 \\ \hline
\end{array}
\qquad
\begin{array}{|X||X|} \hline %
& 1 \\ \hline \hline %
0 & 1 \\ \hline
\end{array}
\qquad
\begin{array}{|X||X|} \hline
& 2 \\ \hline\hline %
10 & 1 \\ \hline
\end{array}
\qquad
\begin{array}{|X||X|X|} \hline %
&3&21 \\ \hline\hline %
010 & & 1 \\ \hline
210 & 1& \\ \hline %
\end{array}
\qquad
\begin{array}{|X||X|X|} \hline %
& 4 & 31 \\ \hline\hline %
0210 & & 1 \\ \hline %
3210 & 1 &  \\ \hline %
\end{array} \\[2mm]
&\begin{array}{|X||X|X|X|} \hline %
& 5&41&32 \\ \hline\hline %
10210 & & &1 \\ \hline %
03210 & &1&  \\ \hline %
23210 & 1&& \\ \hline %
\end{array}
\qquad %
\begin{array}{|X||X|X|X|X|} \hline %
&6&51&42&321\\ \hline\hline %
010210 & &&&1 \\ \hline %
103210 & &&1& \\ \hline %
023210 & &1&& \\ \hline %
123210 & 1&&& \\ \hline %
\end{array}
\qquad \!%
\begin{array}{|X||X|X|X|X|X|} \hline %
&7&61&52&43&421\\ \hline\hline %
0103210 & &&&&1 \\ \hline %
2103210 & &&1&1& \\ \hline%
1023210 & &1&1&& \\ \hline %
0123210 & 1&1&&& \\ \hline %
\end{array} \\
\end{align*}

\end{document}